\documentclass[a4paper,10pt]{article}
\usepackage{amsmath,amssymb,amsthm}
\usepackage[utf8]{inputenc}
 \usepackage[T1]{fontenc}
\usepackage{graphicx}
\usepackage{mathrsfs}
 \usepackage{dsfont}
\usepackage{natbib}
\usepackage{color}

\usepackage{url,xspace}
\definecolor{couleurCitations}{rgb}{0,0.65,0}
\definecolor{couleurRef}{rgb}{0.75,0,0}
\usepackage[pdftex,colorlinks=true,
linkcolor=couleurRef,citecolor=blue,a4paper=true
,bookmarks=true,plainpages=true,
 urlcolor= blue]{hyperref}

\newtheorem{Th}{Theorem}
\newtheorem{Remark}{Remark}

\newtheorem{Prop}{Proposition}
\newtheorem{lemma}{Lemma}

\newtheorem{definition}{Definition}

\def\1{\mathds{1}}
\def\var{\mathrm{Var}}

\renewenvironment{proof}{\noindent{\bf Proof.}}{\hfill $\Box$\par\noindent}

\newcommand{\el}{l}
\newcommand{\vv}{v}
\newcommand{\Sj}{S_j}
\newcommand{\Seta}{S_{\eta}}
\newcommand{\E}{\ensuremath{\mathbb{E}}}
\renewcommand{\P}{\ensuremath{\mathbb{P}}}
\newcommand{\R}{\ensuremath{\mathbb{R}}}
\newcommand{\Z}{\ensuremath{\mathbb{Z}}}

\newcommand{\N}{\ensuremath{\mathbb{N}}}

\newcommand{\e}{\ensuremath{\varepsilon}}

\newcommand{\eps}{\varepsilon}

\newcommand{\parinc}[2]{\parbox[c]{#1}{\includegraphics[width=#1]{#2}}}

\begin{document}
\title{Adaptive wavelet multivariate regression with errors in variables}
\author{Micha\"el Chichignoud \thanks{ETH Z\"urich}, Van Ha Hoang \thanks{Laboratoire Paul Painlev\'e UMR CNRS 8524, Universit\'e Lille 1- Sciences et Technologies.}, Thanh Mai Pham Ngoc \thanks{Laboratoire  de Math\'ematiques, UMR 8628, Universit\'e Paris Sud.} and Vincent Rivoirard\thanks{CEREMADE, UMR CNRS 7534, Universit\'e Paris Dauphine.}}
\maketitle

\begin{abstract}
In the multidimensional setting, we consider the errors-in-variables model. We aim at estimating the unknown nonparametric multivariate regression function with errors in the covariates. We devise an adaptive estimator based on projection kernels on wavelets and a deconvolution operator. We propose an automatic and fully data driven procedure to select the wavelet level resolution. We obtain an oracle inequality and optimal rates of convergence over anisotropic H\"older classes. Our theoretical results are illustrated by some simulations. 

\end{abstract}

   \\%
\textbf{Primary subjects}.  62G08

 \section{Introduction}

 We consider the problem of multivariate nonparametric regression with errors in variables.  We observe the i.i.d dataset
 $$(W_1,Y_1),\dots,(W_n,Y_n)$$
where
 $$ Y_l=m(X_\el) + \varepsilon_\el $$ and
 $$W_\el=X_\el+\delta_\el, $$
 with $Y_\el\in \R$. 
 The covariates errors $\delta_\el$ are i.i.d unobservable random variables having error density $g$. We assume that $g$ is known. The $\delta_\el$'s  are independent of the $X_\el$'s and $Y_\el$'s. The $\varepsilon_\el$'s are i.i.d standard normal random variables with variance $s^2$. We wish to estimate the regression function $m(x), x \in [0,1]^d$, but direct observations of the covariates $X_\el$ are not available. Instead due to the measuring mechanism or the nature of the environment,
the covariates $X_\el$ are measured with errors. Let us denote $f_X$ the density of the $X_\el$'s assumed to be positive and $f_W$ the density of the $W_\el$'s.  

Use of errors-in-variables models appears in many areas of science such as medicine, econometry or astrostatistics  and is appropriate in a lot of  practical experimental problems. For instance, in epidemiologic studies where risk factors are partially observed  (see \cite{Whittemore}, \cite{Fan92})  or in environmental science where air quality is measured with errors (\cite{DelaigleHallJamshidi}).

In the error-free case, that is $\delta_l=0$, one retrieves the classical multivariate nonparametric regression problem. Estimating a function in a nonparametric way from data measured with error is not an easy problem. Indeed, constructing a consistent estimator in this context is challenging as we have to face to a deconvolution step in the estimation procedure. Deconvolution problems arise in many fields where data are obtained with measurement errors and has attracted a lot of attention  in the statistical literature, see \cite{Meister} for an excellent source of references. The nonparametric regression with errors-in-variables model has been the object of a lot of attention as well, we may cite the works of  \cite{Fan92}, \cite {Fan93}, \cite{Ioannides}, \cite{Koo98}, \cite{Meister}, \cite{ComteTaupin}, \cite{Chesneau}, \cite{DuZouWang}, \cite{CarrollDelaigleHall}, \cite{DelaigleHallJamshidi}. The literature has mainly to do with kernel-based approaches, based on the Fourier transform. All the works cited have tackled the univariate case except for \cite {Fan92} where the authors explored the asymptotic normality for mixing processes. In the one dimensional setting, 
\cite{Chesneau} used Meyer wavelets in order to devise his statistical procedure but his assumptions on the model are strong since the corrupted observations $W_\el$ follow a uniform density on $[0,1]$. \cite{ComteTaupin} investigated the mean integrated squared error with a penalized estimator based on projection methods upon Shannon basis. But the authors do not give any clue about how to choose the resolution level of the Shannon basis. Furthermore, the constants in the penalized term are calibrated via intense simulations.

In the present article, our aim is to study  the multidimensional setting and the pointwise risk. We would like to take into account the anisotropy for the function to estimate. Our approach relies on the use of projection kernels on wavelets bases combined with a deconvolution operator taking into account the noise in the covariates. When using wavelets, a crucial point lies in the choice of the resolution level. But it is well-known that theoretical results in adaptive estimation do not provide the way to choose  the numerical constants in the resolution level and very often lead to conservative choices. We may cite the work of \cite{Nickl} which attempts to tackle this problem. For the density estimation problem and the sup-norm loss, the authors based their statistical procedure on Haar projection kernels and provide a way to choose locally the resolution level. Nonetheless, in practice, their procedure relies on  heavy Monte Carlo simulations to calibrate the constants. In our paper the resolution level of our estimator is optimal and fully data-driven. It is automatically selected by a method inspired from \cite{GL} to tackle anisotropy problems. This method has been used recently in various contexts (see  \cite{DHRR},  \cite{ComteLacour} and \cite{BLR}). Furthermore, we do not resort to thresholding which is very popular when using wavelets and our selection rule is adaptive to the unknown regularity of the regression function. 
We obtain oracle inequalities and provide optimal rates of convergence for anisotropic H\"older classes. The performances of our adaptive estimator, the negative impact of the errors in the covariates, the effects of the design density are assessed by examples based on simulations. 

The paper is organized as follows. In Section \ref{secestimation}, we describe our estimation procedure. In Section \ref{secrates}, we provide an oracle inequality and rates of convergences of our estimator for the pointwise risk. Section \ref{secsimu} gives some numerical illustrations. Proofs of Theorems, propositions and technical lemmas are to be found in section \ref{secpreuve}. 

\paragraph{Notation}
 Let $\N=\{0,1,2,\dots\}$ and $j=(j_1,\dots,j_d)\in\N^d$, we set $S_j =\sum_{i=1}^d j_i$
and for any $y\in\R^d$, we set, with a slight abuse of notation,
 $$2^j y:=(2^{j_1}y_1,\dots,2^{j_d}y_d)$$
and for any $k=(k_1,\cdots,k_d)\in\Z^d$,   $$h_{j,k}(y):=2^{\frac \Sj 2}h(2^jy-k)= 2^{\frac \Sj 2}h(2^{j_1}y_1-k_1,\dots,2^{j_d}y_d-k_d),$$
 for any given function $h$. We denote by $\mathcal{F}$ the Fourier transform of any function $f$ defined on $\R^d$ by
 $$\mathcal{F}(f)(t)=\int_{\R^d}  e^{-i<t,y>}f(y)dy, \quad t\in \R^d,$$
 where $<.,.>$ denotes the usual scalar product. 
 
 For two integers $a,b$, we denote $a \wedge b: = \min(a,b)$ and  $a \vee b: = \max(a,b)$.  And $ \lfloor y\rfloor$  denotes the largest integer smaller than $y$ :  $ \lfloor y\rfloor  \leq y < \lfloor y\rfloor +1.$ 

 \section{The estimation procedure}\label{secestimation}
 For estimating the regression function $m$, the idea consists in writing $m$ as the ratio
 $$m(x)=\frac{m(x)f_X(x)}{f_X(x)},\quad x\in[0,1]^d.$$
 In the sequel, we denote $$p(x):=m(x)\times f_X(x).$$ 
So, we estimate $p$, then $f_X$.
Since estimating $f_X$ is a classical deconvolution problem, the main task consists in estimating $p$. We propose a wavelet-based procedure with an automatic choice of the maximal resolution level. Section~\ref{sec:kernel} describes the construction of the projection kernel on wavelet bases depending on a maximal resolution level. Section~\ref{sec:level} describes the Goldenshluger-Lepski procedure to select the resolution level adaptively.
\subsection{Approximation kernels and family of estimators for $p$}\label{sec:kernel}
We consider noise densities $g=(g_1,\cdots,g_d)$ which satisfy the following relationship (see  \cite{FanKoo}) :
 \begin{equation}
   \mathcal{F}(g)(t) = \prod_{\el=1}^d \mathcal{F}(g_\el)(t_\el), \quad t_\el\in \R. 
   \end{equation} 
In the sequel, we consider a father wavelet $\varphi$ on the real line satisfying the following conditions: 
\begin{itemize}
 \item  (A1) The father wavelet $\varphi$ is compactly supported on $[-A,A]$, where $A$ is a positive integer. 
\item (A2) There exists  a positive integer $N$, such that for any $x$
$$\int \sum_{k\in\Z}\varphi(x-k)\varphi(y-k)(y-x)^{\ell}dy=\delta_{0\ell}, \quad \ell=0,\ldots,N.$$
\item (A3) $\varphi$ is of class $\mathcal{C}^r$, where $r\geq 2$. 
\end{itemize}
These properties are satisfied for instance by Daubechies and Coiflets wavelets  (see \cite{Hardle}, chapter 8).  The associated projection kernel on the space $$V_j:=\mbox{span}\{ \varphi_{jk},  k\in \mathbb{Z}^d\}, \quad j \in \N^d, $$
is given for any $x$ and $y$ by 
$$K_j(x,y)= \sum_k  \varphi_{jk}(x){\varphi_{jk}(y}),$$
where for any $x$,
 $$\varphi_{jk}(x)=\prod_{\el=1}^d2^{\frac{j_\el}{2}}\varphi(2^{j_\el}x_\el-k_\el), \quad j\in \N^d, \; k \in \Z^d.$$ 
Therefore, the projection of $p$ on $V_j$ can be written for any $z$,
$$p_j(z):=K_j(p)(z):= \int K_j(z,y)p(y)dy=\sum_{k} p_{jk}\varphi_{jk}(z)$$
with $$p_{jk}=\int p(y) {\varphi_{jk}}(y) dy.$$
 First we estimate unbiasedly any projection $p_j$. Secondly to obtain the final estimate of $p$, it will remain to select a convenient value of $j$ which will be done in section \ref{sec:level}. The natural approach is based on unbiased estimation of the projection coefficients $p_{jk}$. To do so, we adapt the kernel approach proposed by \cite{Fan93} in our wavelets context. 
To this purpose, we set
$$
\hat{p}_{jk}:=  \frac 1 n \sum_{u=1}^n Y_u \times (\mathcal{D}_j \varphi)_{j,k}(W_u)=2^{\frac \Sj  2} \frac 1 n \sum_{u=1}^n Y_u \int e^{-i<t,2^jW_u-k>}  \prod_{\el=1}^d
\frac{\overline{\mathcal{F}(\varphi)(t_\el)}}{\mathcal{F}(g_\el)(2^{j_\el}t_\el)} dt, $$
$$\hat{p}_{j}(x)= \frac 1 n   \sum_k \sum_{u=1}^n Y_u \times (\mathcal{D}_j \varphi)_{j,k}(W_u)\varphi_{jk}(x),$$
 where the deconvolution operator $\mathcal{D}_j$ is defined as follows for a function $f$ defined on $\R$ 
 \begin{equation}\label{operatorK_j}
 (\mathcal{D}_j f)(w)= \int  e^{-i<t,w>} \prod_{\el=1}^d
\frac{\overline{\mathcal{F}(f)(t_\el)}}{\mathcal{F}(g_\el)(2^{j_\el}t_\el)} dt, w \in \R^d.
\end{equation}
Lemma \ref{SansBiais}, proved in section \ref{preuveprop} states that $\E[\hat{p}_j(x)]=p_j(x)$ which justifies our approach. Furthermore, the deconvolution operator $(\mathcal{D}_j f)(w)$ in (\ref{operatorK_j}) is the multidimensional wavelet analogous of the operator $K_n(x)$ defined in (2.4) in \cite{Fan93}: the Fourier transform of their kernel $K$ has been replaced in our procedure by the Fourier transform of the wavelet $\varphi_{jk}$ and their bandwith $h$ by $2^{-j}$.

Note that the definition of the estimator $\hat{p}_{j}(x)$ still makes sense when we do not have any noise on the variables $X_\el$ i.e $g(x)=\delta_0(x)$ because
in this case $\mathcal{F}(g)(t)=1$.
\subsection{Selection rule by using the Goldenshluger-Lepski methodology}\label{sec:level}
The second and final step consists in selecting the multidimensional resolution level $j$ depending on $x$ and based on a data-driven selection rule inspired from a method exposed in \cite{GL}. To define this latter we have to introduce some quantities.
In the sequel we denote for any $w \in \R^d$,
$$T_j(w):=\sum_k (\mathcal{D}_j \varphi)_{j,k}(w) \varphi_{jk}(x)$$
and
$$U_j(y,w):=y\sum_k (\mathcal{D}_j \varphi)_{j,k}(w) \varphi_{jk}(x)=y\times T_j(w),$$
so we have
$$\hat{p}_{j}(x)=\frac{1}{n}\sum_{u=1}^nU_j(Y_u,W_u).$$
Proposition~\ref{Th:conc} {in Section \ref{preuveprop}} shows that $\hat p_j(x)$ concentrates around $p_j(x)$. So the idea is to find  a maximal resolution $\hat j$ that mimics the oracle index. The oracle index minimizes a bias variance trade-off. So we have to find an estimation for the bias-variance decomposition of $\hat p_j(x)$. We denote $ \sigma_j^2:= \var(U_j(Y_1,W_1)) $ and the variance of $\hat p_j$ is thus equal to $\frac{\sigma_j^2}{n}.$ We set :
\begin{equation} \label{sigmahat}
\hat\sigma^2_j:=\frac{1}{n(n-1)}\sum_{\el =2}^n\sum_{\vv=1}^{\el -1}(U_j(Y_\el ,W_\el )-U_j(Y_\vv,W_\vv))^2, 
\end{equation} 
and since $\E(\hat \sigma_j^2)=\sigma_j^2$, $\hat\sigma^2_j$ is a natural estimator of $\sigma^2_j$. To devise our procedure, we introduce a slightly overestimate of $\sigma_j^2$  given by:
\begin{equation}\label{sigmatilde}
\tilde\sigma^2_{j,\tilde\gamma}:=\hat\sigma^2_j+2C_j\sqrt{2 \tilde\gamma \hat\sigma_j^2\frac{\log n}{n}}+8 \tilde\gamma  C_j^2 \frac{\log n}{n},
\end{equation}
where $\tilde \gamma$ is a positive constant and 
 $$C_j : =\left(\|m\|_\infty+s\sqrt{2\tilde\gamma\log n}\right)\|T_j\|_\infty.$$
 For any $\eps>0$, let $\gamma > 0$ and 
  \begin{equation*}
   {\Gamma}_{\gamma}(j): =  \sqrt{\frac{2 \gamma (1+\e)\tilde\sigma_{j,\tilde\gamma}^2\log n}{n}}+\frac{c_j \gamma \log n}{n}, 
  \end{equation*}
  where 
  $$c_j : =16\left(2\|m\|_\infty+s\right)\|T_j\|_\infty.$$ Let
$$ { \Gamma}_{\gamma}(j,j'):= {\Gamma}_{\gamma}(j)+ {\Gamma}_{\gamma}(j\wedge j'),  $$
and
 \begin{equation}\label{Gammaqjetoile}
  {\Gamma}_{\gamma}^*(j):=\sup_{j'}  { \Gamma}_{\gamma}(j,j').
   \end{equation} 
We now define the selection rule for the resolution index. Let
\begin{equation}\label{R_j}
\hat R_j:=\sup_{j'}\Big\{\left|\hat p_{j\wedge j'}(x)-\hat p_{j'}(x)\right|-{ \Gamma}_{\gamma}(j',j)\Big\}_{+}+ 
{\Gamma}^*_{\gamma}(j).
\end{equation}
Then $\hat p_{\hat j}(x)$ is the final estimator of $p(x)$ with $\hat{j}$ such that
\begin{equation}\label{hatj}
\hat j:=\arg\min_{j\in J } \hat R_j,
\end{equation}
where the set $J$ is defined as 
\begin{equation}\label{setJ}
 J: =\left \{ j\in \N^d:\quad 2^{\Sj}\leq \left \lfloor  {\frac{n}{\log^2n }} \right \rfloor \right \}.
 \end{equation}
Now, we shall highlight how the above quantities interplay in the estimation of the risk decomposition of $\hat p_j$. 
An inspection of the proof of Theorem \ref{oracle} shows that a control of  the bias of $\hat p_j$ is provided by :
$$ \sup_{j'}\Big\{\left|\hat p_{j\wedge j'}(x)-\hat p_{j'}(x)\right|-{ \Gamma}_{\gamma}(j',j)\Big\}_{+}.$$ 
The term  $|\hat p_{j\wedge j'}(x)-\hat p_{j'}| $ is classical when using the Goldenshluger Lepski method (see sections 2.1 and 5.2 in \cite{BLR}).  Furthermore for technical reasons (see proof of Theorem \ref{oracle}), we do not estimate the variance of $\hat p_j(x)$ by $\frac{\hat \sigma_j^2}{n}$ but we replace it by  $\Gamma^2_\gamma(j)$. Note that we have the straightforward control 
$$
\Gamma_\gamma(j) \leq C \left (\hat \sigma_j \sqrt{\frac{\log n }{n}} + (C_j+c_j) \frac{\log n }{n}\right ),
$$
where $C$ is a constant depending on $\varepsilon$, $\tilde \gamma$ and $\gamma$. 
Actually we prove   that $\Gamma^2_\gamma(j) $ is of order $\frac{\log n}{n} \sigma_j^2 $ (see Lemma \ref{Rosenthal} and \ref{ordredegrandeursigmajTj}). The dependance  of  $\tilde\sigma^2_{j,\tilde\gamma}$ (\ref{sigmatilde}) in $\| m\|_{\infty}$ appears only in smaller order terms.
In conclusion, up to the knowledge of $\| m\|_{\infty}$  the procedure is completely data-driven.  Next section explains how to choose the constants $\gamma$ and $\tilde \gamma$. Our approach is non asymptotic and based on sharp concentration inequalities.  

\bigskip



 


\section{Rates of convergence} \label{secrates}
There exists $C_1>0$ such that for any $x\in [0,1]^d$, $f_X(x)\geq C_1$.

As we face a deconvolution problem, we need to define the assumptions made on the smoothness of the density of the errors covariates $g$. There exist positive constants $c_g$ and $C_g$ such that 
  \begin{equation}\label{hypobruit1}
  c_g  {(1+|t_\el|)^{-\nu }}  \leq  |\mathcal{F}(g_\el)(t_\el)| \leq C_g {(1+|t_\el|)^{-\nu }}, \quad 0 \leq \nu \leq  r-2, \quad t_l\in \R.
 \end{equation}
   
 We also require a condition for the derivative of the Fourier transform of $g$. There exists a positive constant $\mathcal{C}_g$ such that
     \begin{equation} \label{hypobruit2}
       |  {\mathcal{F}'(g_\el)(t_\el)}|  \leq  \mathcal{C}_g {(1+|t_\el|)^{-\nu-1}},  \quad t_l\in \R. 
     \end{equation}
Laplace and Gamma distributions satisfy the above assumptions (\ref{hypobruit1}) and (\ref{hypobruit2}). Assumptions (\ref{hypobruit1}) and (\ref{hypobruit2}) control the decay of the Fourier transform of  $g$ at a polynomial rate. Hence we deal with a midly ill-posed inverse problem. The index $\nu$ is usually known as the degree
of ill-posedness of the deconvolution problem at hand. \\

\subsection{Oracle inequality and rates of convergence for $p(\cdot)$}
First, we state an oracle inequality which highlights the bias-variance decomposition of the risk.
\begin{Th} \label{oracle}
 Let $ q\geq1 $ be fixed and let $ \hat j $ be the adaptive index defined as above. Then, it
holds for any $ \gamma>q(\nu+1) $ and $\tilde \gamma > 2q(\nu+2),$
$$
\E \left [ \left|\hat p_{\hat j}(x)-p(x)\right|^q \right ] \leq R_1 \left( \inf_{\eta} \E \left  [ \left\{B(\eta)+
\Gamma^*_{\gamma}(\eta)\right\}^q \right ]\right)+ o(n^{-q})  ,
$$
where 
$$
 B(\eta):=\max\left(\sup_{j'}\left|\E \left [\hat
p_{\eta\wedge j'}(x) \right ] -\E \left [ \hat p_{j'}(x)\right  ] \right | ,   \left |   \E [\hat p_{\eta}(x)]-p(x)\right|  \right )
$$  and $R_1$ a constant depending only on $q$.
\end{Th}

The oracle inequality in Theorem \ref{oracle} illustrates a bias-variance decomposition of the risk. The term $B(\eta)$ is a bias term. Indeed, one recognizes on the right side the classical bias term $$ \left |   \E [\hat p_{\eta}(x)]-p(x)\right| =  |p_{\eta}(x)-p(x) | . $$ Concerning  $  \left|\E \left [\hat p_{\eta\wedge j'}(x) \right ] -\E \left [ \hat p_{j'}(x)\right  ] \right | $, for sake of clarity let us consider for instance the univariate case : if $j'\leq \eta$ this term is equal to zero. If $j' \geq \eta $, it turns to be 
$$ |\E \left [\hat p_{\eta}(x) \right ] -\E \left [ \hat p_{j'}(x)\right  ] |= |p_\eta(x)-p_{j'}(x)| \leq | p_\eta(x)-p(x)| + | p_{j'}(x)-p(x)| .$$ As we have the following inclusion for the projection spaces $ V_\eta \subset V_{j'}  $, the term  $p_{j'}$ is closer to $p$ than $p_\eta$ for the $L_2$-distance. Hence we expect a good control of $  |p_{j'}(x)-p(x)|$ with respect to  $| p_\eta(x)-p(x)| $.
  
We study the rates of convergence of the estimators over anisotropic H\"older Classes. Let us define them.
 \begin{definition}[Anisotropic H\"{o}lder Space]\label{def_holder_Anisot}
Let $ \vec{\beta}=(\beta_1,\beta_2,\ldots,\beta_d)\in(\R_+^*)^d $ and $ L>0 $.
 We say that $f: [0,1]^d\rightarrow\R $ belongs to the anisotropic H\"{o}lder class
$\mathbb{H}_d(\vec{\beta},L)$ of functions if $f$ is bounded and for any $ \el =1,...,d $ and for all $
z\in\R $
$$
\sup_{x\in[0,1]^d}\left|\frac{\partial^{\lfloor \beta_l\rfloor }f}{\partial x_\el ^{\lfloor
\beta_\el \rfloor}}(x_1,\ldots,x_\el +z,\ldots,x_d)-\frac{\partial^{\lfloor \beta_\el \rfloor }f}{\partial x_\el ^{\lfloor
\beta_\el \rfloor}}(x_1,\ldots,x_\el ,\ldots,x_d)\right|\leq
L|z|^{\beta_\el -\lfloor \beta_\el \rfloor}.
$$
\end{definition}

The following theorem  gives the rate of convergence of the estimator $\hat p_{\hat j}(x)$ and justifies the optimality of our oracle inequality. 

\begin{Th}\label{th adaptive minimax p}
Let $ q\geq1 $ be fixed  and let $ \hat j $ be the adaptive index defined in (\ref{hatj}). 
Then, for any $
\vec\beta\in(0,1]^d $ and $L>0$, it holds
$$
\sup_{p\in\mathbb{H}_d(\vec{\beta},L)}\E\left|\hat p_{\hat j}(x)-p(x)\right|^q\leq  L^{\frac{q(2\nu+1)}{2 \bar \beta +2\nu +1}}  R_2
\left(\frac{\log(n)}{n}\right)^{q\bar\beta/(2\bar\beta+2\nu+1)}
,$$
with $\bar\beta= \frac{1}{\frac{1}{\beta_1}+\dots+ \frac{1}{\beta_d}}$ and $R_2$ a constant depending on 
$\gamma,q,\eps, \tilde \gamma, \| m\|_{\infty}, s,  \| f_X\|_{\infty},\varphi, c_g, \mathcal{C}_g, \vec{\beta}$. 
\end{Th}

\begin{Remark}\label{rate_optimal}
 The estimate $ \hat p $ achieves the optimal rate of convergence up to a logarithmic term (see section 3.3 in  \cite{ComteLacour}). This logarithmic loss, due to adaptation, is known to be nevertheless unavoidable for d = 1 and one
can conjecture that it is also the case for higher dimension (see Remark 1 in \cite{ComteLacour}) .  

\end{Remark}

\subsection{Rates of convergence  for $m(\cdot)$}

As mentioned above, the estimation of $ m $ requires an adaptive estimate of $ f_X $. This is due to kernel estimators, e.g. projection estimators do not need the additional estimate (see \cite{BLR}). For this purpose, we use an estimate introduced by \cite{ComteLacour} (Section 3.4) denoted by $ \hat f_X $. This estimate is constructed from a deconvolution kernel and the bandwidth is selected via a method described in \cite{GL}. We will not give the explicit expression of $\hat f_X$ for ease of exposition. Then, we define the estimate of $m$ for all $x$ in $[0,1]^d$ :
\begin{equation}\label{estimateurdem}
\hat m(x)=\frac{\hat{p}_{\hat j}(x)}{ \hat f_X(x)\vee n^{-1/2}}.
\end{equation}
The term $n^{-1/2}$ is added to avoid the drawback when $\hat f_X$ is closed to $0$.

\begin{Th}\label{th adaptive minimax m} Let $ q\geq1 $ be fixed and let $ \hat m $ defined as above. Then, for any $
\vec\beta\in(0,1]^d $ and $L>0$, it holds
$$
\sup_{m \in\mathbb{H}_d(\vec{\beta},L)}\E\left|\hat m(x)-m(x)\right|^q\leq    L^{\frac{q(2\nu+1)}{2 \bar \beta +2\nu +1}} R_3
\left(\frac{\log(n)}{n}\right)^{q\bar\beta/(2\bar\beta+2\nu+1)},
$$
with $R_3$ a constant depending on 
$\gamma,q,\eps, \tilde \gamma, \| m\|_{\infty}, s,  \| f_X\|_{\infty},\varphi, c_g, \mathcal{C}_g, \vec{\beta}$.
\end{Th}

 The estimate $ \hat m $ is again optimal up to a logarithmic term (see Remark \ref{rate_optimal}).

\bigskip

\section{Numerical results }\label{secsimu}
In this section, we implement some simulations to illustrate the theoretical results. We aim at estimating the Doppler regression function $m$ at two points $x_0= 0.25$ and $ x_0=0.90 $  (see Figure 1).  We have  $n=1024$ observations and the regression errors $\varepsilon_l$'s follow a standard normal density with variance $s^2=0.15^2$. As for the design density of the $X_l$'s, we consider the Beta density and the uniform density on $[0,1]$. The uniform distribution is quite classical  in regression with random design. The $Beta(2,2)$ and $Beta(0.5,2)$ distributions reflect two very different behaviors on $[0,1]$.  Indeed, we recall that the Beta density with parameters $(a,b)$ (denoted here by  ${Beta(a,b)}$) is proportional to $x^{a-1}(1-x)^{b-1}\mathds{1}_{[0,1]}(x).$ In Figure 2, we plot the noisy regression Doppler function according to the three design scenario. For the covariate errors $\delta_i$'s, we focus on the centered Laplace density with scale parameter $\sigma_{g_L}>0$ that we denote $g_L$. This latter has the following expression :
$$g_L(x)= \frac{1}{2\sigma_{g_L}} e^{-\frac{|x|}{\sigma_{g_L}}}.$$ The choice of the centered Laplace noise is motivated by the fact that the Fourier transform of $g_L$ is given by  
$$ \mathcal{F}(g_L)(t)=\frac{1}{1+ \sigma_{g_L}^2t^2},$$
and according to assumption (\ref{hypobruit1}), it gives an example of an ordinary smooth noise with degree of ill-posedness $\nu=2$. 
Furthermore, when facing regression problems with errors in the design, it is common to compute the so-called reliability ratio (see \cite{Fan93}) which is given by
$$ R_r := \frac{\var(X)}{\var(X)+2\sigma_{g_L}^2}. $$
 $R_r$ permits to assess the amount of noise in the covariates. The closer to $0$ $R_r$ is, the bigger the amount of noise in the covariates is and the more difficult the deconvolution step will be.  For instance, \cite{Fan93} chose $R_r=0.70$. We computed the reliability ratio in Table \ref{reliability} for the considered simulations. 
  
\begin{table}[!h]
\begin{center}
\begin{minipage}[t]{.49\linewidth}
\begin{tabular}{c|c c c }
 $\sigma_{g_L}$ &   \multicolumn{3}{c}{design of the $X_i$} \\    & $\mathcal{U}[0,1]$  &  $Beta(2,2)$  &  $Beta(0.5,2)$    \\ 
 \hline
 0.075 & 0.88  &0.81    & 0.80     \\
0.10  &    0.80 & 0.71   &   0.69  \\ 
\end{tabular}\\
\end{minipage}
\end{center}
\caption{Reliability ratio.}
\label{reliability}
\end{table}

 We recall that our estimator of $m(x)$  is given by the ratio of two estimators (see (\ref{estimateurdem})) :

 \begin{equation}
\hat m(x)=\frac{\hat{p}_{\hat j}(x)}{ \hat f_X(x)\vee n^{-1/2}}.
\end{equation}
First, we compute $\hat{p}_{\hat j }(x)$ an estimator of $p(x)=m(x)\times f_X(x)$ which is denoted "GL" in the graphics below. We use coiflet wavelets of order $5$. Then we divide $\hat{p}_{\hat j }(x)$ by the adaptive deconvolution density estimator $\hat f_X(x)$ of \cite{ComteLacour}. This latter is constructed with a deconvolution kernel and an adaptive bandwidth. For the selection of the coiflet level $\hat j$ in $\hat{p}_{\hat j }(x)$,  we advise to use  $\hat\sigma_j^2$ instead of  $\tilde\sigma_{j,\tilde\gamma}^2$  and $\frac{2\max_i|Y_i|\|T_j\|_\infty}{3}$ instead of $c_j$. It remains to settle the value of the constant $\gamma$. To do so, we compute the pointwise risk of $\hat{p}_{\hat j}(x)$ in function of $\gamma$: Figure 3 shows a clear "dimension jump" and accordingly the value $\gamma=0.5$ turns to be reasonable. Hence we fix $\gamma=0.5$ for all simulations and our selection rule is completely data-driven.  \\

 \begin{figure}[!ht]\begin{center}
\begin{tabular}{ccc}
\includegraphics[scale=0.2]{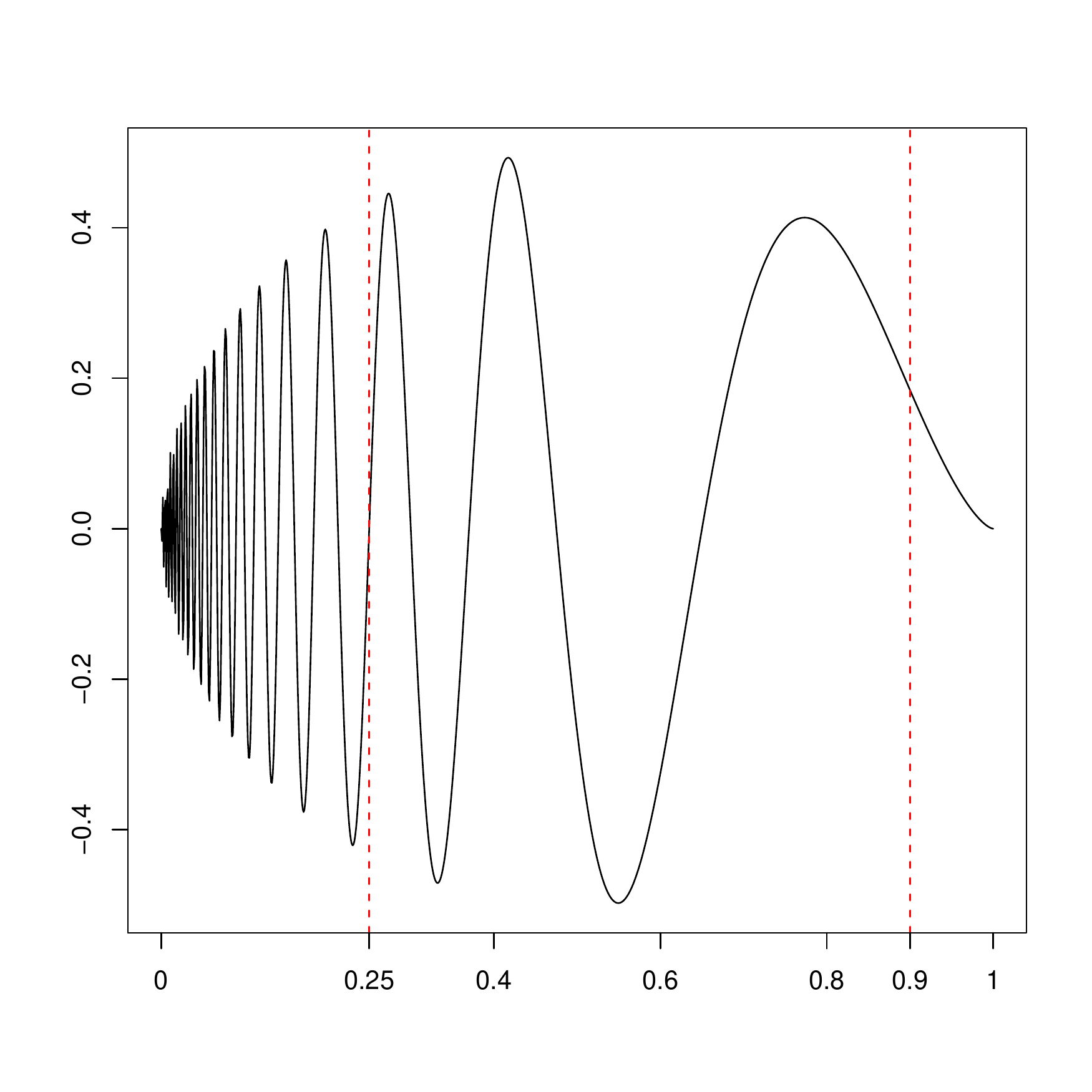} &
\includegraphics[scale=0.2]{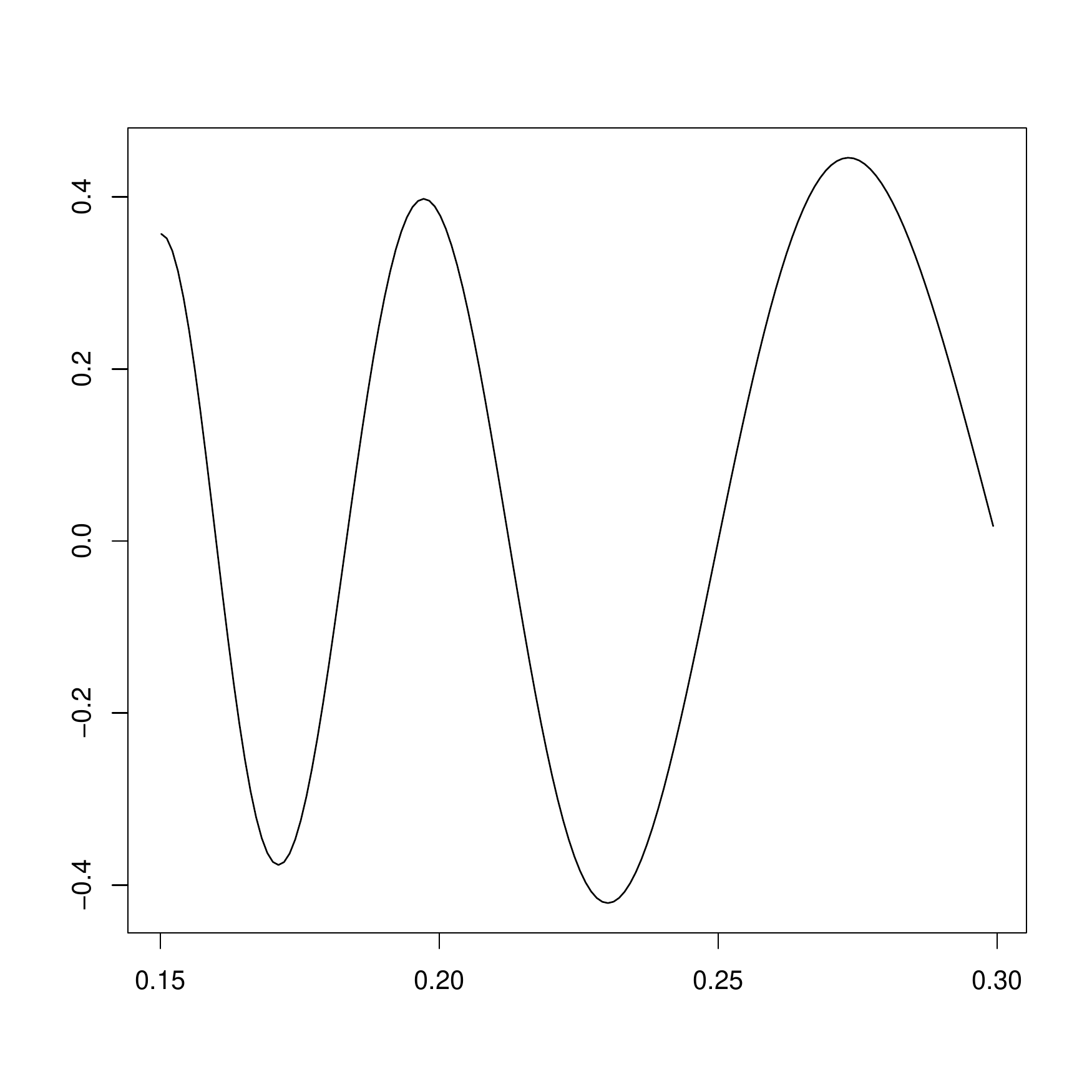} &
\includegraphics[scale=0.2]{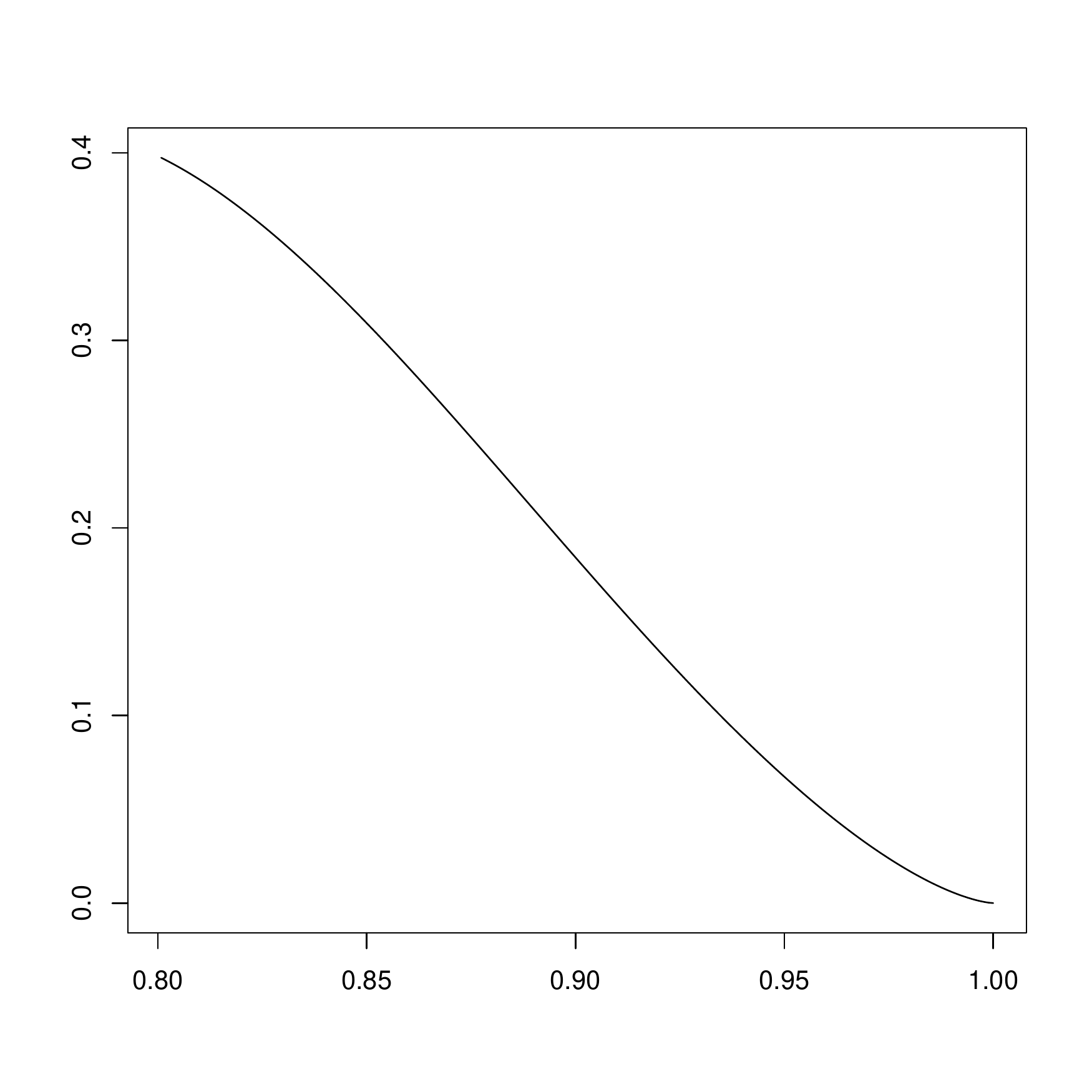} \\
(a) & (b) & (c)
\end{tabular}
\caption{a/ Representation of Doppler function. b/ A zoom of Doppler function on $[0.15, 0.30]$. c/  A zoom of Doppler function on $[0.80,1]$.}
\label{doppler}
\end{center}\end{figure}

 \begin{figure}[!ht]\begin{center}
\begin{tabular}{ccc}
\includegraphics[scale=0.2]{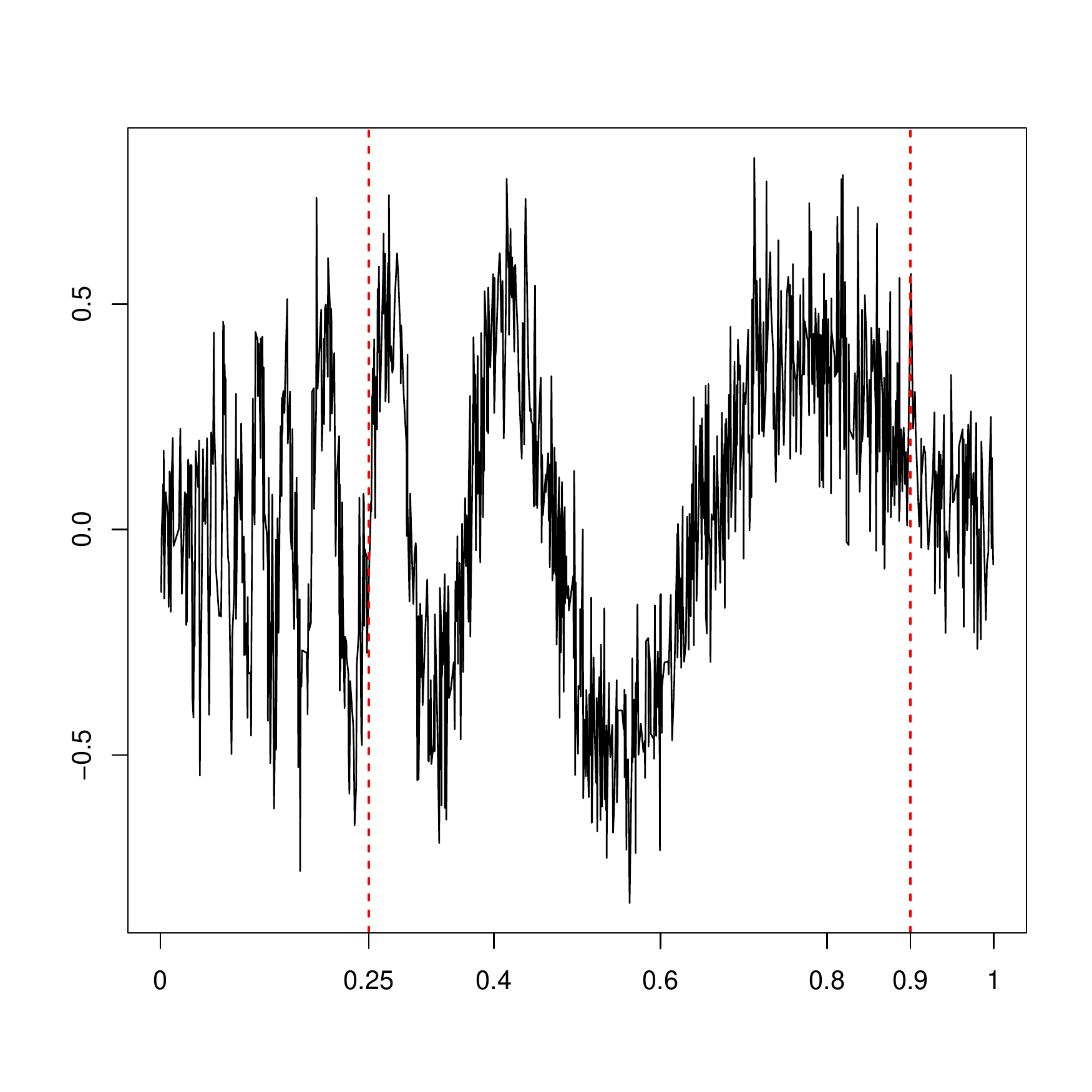} &
\includegraphics[scale=0.2]{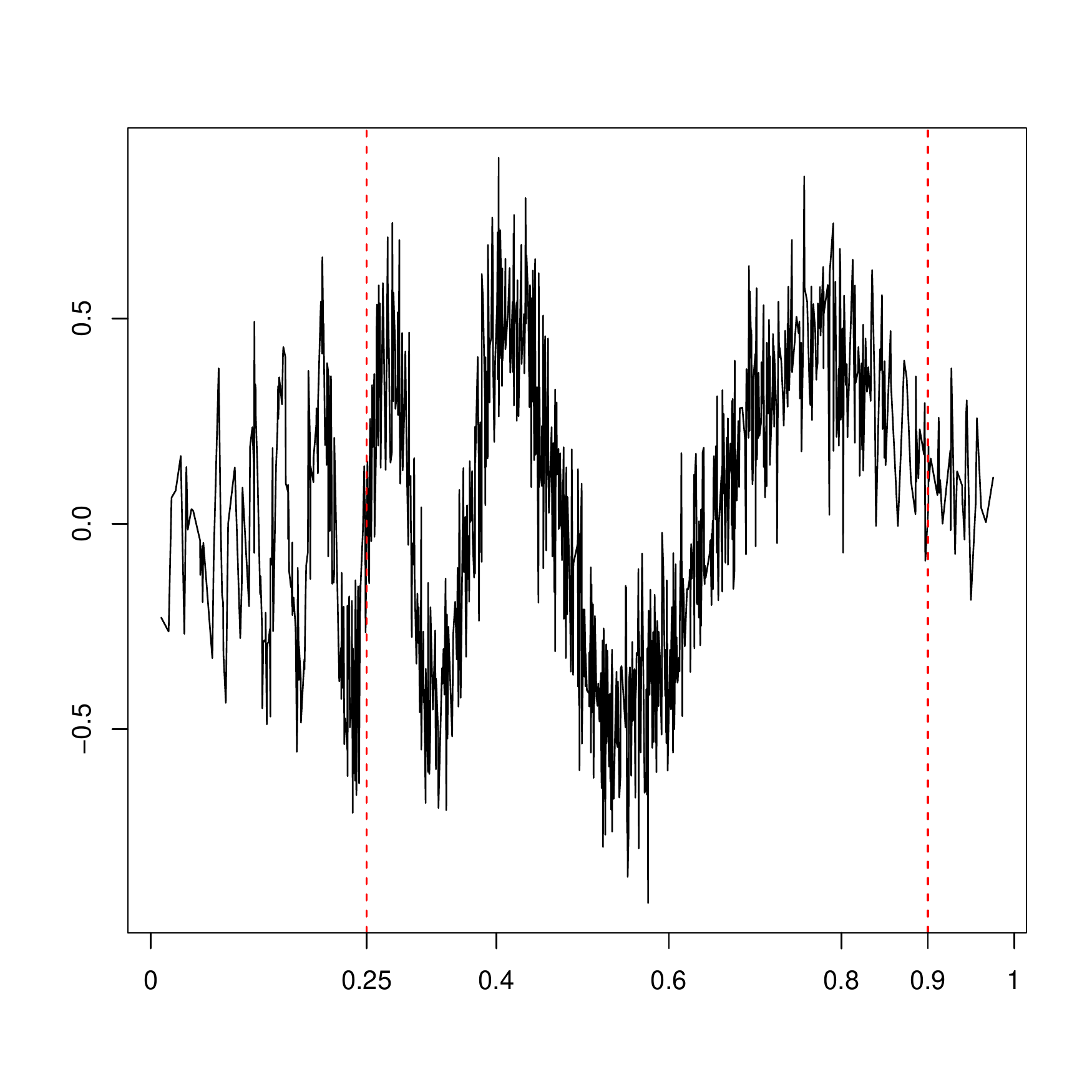}&
\includegraphics[scale=0.2]{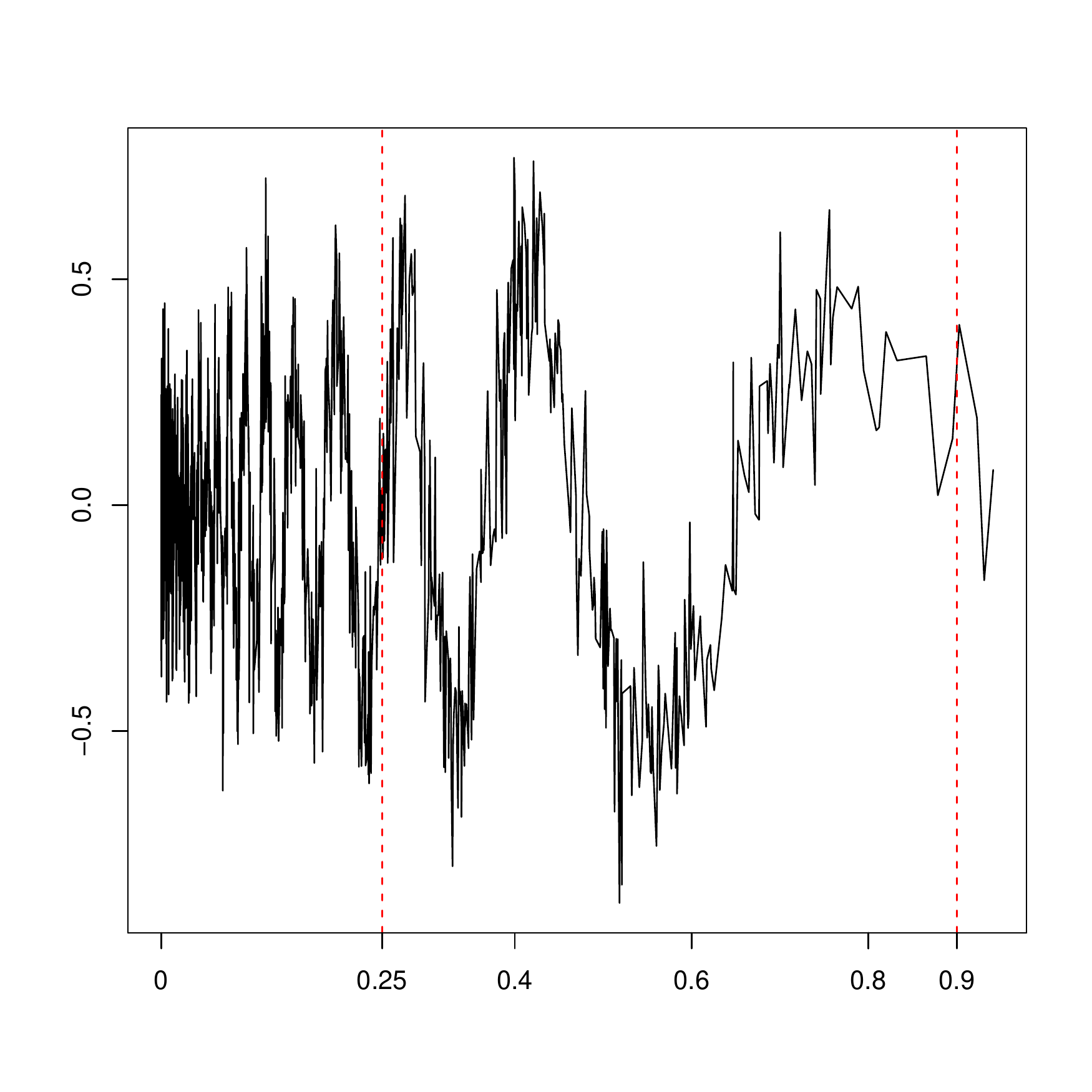} 
\\
(a) & (b) & (c)
\end{tabular}
\caption{a/ Noisy Doppler  with $ X_i \sim \mathcal{U}[0,1]$. b/ Noisy Doppler  with $X_i \sim {Beta(2,2)}$. c/ Noisy Doppler function with $X_i \sim {Beta(0.5,2)}$. }
\label{dopplerbruit}
\end{center}\end{figure}

 \begin{figure}[!ht]\begin{center}
\begin{tabular}{c}
\includegraphics[scale=0.25]{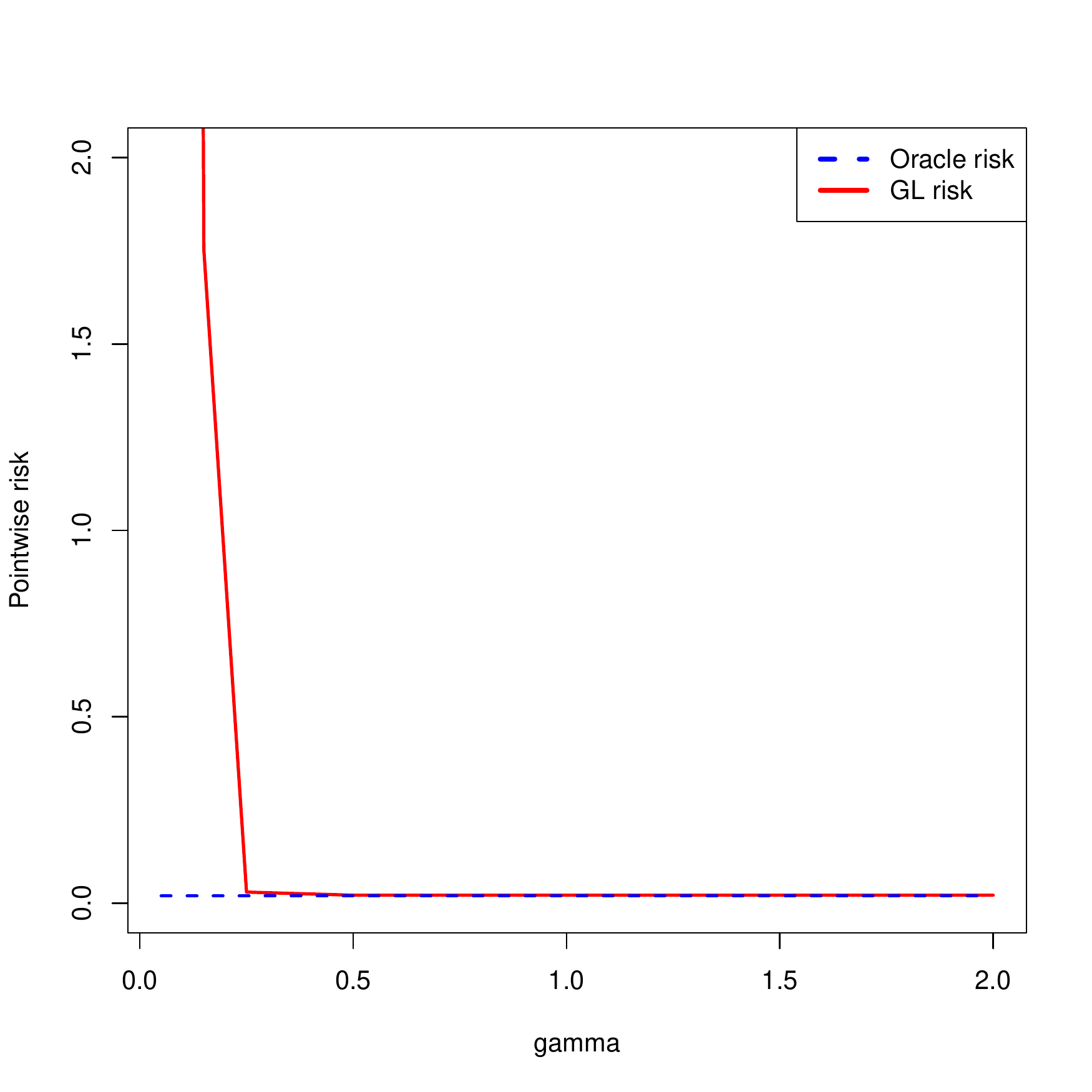} 
\end{tabular}
\caption{Pointwise risk of $\hat {p}_{\hat j }$ at $x_0=0.25$ in function of parameter $\gamma$ for the $Beta(2,2)$ design and $\sigma_{g_L}=0.075$.}
\label{risquegamma}
\end{center}\end{figure}

 \begin{figure}[!h]
 \def\widthboxplot1{4.2cm}
 \centering
\begin{tabular}{cccc }
& $\mathcal{U}[0,1]$ & $Beta(2,2)$ & $Beta(0.5,2)$ \\
$\sigma_{g_L}=0.075$ & \parinc{\widthboxplot1}{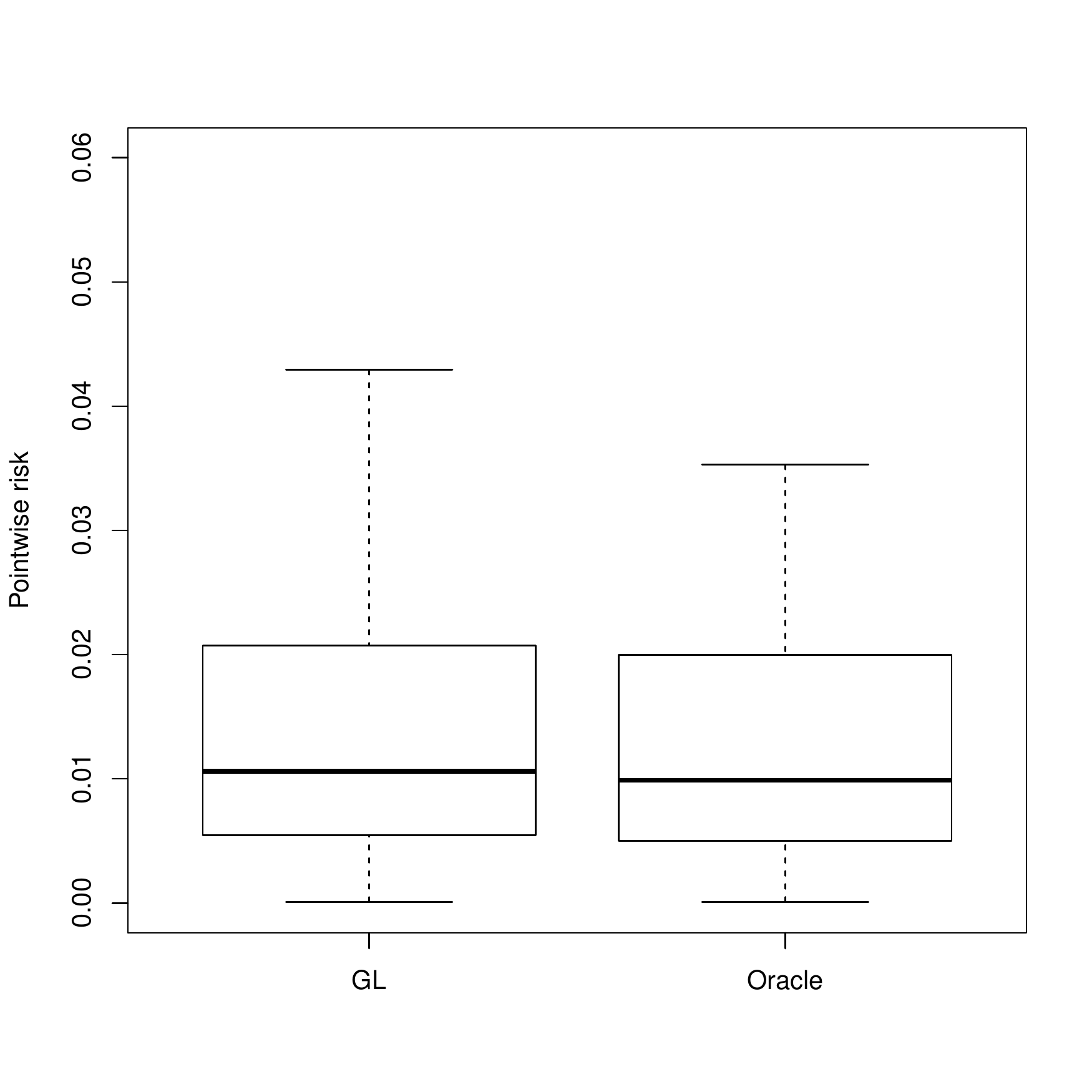} &
 \parinc{\widthboxplot1}{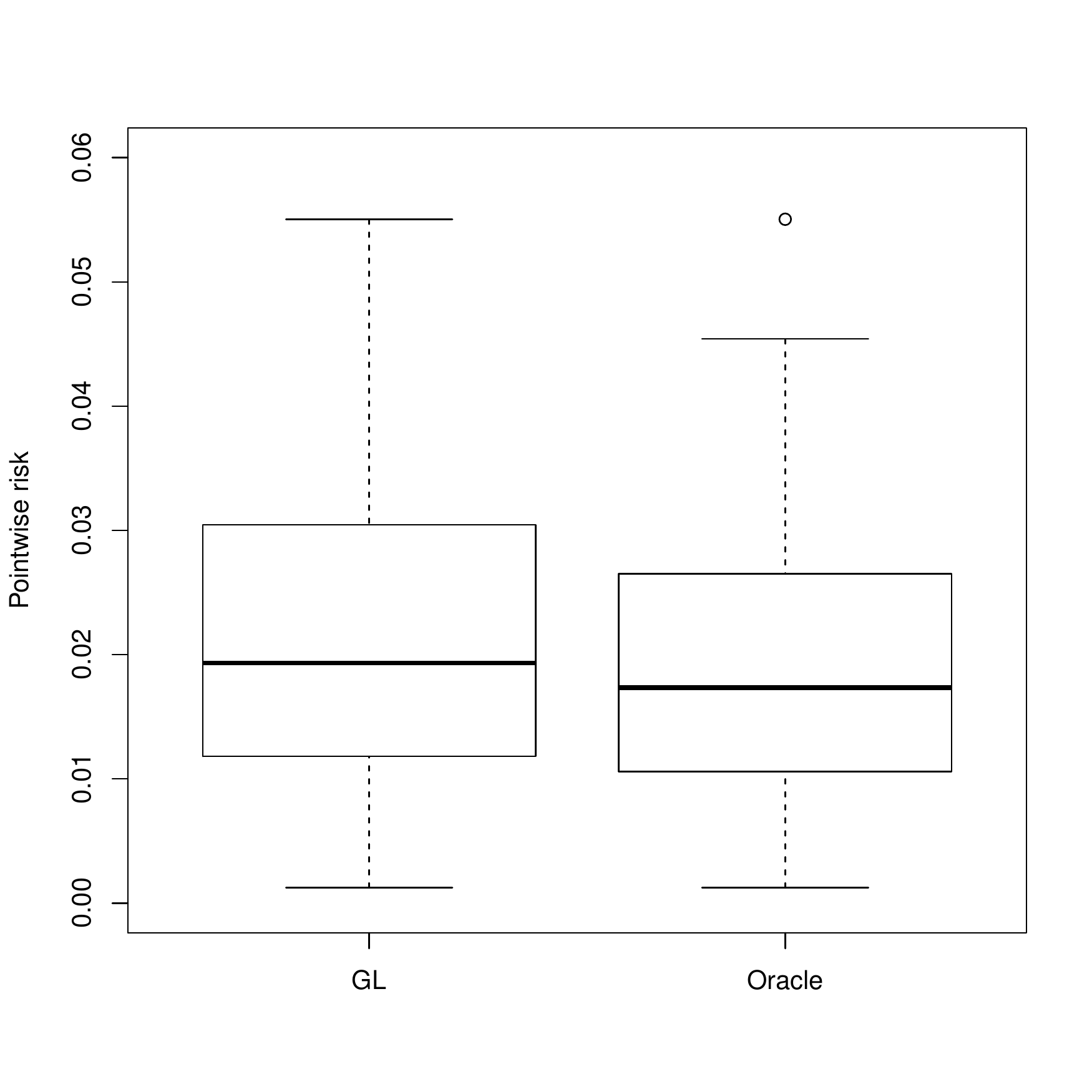} &
 \parinc{\widthboxplot1}{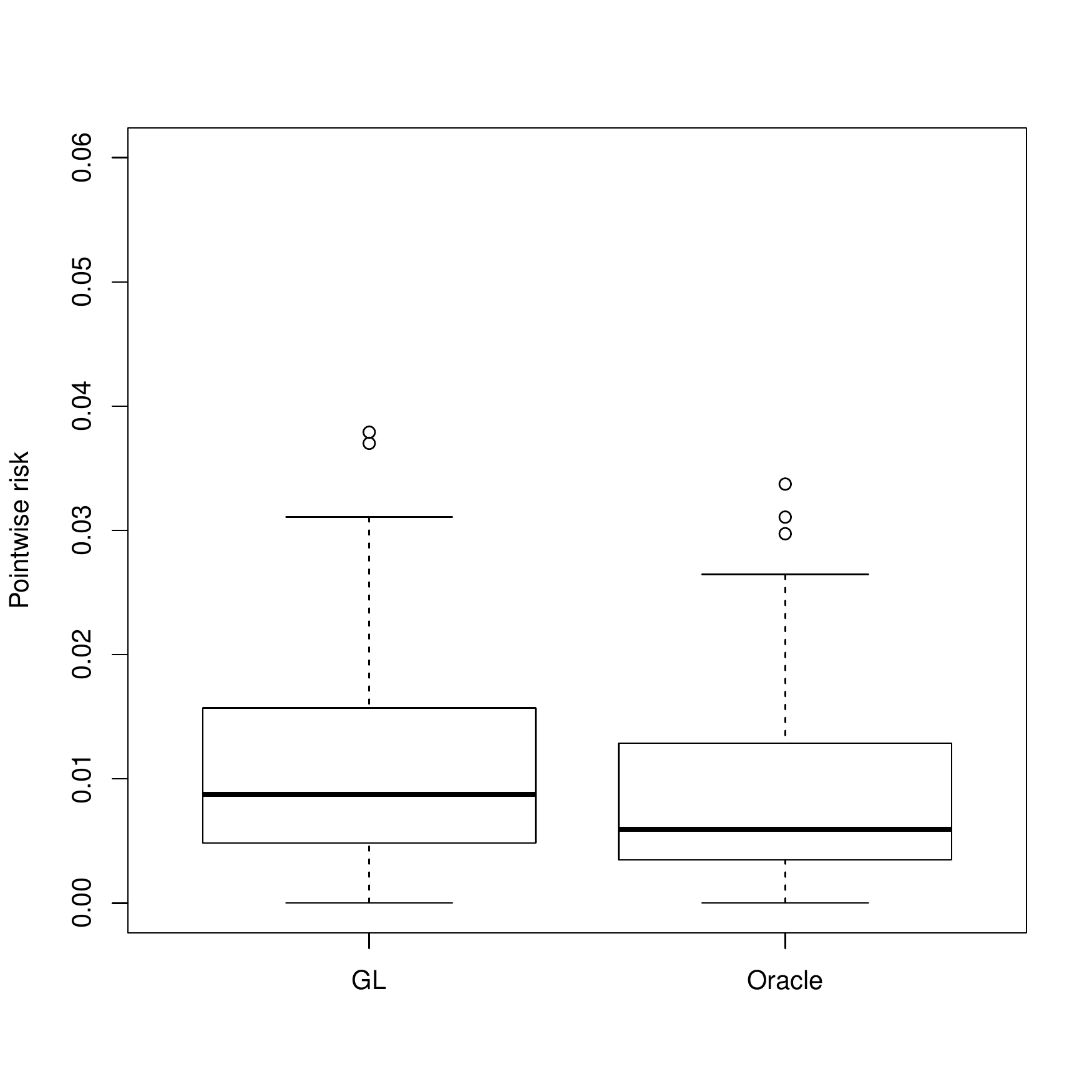} \\
$ \sigma_{g_L}=0.10$ & \parinc{\widthboxplot1}{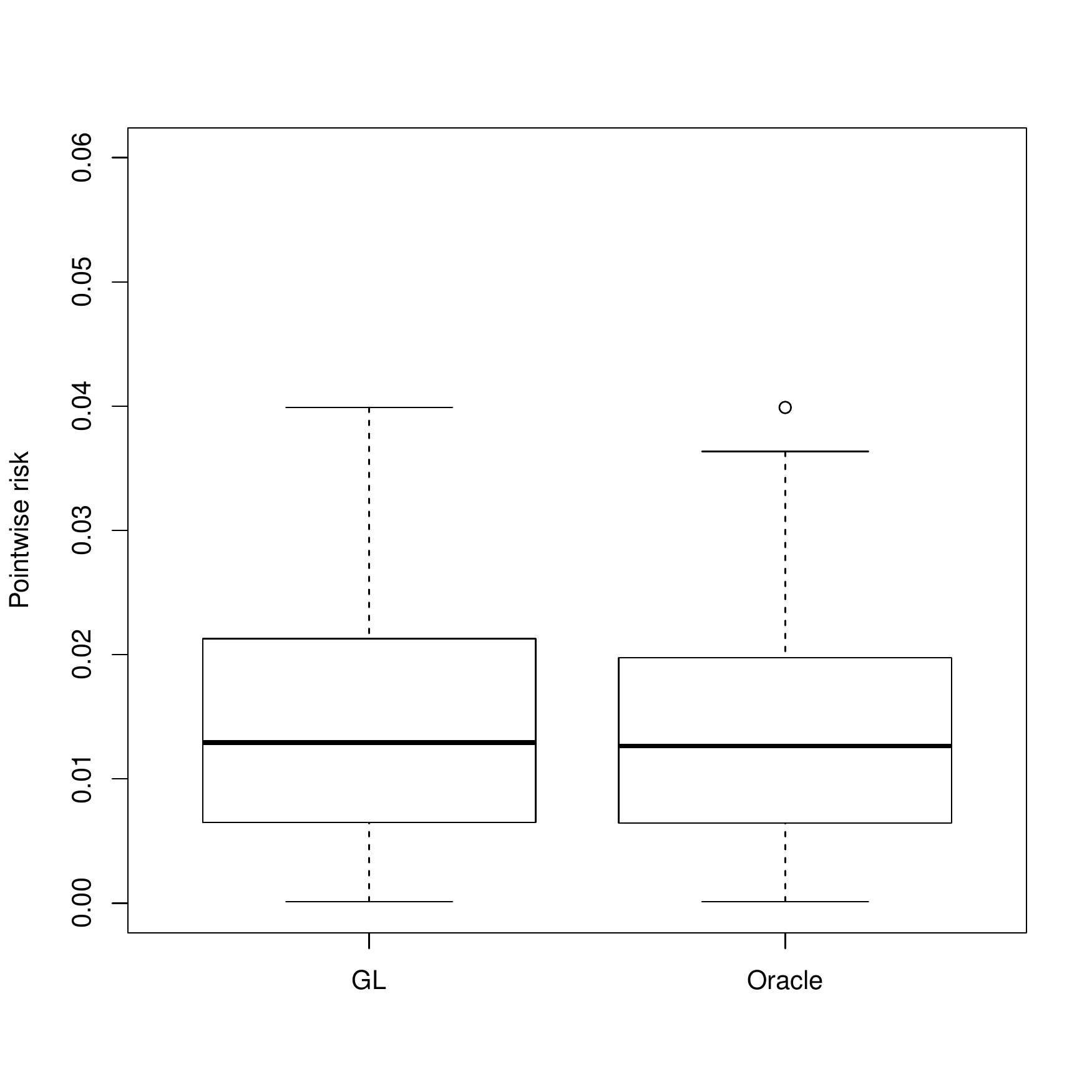} &
 \parinc{\widthboxplot1}{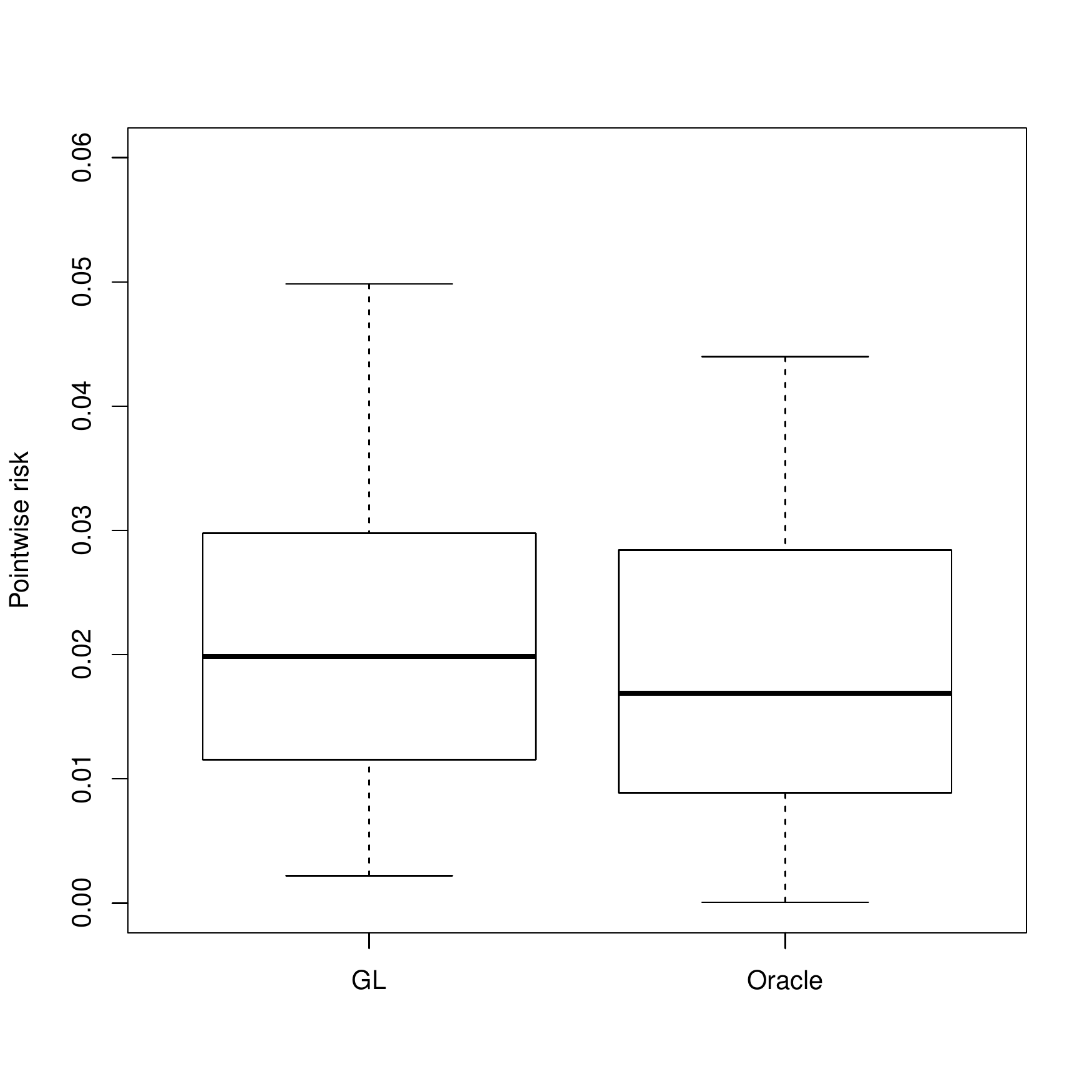}&
 \parinc{\widthboxplot1}{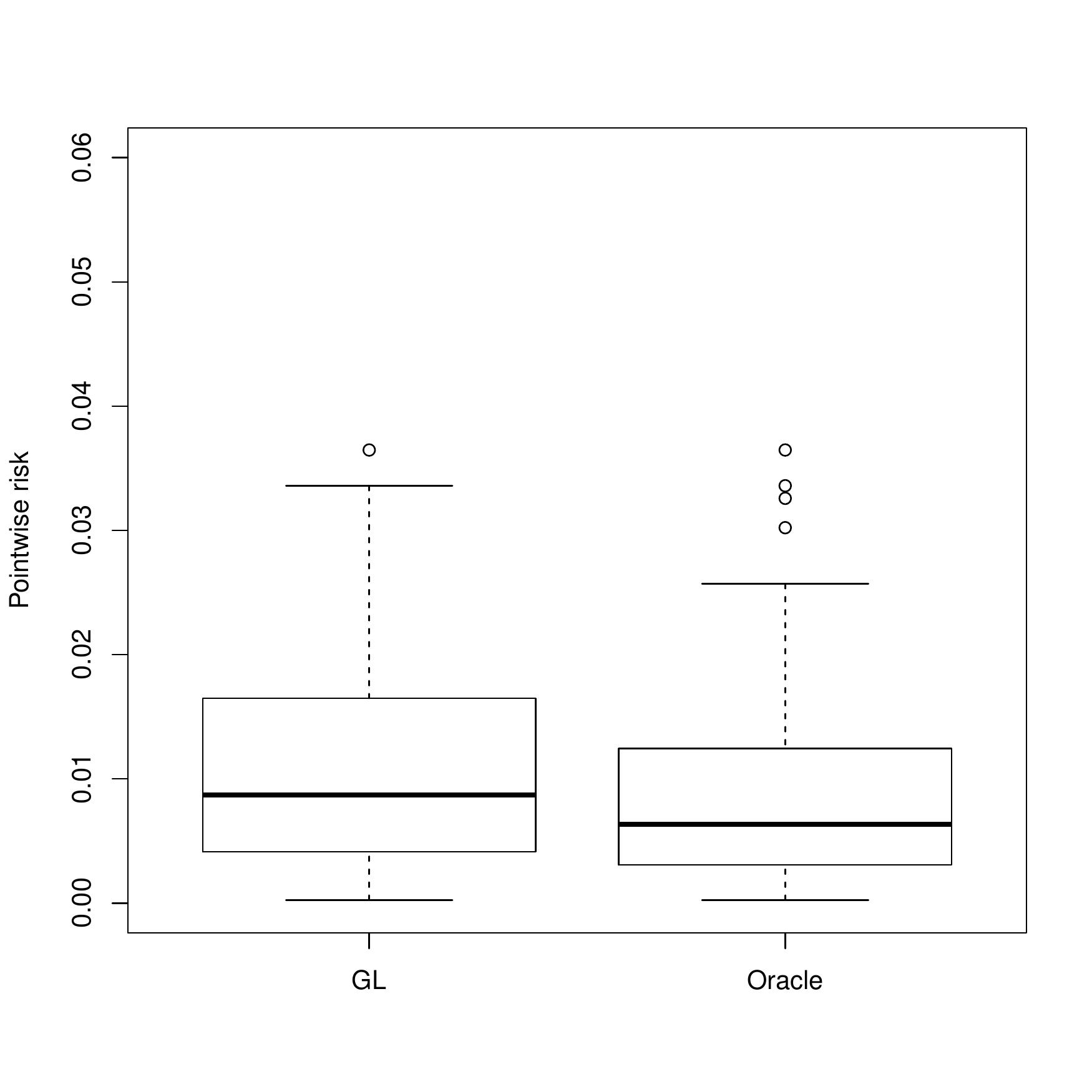} \\
\end{tabular}
\caption{Estimation of $p(x)$ at $x_0=0.25$ }
\label{boxplotx02}
\end{figure}

 \begin{figure}[!h]
 \def\widthboxplot1{4.2cm}
 \centering
\begin{tabular}{cccc }
& $\mathcal{U}[0,1]$ & $Beta(2,2)$ & $Beta(0.5,2)$ \\
$\sigma_{g_L}=0.075$ & \parinc{\widthboxplot1}{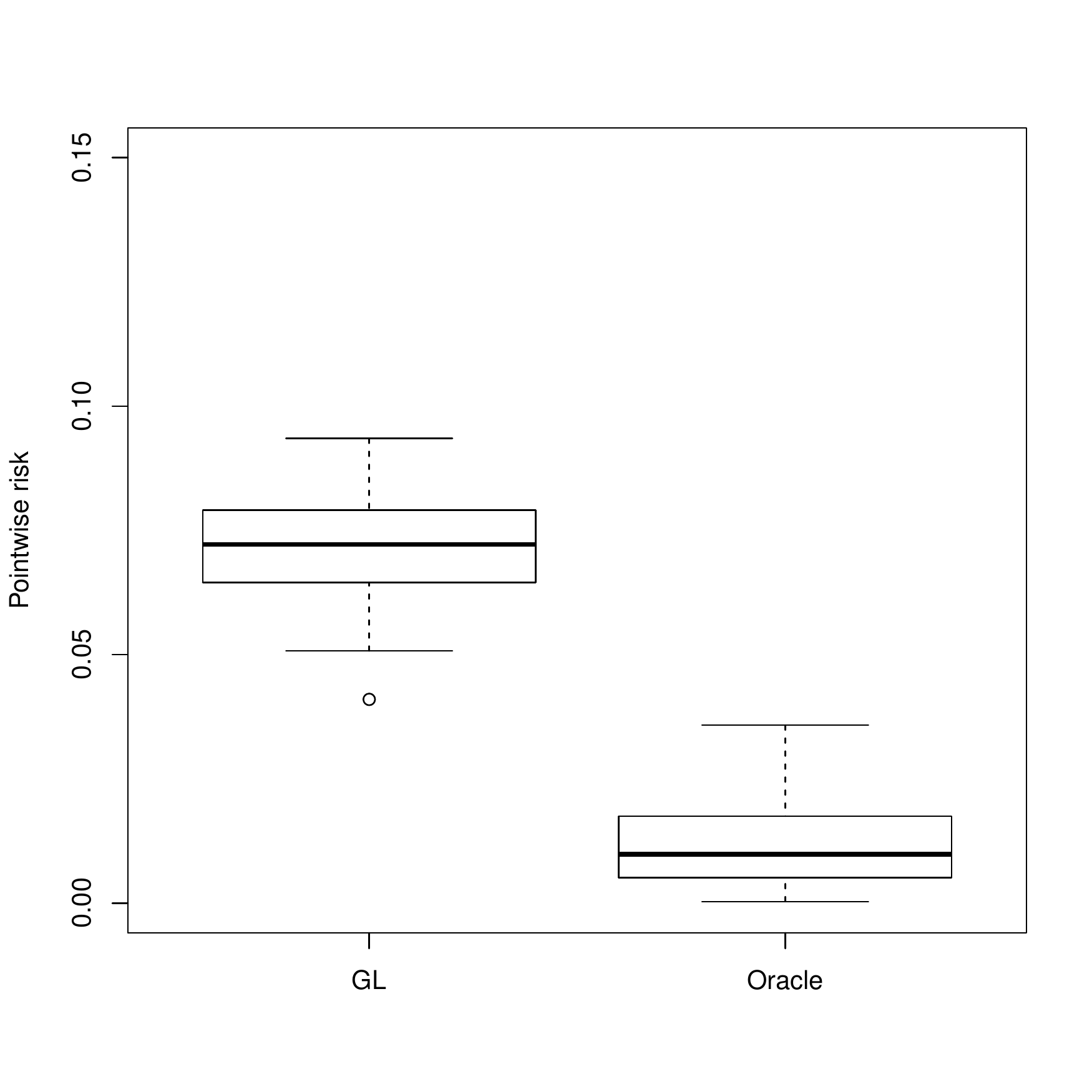} &
 \parinc{\widthboxplot1}{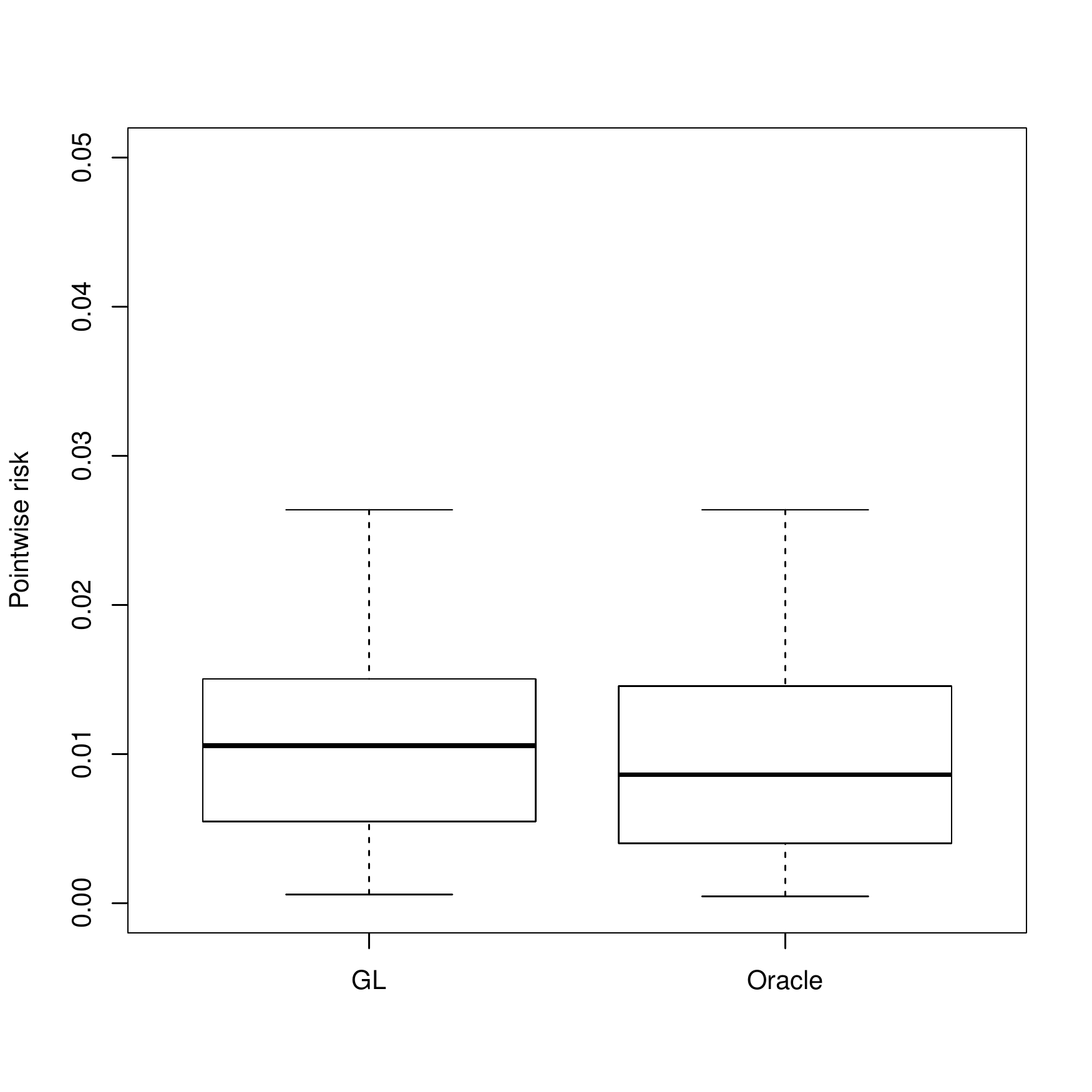} &
 \parinc{\widthboxplot1}{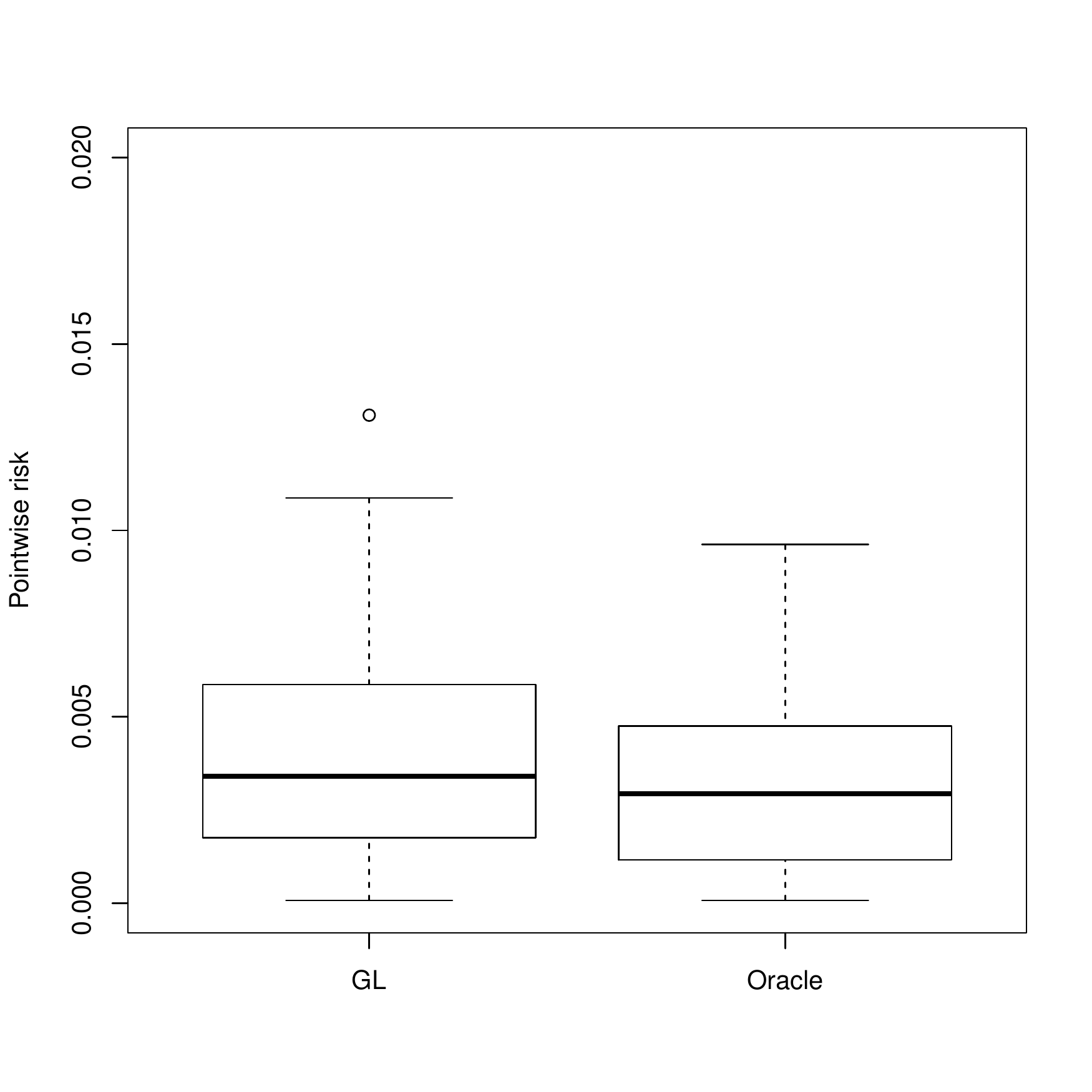} \\
$ \sigma_{g_L}=0.10$ & \parinc{\widthboxplot1}{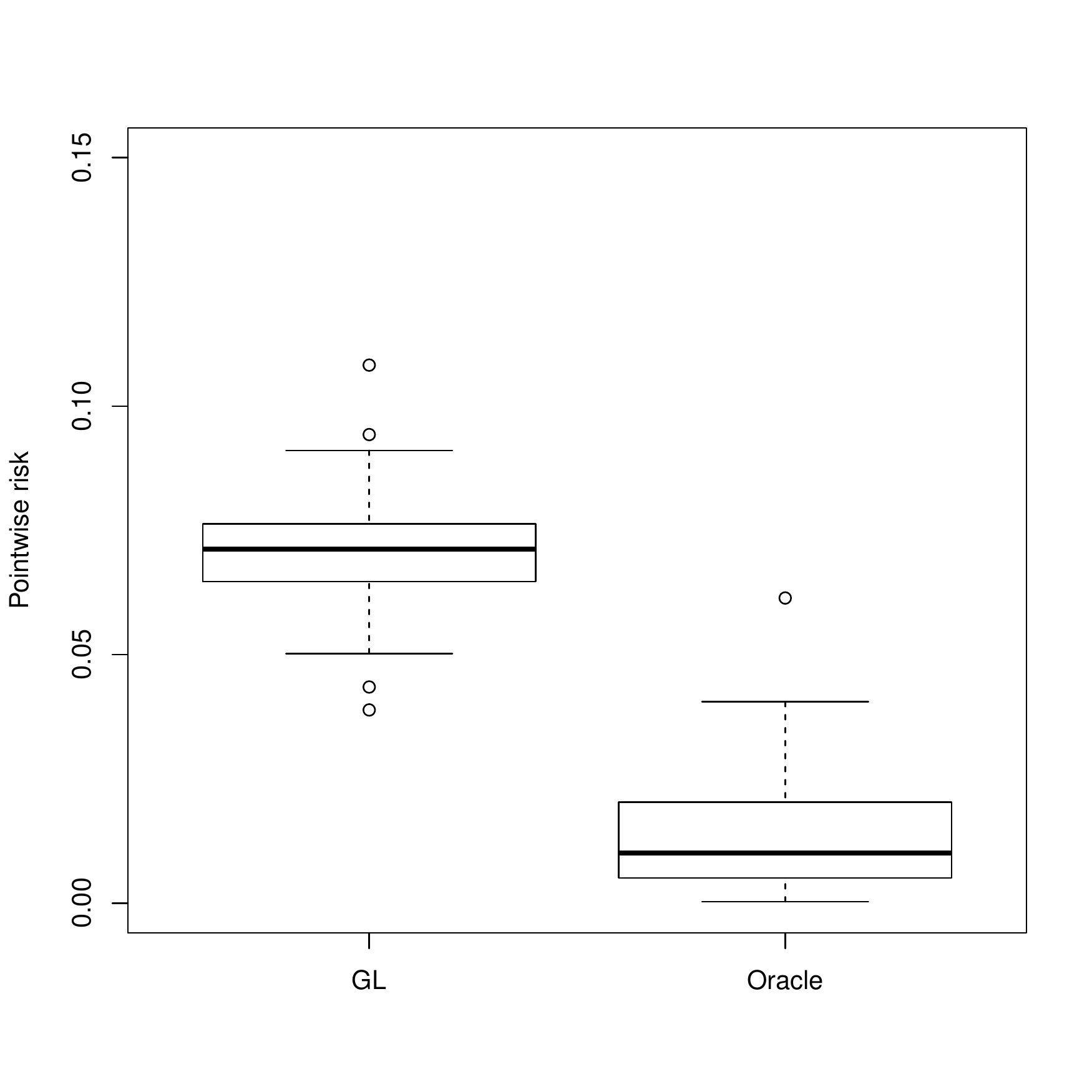} &
 \parinc{\widthboxplot1}{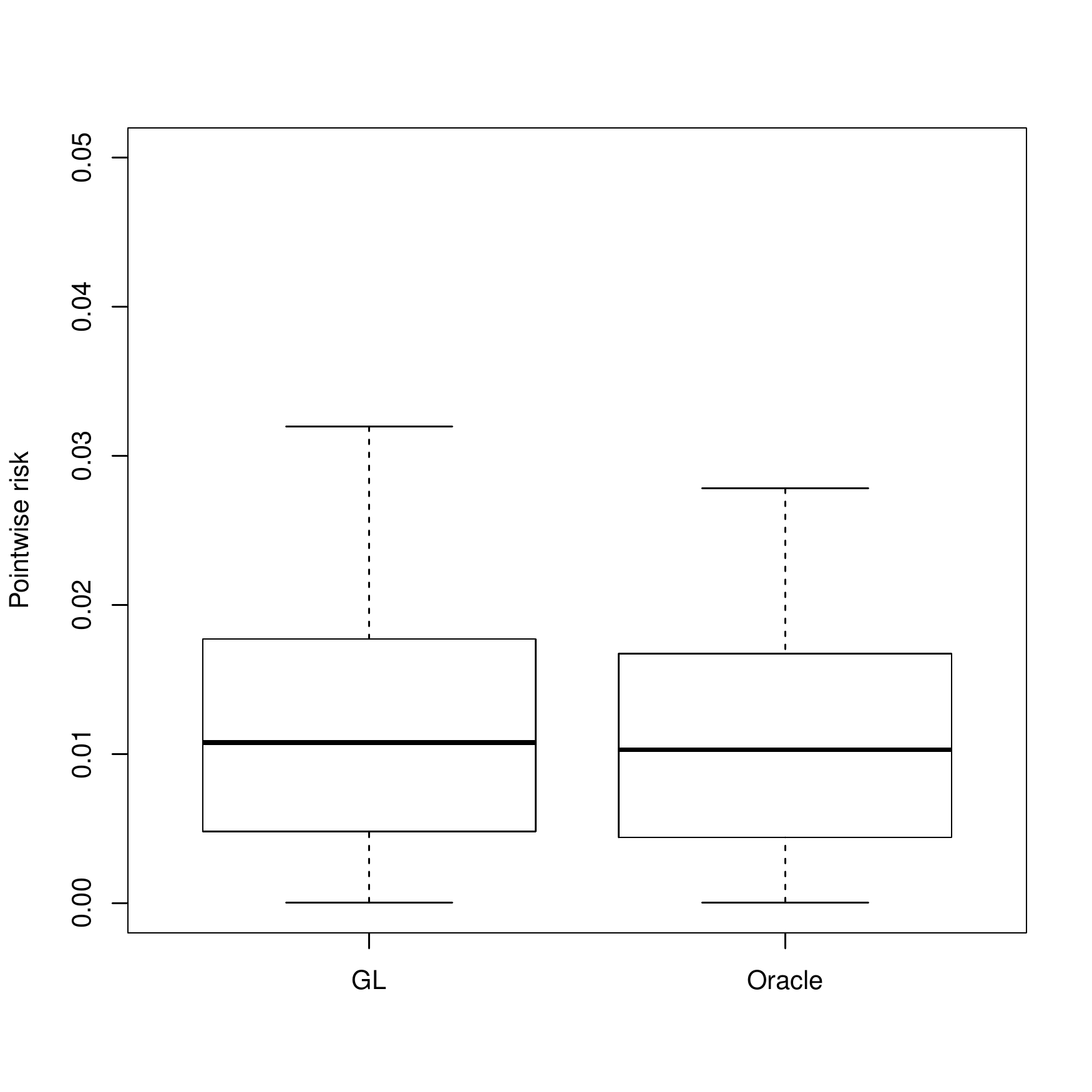}&
 \parinc{\widthboxplot1}{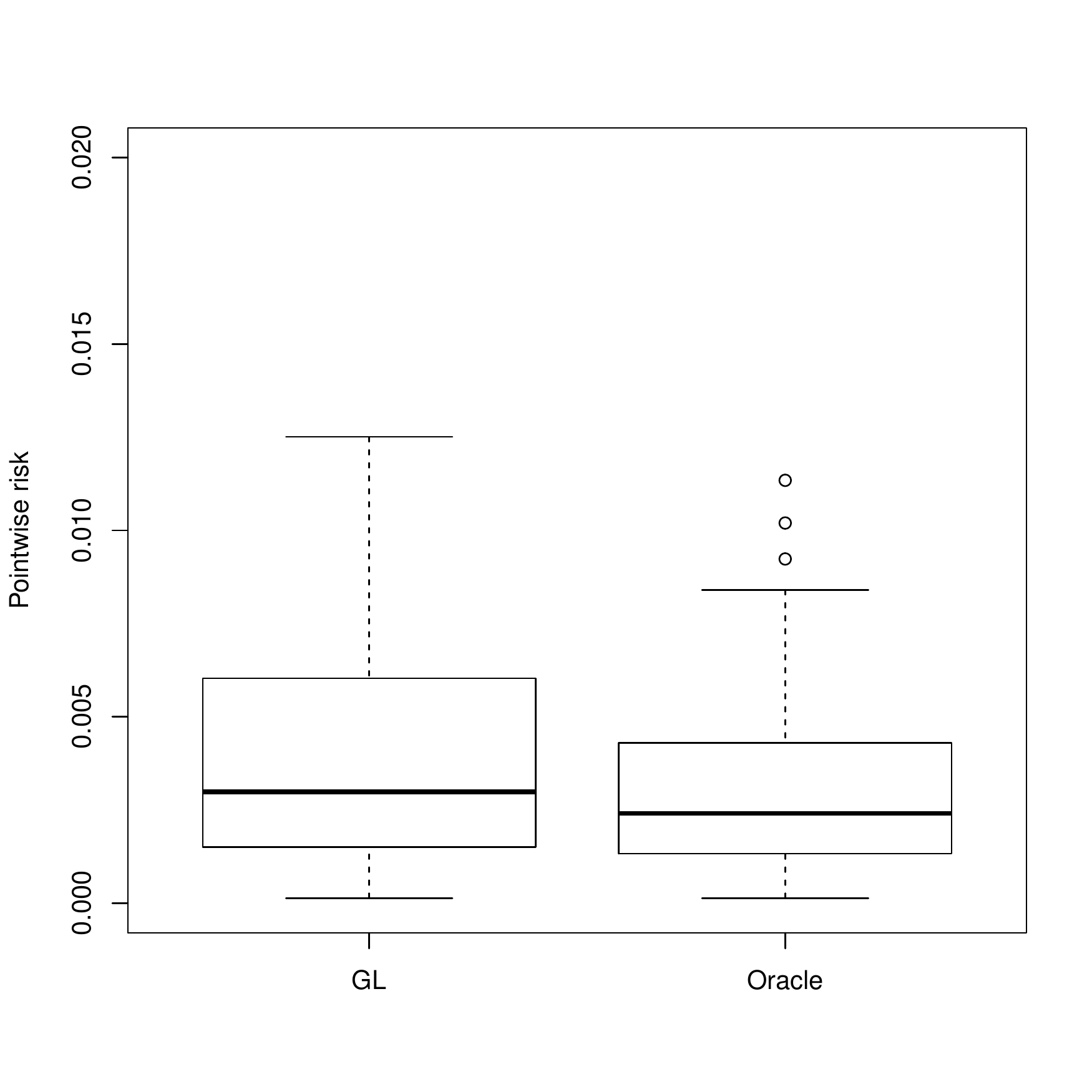} \\
\end{tabular}
\caption{Estimation of $p(x)$ at $x_0=0.90$ }
\label{boxplotx04}
\end{figure}

\begin{table}[!h]
\begin{center}
\begin{minipage}[t]{.49\linewidth}
\begin{tabular}{c|c c c }
 $\sigma_{g_L}$ &   \multicolumn{3}{c}{design of the $X_i$} \\    & $\mathcal{U}[0,1]$  &  $Beta(2,2)$  &  $Beta(0.5,2)$    \\ 
 \hline
 0.075 & 0.0144    &0.0204    &0.0071     \\
0.10  &  0.0156   & 0.0206    & 0.0072    \\ 
\end{tabular}\\
\end{minipage}
\begin{minipage}[t]{.49\linewidth}
\begin{tabular}{c|c c c }
 $\sigma_{g_L}$ &   \multicolumn{3}{c}{design of the $X_i$} \\    & $\mathcal{U}[0,1]$  &  $Beta(2,2)$  &  $Beta(0.5,2)$    \\ 
 \hline
 0.075 & 0.0212  & 0.0177    &  0.1012  \\
0.10  &  0.0192   & 0.0195   &   0.104  \\
\end{tabular}
\end{minipage}
\caption{MAE of $\hat m (x)$: on the left at $x_0=0.25$ and on the right $x_0=0.90$.  }
\label{risquem}
\end{center}
\end{table}

 Boxplots in Figure \ref{boxplotx02} and \ref{boxplotx04} summarize our numerical experiments. Theorem 1 gives an oracle inequality for the estimation of $p(x)$.  We compare the pointwise risk error of  $\hat{p}_{\hat j}(x)$ (computed with $100$ Monte Carlo repetitions) with the oracle risk one. The oracle  is $\hat p_{j_{oracle}}$ with the  index $j_{oracle}$ defined as follows:
 $$j_{oracle} :=\arg \min_{j \in J } |\hat p_j(x) -p(x)|.$$

  In Table  \ref{risquem}, we have computed the MAE (Mean Absolute Error) of $\hat m(x)$ over 100 Monte Carlo runs.

Our performances are close to those of the oracle (see Figure \ref{boxplotx02} and \ref{boxplotx04}) and are quite satisfying both at $x_0=0.25$ and $x_0= 0.90$. When going deeper into details, increasing the Laplace noise  parameter $\sigma_{g_L}$ deteriorates sligthly the performances. Hence it seems that our procedure is robust  to the noise in the covariates and accordingly to the deconvolution step. Concerning the role of the design density, when considering the $Beta(0.5,2)$ distribution, we expect the performances to be better near $0$ as the observations tend to concentrate near $0$ and to be bad close to $1$. Indeed, this phenomenon is confirmed by Table \ref{risquem}. And when comparing the $Beta(2,2)$ and $Beta(0.5,2)$ distributions, the performances are much better for the $Beta(0.5,2)$ at $x_0= 0.25$ whereas the $Beta(2,2)$ distribution yields better results at $x_0= 0.90$. This is what is expected as the two densities charge points near $0$ and $1$ differently. 

\section{Proofs}\label{secpreuve}

\subsection{ Proofs of theorems }\label{preuvetheoremes}

This section is devoted to the proofs of theorems. These proofs use some propositions and technical lemmas which are respectively in section \ref{preuveprop} and \ref{appendix}. 
In the sequel, $C$ is a constant which may vary from one line to another one.  
\subsubsection{Proof of Theorem \ref{oracle}}
\begin{proof}
We firstly recall the basic inequality $ (a_1+\cdots+a_p)^q\leq p^{q-1}(a_1^q+\cdots+a_p^q) $
for all $
a_1,\dots,a_p\in\R_+^p $, $p\in\mathbb N$ and $ q\geq1 $. For ease of exposition, we denote $
\hat p_{\hat
j}(x)=\hat p_{\hat j} $. So, we can show for any $ \eta\in\mathbb N^d $:
\begin{align*}
 \left|\hat p_{\hat j}-p(x)\right|&\leq \left|\hat p_{\hat j}-\hat p_{\hat
j\wedge
\eta}\right|+\left|\hat p_{\hat j\wedge \eta}-\hat p_{\eta}\right|+\left|\hat
p_{\eta}-p(x)\right|\\
&\leq \left|\hat p_{\eta\wedge \hat j}-\hat p_{\hat j}\right|-\Gamma_\gamma(\hat
j,\eta)+\Gamma_\gamma(\hat
j,\eta) +\left|\hat p_{\hat j\wedge \eta}-\hat p_{\eta}\right|-\Gamma_\gamma(\eta,\hat
j)+\Gamma_\gamma(\eta,\hat
j)+\left|\hat p_{\eta}-p(x)\right|\\
&\leq \left|\hat p_{\eta\wedge \hat j}-\hat p_{\hat j}\right|-\Gamma_\gamma(\hat j,\eta)+
\Gamma_\gamma(\eta,\hat j)+\left|\hat p_{\hat j\wedge \eta}-\hat p_{\eta}\right|
-\Gamma_\gamma(\eta,\hat j)+\Gamma_\gamma(\hat
j,\eta)+\left|\hat p_{\eta}-p(x)\right|\\
&\leq \left|\hat p_{\eta\wedge \hat j}-\hat p_{\hat j}\right|-\Gamma_\gamma(\hat
j,\eta)+
\Gamma_\gamma^*(\eta)+\left|\hat p_{\hat j\wedge \eta}-\hat p_{\eta}\right|
-\Gamma_\gamma(\eta,\hat j)+\Gamma_\gamma^*(\hat
j)+\left|\hat p_{\eta}-p(x)\right|\\
&\leq \hat R_\eta+\hat R_{\hat j}+\left|\hat p_{\eta}-p(x)\right|\\
&\leq \hat R_\eta+\hat R_{\hat j}+\left|\E[\hat p_{\eta}]-p(x)\right|+\left|\hat
p_{\eta}-\E[\hat p_{\eta}] \right|\\
&\leq \hat R_\eta+\hat R_{\hat j}+\left|\E[\hat p_{\eta}]-p(x)\right|+\left|\hat
p_{\eta}-\E[\hat p_{\eta}]\right|-\Gamma_\gamma(\eta)+\Gamma_\gamma(\eta)\\
&\leq \hat R_\eta+\hat R_{\hat j}+\left|\E[\hat p_{\eta}]-p(x)\right|+\sup_{j'}\Big\{\left|\hat
p_{j'}-\E[\hat p_{j'}]\right|-\Gamma_\gamma(j')\Big\}_++\Gamma_\gamma^*(\eta)
\end{align*}
By definition of $ \hat j $, we recall that $ \hat R_{\hat j}\leq\inf_\eta \hat R_\eta  $ and
$$
\hat R_\eta\leq \sup_{j,j'}\Big\{\left|\hat p_{j\wedge j'}-\E[\hat p_{j\wedge
j'}]\right|-\Gamma_\gamma(j\wedge
j')\Big\}_++\sup_{j'}\Big\{\left|\hat p_{j'}-\E[\hat
p_{j'}]\right|-\Gamma_\gamma(j')\Big\}_++\sup_{j'}\left|\E[\hat
p_{\eta\wedge j'}]-\E[\hat p_{j'}]\right|+\Gamma^*_\gamma(\eta).
$$
Hence
\begin{eqnarray*}
 \left|\hat p_{\hat j}-p(x)\right| &\leq& 2 \left [   \sup_{j,j'}\Big\{\left|\hat p_{j\wedge j'}-\E[\hat p_{j\wedge
j'}]\right|-\Gamma_\gamma(j\wedge
j')\Big\}_+     +\sup_{j'}\Big\{\left|\hat p_{j'}-\E[\hat
p_{j'}]\right|-\Gamma_\gamma(j')\Big\}_+  + \sup_{j'}\left|\E[\hat
p_{\eta\wedge j'}]-\E[\hat p_{j'}]\right| \right ]\\
&&+2\Gamma^*_\gamma(\eta) + \left|\E[\hat p_{\eta}]-p(x)\right|+\sup_{j'}\Big\{\left|\hat
p_{j'}-\E[\hat p_{j'}]\right|-\Gamma_\gamma(j')\Big\}_++\Gamma_\gamma^*(\eta)
\end{eqnarray*}
By definition of $ B(\eta)=\max\left(\sup_{j'}\left|\E\hat
p_{\eta\wedge j'}-\E\hat p_{j'}\right|,\left|\E\hat p_{\eta}-p(x)\right|\right) $, we get 

$$
 \left|\hat p_{\hat j}-p(x)\right|  \leq 2    \sup_{j,j'}\Big\{\left|\hat p_{j\wedge j'}-\E[\hat p_{j\wedge
j'}]\right|-\Gamma_\gamma(j\wedge
j')\Big\}_+   + 3 \sup_{j'}\Big\{\left|\hat p_{j'}-\E[\hat
p_{j'}]\right|-\Gamma_\gamma(j')\Big\}_+ +3 B(\eta) +3 \Gamma_\gamma^*(\eta)
$$
Consequently 
$$
\left|\hat p_{\hat j}-p(x)\right|^q\leq 3^{2q-1} \left(\big[B(\eta)+
\Gamma^*_\gamma(\eta)\big]^q+\sup_{j'}\Big\{\left|\hat p_{j'}-\E\hat
p_{j'}\right|-\Gamma_\gamma(j')\Big\}_+^q+\sup_{j,j'}\Big\{\left|\hat p_{j\wedge j'}-\E\hat
p_{j\wedge j'}\right|-\Gamma_\gamma(j\wedge
j')\Big\}_+^q\right).
$$
Using Proposition \ref{BorneMajorant}, we have
$$
\E \left |\hat p_{\hat j}-p(x)\right |^q \leq C \left (\E \left [\left (B(\eta)+
\Gamma^*_\gamma(\eta)\right )^q \right ] \right)+o(n^{-q}).
$$
Then, we get
$$
\E\left|\hat p_{\hat j}-p(x)\right|^q\leq R_1\left (\inf_{\eta}\E \left [\left (B(\eta)+
\Gamma^*_\gamma(\eta)\right )^q \right ]\right)+o(n^{-q}),
$$
where $R_1$ is a constant only depending on $q$.

\end{proof}

\subsubsection{Proof of Theorem \ref{th adaptive minimax p}}

\begin{proof}
The proof is a direct application of Theorem \ref{oracle} together with a standard
bias-variance trade-off. We first recall the assertion of this theorem:
$$
\E\left [ \left|\hat p_{\hat j}(x)-p(x)\right|^q \right ]\leq C\left (\inf_{\eta} \E \left [ \left (B(\eta)+
\Gamma^*_\gamma(\eta)\right )^q \right ] \right)+o(n^{-q}) .
$$
For the bias term, we use Proposition \ref{PropBiais} to get:
$$B(\eta)\leq C L \sum_{\el=1}^d
2^{-\eta_\el\beta_\el}, \textrm{for all}\;   \eta\in J. $$ 
Now let us focus on $ \E \left [ 
\Gamma^*_\gamma(\eta)^q \right ] $. We have

\begin{eqnarray*}   
   \E \left [ 
\Gamma_\gamma(\eta)^q \right ] &=  & \E \left [  \left ( \sqrt{\frac{2 \gamma (1+\e)\tilde\sigma_{\eta,\tilde\gamma}^2\log n}{n}}+\frac{c_\eta \gamma \log n}{n} \right)^q \right ]  \\ \nonumber
&\leq& 2^{q-1}\left(   \left ( {\frac{2 \gamma (1+\e)\log n}{n}}  \right )^{\frac q 2 } \E [ \tilde\sigma_{\eta,\tilde\gamma}^q ]+\left (\frac{c_\eta \gamma \log n}{n} \right)^{ q } \right )\\
&\leq&
C \left ( \left  (\frac{\log n}{n} \right)^{\frac q 2} 2^{\Seta(2\nu+1)\frac{q}{2}}+ \left (\frac{c_\eta \log n}{n} \right)^{q } \right),
\end{eqnarray*}
using Lemma \ref{Rosenthal}. 
But

$$c_\eta  =16\left(2\|m\|_\infty+ s\right)\|T_\eta\|_\infty \leq C 2^{\Seta(\nu+1)}, $$
using Lemma \ref{ordredegrandeursigmajTj}.
Hence 
$$
 \E \left [ 
\Gamma_\gamma(\eta)^q \right ]  \leq  C \left (  \left  (\frac{\log n}{n} \right)^{\frac q 2} 2^{\Seta(2\nu+1)\frac{q}{2}} + \left ( \frac{\log n}{n}\right )^q 2^{\Seta(\nu+1)q}   \right).
$$
We have 
\begin{eqnarray*}
\left(\frac{\log n}{n} \right)^{\frac q 2} 2^{\Seta(2\nu+1)\frac{q}{2}}  &\geq& \left(\frac{\log n}{n}\right)^q 2^{\Seta(\nu+1)q} {\Longleftrightarrow} 2^{\Seta} \leq \frac{n}{\log n},
\end{eqnarray*}
which is true since by (\ref{setJ}), $2^{\Seta} \leq \frac{n}{\log^2 n}$ .

This yields
$$\E [\Gamma^*_\gamma (\eta)^q ]\leq C
 \left ( {\frac{ 2^{\Seta(2\nu+1)} \log n}{n}}\right )^{\frac q 2}. $$

Eventually, we obtain the bound for the pointwise risk:
$$
\E\left|\hat p_{\hat j}(x)-p(x)\right|^q\leq C \left (\inf_{\eta}\left\{L\sum_{\el=1}^d 
2^{-\eta_\el\beta_\el}+\sqrt{\frac{ 2^{(2\nu+1)\Seta}\log(n)}{n}}\right\}^q  \right )+o(n^{-q}).
$$
Setting the gradient of the right hand side of the inequality above with respect to $\eta$  it turns out that the optimal $ \eta_{\el}$ is proportional to $\frac{2}{\log 2}\frac{\bar\beta}{\beta_{\el}(2\bar\beta+2\nu+1)}(\log L + \frac 1 2 \log(\frac{n}{\log(n)}))
$, which leads  for $n$ large enough to
\begin{align*}
\E\left|\hat p_{\hat j}(x)-p(x)\right|^q
&\leq L^{\frac{q(2\nu+1)}{2 \bar \beta +2\nu +1}} R_2
\left(\frac{\log(n)}{n}\right)^{ \frac{\bar\beta q}{2\bar\beta+2\nu+1}},
\end{align*}
with $R_2$ a constant depending on $\gamma,q,\eps, \tilde \gamma, \| m\|_{\infty}, s,  \| f_X\|_{\infty},\varphi, c_g, \mathcal{C}_g, \vec{\beta}$.
The proof of Theorem  \ref{th adaptive minimax p} is completed.

\end{proof}

\medskip

\subsubsection{Proof of Theorem \ref{th adaptive minimax m}}

\begin{proof}
We recall that $m(x)=\frac{p(x)}{f_X(x)}$ and $ \hat m(x)=\frac{\hat{p}_{\hat j}(x)}{ \hat f_X(x)\vee n^{-1/2}}$.
We now state the main properties of the adaptive estimate $ \hat f_X $ showed by \cite{ComteLacour} (Theorem 2): for all $q\geq1$, all $
\vec\beta\in(0,1]^d $, all $L>0$  and $n$ large enough, it holds
\begin{equation}\label{eq result Lacour}
\mathbb{P}\left(E_1\right):=\mathbb{P}\left(|\hat f_X(x)-f_X(x)|\geq C \phi_{n}(\vec\beta)\right)\leq   n^{-2q},
\end{equation}
and
\begin{equation}\label{eq supbound Lacour}
\mathbb{P}\left(|\hat f_X(x)-f_X(x)|\leq C  n\right)=1,
\end{equation}
where $\phi_{n}(\vec\beta):=\left({\log(n)}/{n}\right)^{\bar\beta/(2\bar\beta+2\nu+1)}$. Although the construction of the estimate $\hat f_X(x)$ depends on $q$, we remove the dependency for ease of exposition (see \cite{ComteLacour} Section 3.4 for further details).
From (\ref{eq result Lacour}), we easily deduce, since $f_X(x)\geq C_1>0$, for $ n $ large enough that 
\begin{equation}\label{eq minimum estimate}
\mathbb{P}\left(E_2\right):=\mathbb{P}\left(\hat f_X(x)<\frac{C_1}{2}\right)\leq  n^{-2q}.
\end{equation}
We now start the proof of the theorem. We have together with \eqref{eq supbound Lacour}
\begin{align*}
\left|\hat m(x)-m(x)\right|=\left |\frac{\hat p_{\hat j}(x)}{\hat f_X(x)\vee n^{-1/2}}-\frac{p(x)}{f_X(x)}\right| 
&\leq \left | \frac{\hat{p}_{\hat j}(x)}{\hat f_X(x)\vee n^{-1/2}} -\frac{p(x)}{\hat f_X(x)\vee n^{-1/2}} \right |
+ \left | \frac{p(x)}{\hat f_X(x)\vee n^{-1/2}} -\frac{p(x)}{f_X(x)}\right |\\
&\leq \left | \frac{\hat{p}_{\hat j}(x)-p(x)}{\hat f_X(x)\vee n^{-1/2}} \right|
+ \|m\|_\infty\|f_X\|_\infty\left | \frac{(\hat f_X(x)\vee n^{-1/2})-f_X(x)}{f_X(x)(\hat f_X(x)\vee n^{-1/2})}\right |\\
&:=\mathcal A_1+\|m\|_\infty\|f_X\|_\infty\mathcal A_2.
\end{align*}
\paragraph{Control of $\E[\mathcal A_1^q]$.}
Using Cauchy-Schwarz inequality and the inequality $ \hat f_X(x)\vee n^{-1/2}\geq n^{-1/2} $, we obtain for $ n $ large enough
\begin{align*}
 \E[\mathcal A_1^q ]=\E [\mathcal A_1^q\1_{E_2^c} ]+\E [\mathcal A_1^q\1_{E_2}] &\leq \E [\mathcal A_1^q\1_{E_2^c} ]+\sqrt{\E[\mathcal A_1^{2q}] }\sqrt{\P (E_2)}\\
&\leq C \E \left [\left|\hat p_{\hat j}(x)-p(x)\right|^q \right ]+n^{q/2}\sqrt{\E \left [\left|\hat p_{\hat j}(x)-p(x)\right|^{2q} \right ]}\sqrt{\P (E_2)}.\\ 
 \end{align*}
Then, using Theorem \ref{th adaptive minimax p} and \eqref{eq minimum estimate}, we finally have $ \E[\mathcal A_1^q]\leq C  \phi_{n}^q(\vec\beta)$.

\paragraph{Control of $\E[\mathcal A_2^q]$.}
Using \eqref{eq supbound Lacour} and the inequality $ \hat f_X(x)\vee n^{-1/2}\geq n^{-1/2} $, it holds for $ n $ large enough
$$
\E [\mathcal A_2^q ]\leq\E [\mathcal A_2^q\1_{E_1^c\cap E_2^c} ]+\E [\mathcal A_2^q(\1_{E_1}+\1_{E_2}) ] \leq\E [\mathcal A_2^q\1_{E_1^c\cap E_2^c} ]
+Cn^{3q/2}(\P({E_1})+\P({E_2})).
$$
Then, using the definition of $\mathcal A_2$, \eqref{eq result Lacour} and \eqref{eq minimum estimate}, we obtain $ \E[\mathcal A_2^q] \leq C  \phi_{n}^q(\vec\beta)$.

\medskip
\noindent Eventually, by definitions of $\mathcal A_1$ and $\mathcal A_2$, the proof is completed and
\begin{align*}
 \E [\left|\hat m(x)-m(x)\right|^q] \leq C ( \E [\mathcal A_1^q]+\E [\mathcal A_2^q ])\leq L^{\frac{q(2\nu+1)}{2 \bar \beta +2\nu +1}} R_3  \left(\frac{\log(n)}{n}\right)^{q\bar\beta/(2\bar\beta+2\nu+1)},
 \end{align*}
 where $R_3$ is a constant depending on $\gamma,q,\eps, \tilde \gamma, \| m\|_{\infty}, s,  \| f_X\|_{\infty},\varphi, c_g, \mathcal{C}_g, \vec{\beta}$. This completes the proof of Theorem \ref{th adaptive minimax m}.

\end{proof}

\subsection{Statements and proofs of auxiliary results}
This section is devoted to statements and proofs of auxiliary  results used in section  \ref{preuvetheoremes}

\subsubsection{ Statements and proofs of propositions } \label{preuveprop}

Let us start with Proposition \ref{Th:conc} which states a concentration inequality of $\hat p_j$ around $p_j$. 
\begin{Prop}\label{Th:conc}
Let $j$ be fixed. For any $u>0$,
\begin{equation}\label{concentrationp_jsigma_j}
\P\left(|\hat p_j(x)-p_j(x)|\geq \sqrt{\frac{2\sigma_j^2u}{n}}+\frac{c_ju}{n}\right)\leq 2e^{-u},
\end{equation}
where $$\sigma_j^2=\var(Y_1T_j(W_1)).$$
For any $\tilde\gamma >1$ we have for any $\tilde\e>0$ that there exists $R_4$ only depending on $\tilde\gamma$ and $\tilde\e$ such that
$$\P(\sigma^2_j\geq (1+\tilde\e)\tilde\sigma^2_{j,\tilde\gamma})\leq R_4 n^{-\tilde\gamma},$$
$\tilde \sigma_{j,\tilde\gamma}^2$ being defined in (\ref{sigmatilde}).

\end{Prop}

\begin{proof} 

First, note that
 $$\hat{p}_j(x)=\sum_k \hat{p}_{jk} \varphi_{jk}(x)=\frac 1 n \sum_{\el=1}^n Y_\el \sum_k (\mathcal{D}_j \varphi)_{j,k}(W_\el) \varphi_{jk}(x)=\frac 1 n \sum_{\el=1}^n U_j(Y_\el,W_\el).$$

To prove Proposition~\ref{Th:conc}, we apply the Bernstein inequality to the variables $U_j(Y_\el,W_\el)-\E[U_j(Y_\el,W_\el)]$ that are independent.
Since,
$$U_j(Y_\el,W_\el)=Y_\el T_j(W_\el),$$
and 
$$\E\left[\e_\el T_j(W_\el)\right]=0,$$
we have for any $q\geq 2$,
\begin{equation}\label{defA_q}
A_q:=\sum_{\el=1}^n\E[|U_j(Y_\el,W_\el)-\E[U_j(Y_\el,W_\el)]|^q]=\sum_{\el=1}^n\E\left[|m(X_\el)T_j(W_\el)+\e_\el T_j(W_\el)-\E[m(X_\el)T_j(W_\el)]|^q\right].
\end{equation}
With $q=2$,
\begin{eqnarray*}
A_2&=&\sum_{\el=1}^n\E[|U_j(Y_\el,W_\el)-\E[U_j(Y_\el,W_\el)]|^2]\\
&=&n\var(Y_1T_j(W_1))\\
&=&n\E[(m(X_1)T_j(W_1)+\e_1T_j(W_1)-\E[m(X_1)T_j(W_1)])^2]\\
&=&n\E[\e_1^2T_j^2(W_1)]+n\var(m(X_1)T_j(W_1))\\
&=&n\left(\sigma^2_\e\E[T_j^2(W_1)]+\var(m(X_1)T_j(W_1))\right).
\end{eqnarray*}
Now, for any $q\geq 3$, with $Z\sim{\mathcal N}(0,1)$,
\begin{eqnarray*}
A_q&\leq&n2^{q-1}\left(\E[|m(X_1)T_j(W_1)-\E[m(X_1)T_j(W_1)]|^q]+\E[|\e_1T_j(W_1)|^q]\right)\\
&\leq&n2^{q-1}\left(\E[|m(X_1)T_j(W_1)-\E[m(X_1)T_j(W_1)]|^q]+s^q\E[|Z|^q]\E[|T_j(W_1)|^q]\right)\\
&\leq&n2^{q-1}\left(\E[|m(X_1)T_jW_1)-\E[m(X_1)T_j(W_1)]|^q]+s^q\E[|Z|^q]\E[T_j^2(W_1)]\|T_j\|_\infty^{q-2}\right).
\end{eqnarray*}
Furthermore,
\begin{eqnarray*}
\E[|m(X_1)T_j(W_1)-\E[m(X_1)T_j(W_1)]|^q]&\leq&\E[(m(X_1)T_j(W_1)-\E[m(X_1)T_j(W_1)])^2]\times (2\|m\|_\infty\|T_j\|_\infty)^{q-2}\\
&=&\var(m(X_1)T_j(W_1))\times (2\|m\|_\infty\|T_j\|_\infty)^{q-2}.
\end{eqnarray*}
Finally,
\begin{eqnarray*}
A_q&\leq&n2^{q-1}\|T_j\|_\infty^{q-2}\left(\var(m(X_1)T_j(W_1))\times (2\|m\|_\infty)^{q-2}+s^q\E[|Z|^q]\E[T_j^2(W_1)]\right)\\
&\leq&n2^{q-1}\|T_j\|_\infty^{q-2}\E[|Z|^q]\left(\var(m(X_1)T_j(W_1))\times (2\|m\|_\infty)^{q-2}+s^q\E[T_j^2(W_1)]\right)\\
&\leq&n2^{q-1}\|T_j\|_\infty^{q-2}\E[|Z|^q]\left(\var(m(X_1)T_j(W_1))+s^2\E[T_j^2(W_1)]\right)\times \left((2\|m\|_\infty)^{q-2}+s^{q-2}\right)\\
&\leq&2^{q-1}\|T_j\|_\infty^{q-2}\E[|Z|^q]\times A_2\times \left(2\|m\|_\infty+s\right)^{q-2}.
\end{eqnarray*}
Besides we have (see page 23 in \cite{Patel})  denoting $\Gamma$  the Gamma function
\begin{equation} \label{MajorationMomentGaussienne}
 \E[|Z|^q]=\frac{2^{q/2}}{\sqrt{\pi}} \Gamma \left (\frac{q+1}{2} \right)  \leq 2^{q/2}2^{-1/2} q! \leq  2^{(q-1)/2}q! ,
 \end{equation}
 as $\frac{1}{\sqrt{\pi}}\leq \frac{ 1}{ \sqrt{2}} $ and $\Gamma(\frac{q+1}{2})\leq \Gamma(q+1)=q!$.
 So, for $q\geq 3$,
\begin{eqnarray*}
A_q & \leq & 2^{q-1}\|T_j\|_\infty^{q-2}   2^{(q-1)/2}q!   \times A_2\times \left(2\|m\|_\infty+s\right)^{q-2} \\
&\leq &\frac{q!}{2}\times A_2\times \left ( 2^{\frac{3q-1}{2(q-2)}}\|T_j\|_\infty\left(2\|m\|_\infty+s\right) \right  )^{q-2},
\end{eqnarray*}
The function $ \frac{3q-1}{2(q-2)}$ is decreasing in $q$. Hence for any $q\geq 3$, $2^{\frac{3q-1}{2(q-2)}}\leq 16$.\\
Thus
\begin{equation}\label{controlA_q}
 A_q\leq \frac{q!}{2}\times A_2\times  {c_j }^{q-2},
 \end{equation}
with
$$c_j:=  16\|T_j\|_\infty\left(2\|m\|_\infty+s\right).$$
We can now apply Proposition 2.9 of Massart (2007). We denote $f_W$ the density of the $W_\el$'s. We have
\begin{eqnarray*}
\E[T_j^2(W_1)]&=&\int T_j^2(w)f_W(w)dw\\
&\leq&\|f_X\|_\infty \|T_j\|_2^2,
\end{eqnarray*}
since  the density $f_W$ is the convolution of $f_X$ and $g$, $\|f_W\|_\infty=\| f_X\star g\|_\infty\leq \|f_X\|_\infty$.
 We have 
\begin{eqnarray*}
\var(m(X_1)T_j(W_1))&\leq &\E[m^2(X_1)T_j^2(W_1)]\\
&\leq&\|m\|_\infty^2\int T_j^2(w)f_W(w)dw\\
&\leq&\|m\|_\infty^2\|f_X\|_\infty\|T_j\|_2^2.
\end{eqnarray*}

Therefore, with
\begin{equation}\label{sigma}
\sigma_j^2=\frac{A_2}{n}=\var(Y_1T_j(W_1)),
\end{equation}
\begin{eqnarray}
\sigma_j^2&=& \sigma^2_\e\E[T_j^2(W_1)]+\var(m(X_1)T_j(W_1)) \label{eqsigma_j}  \\
&\leq&\sigma^2_\e\|f_X\|_\infty \|T_j\|_2^2+ \|m\|_\infty^2\|f_X\|_\infty\|T_j\|_2^2 \nonumber \\
&\leq&\|f_X\|_\infty \|T_j\|_2^2(\sigma^2_\e+\|m\|_\infty^2). \nonumber \\
\end{eqnarray}
We conclude that for any $u>0$,
\begin{equation}\label{conc1}
\P\left(|\hat p_j(x)-p_j(x)|\geq \sqrt{\frac{2\sigma_j^2u}{n}}+\frac{c_ju}{n}\right)\leq 2e^{-u}.
\end{equation}
Now, we can write 
\begin{eqnarray*}
\hat\sigma^2_j&=&\frac{1}{n(n-1)}\sum_{\el=2}^n\sum_{\vv=1}^{\el-1}(U_j(Y_\el,W_\el)-U_j(Y_\vv,W_\vv))^2\\
&=&\frac{1}{n(n-1)}\sum_{\el=2}^n\sum_{\vv=1}^{\el-1}(U_j(Y_\el,W_\el)-\E[U_j(Y_\el,W_\el)]-U_j(Y_{\vv},W_\vv)+ \E[U_j(Y_\vv,W_\vv)])^2\\
&=&s_j^2-\frac{2}{n(n-1)}\xi_j,
\end{eqnarray*} 
with 
\begin{eqnarray*}
s_j^2&:=&\frac{1}{n(n-1)}\sum_{\el=2}^n\sum_{\vv=1}^{\el-1}(U_j(Y_\el,W_\el)-\E[U_j(Y_\el,W_\el)])^2+(U_j(Y_\vv,W_\vv)- \E[U_j(Y_\vv,W_\vv)])^2 \\
&=&\frac{1}{n}\sum_{\el=1}^n(U_j(Y_\el,W_\el)-\E[U_j(Y_\el,W_\el)])^2
\end{eqnarray*}
and 
$$\xi_j:=\sum_{\el=2}^n\sum_{\vv=1}^{\el-1}(U_j(Y_\el,W_\el)-\E[U_j(Y_\el,W_\el)])\times(U_j(Y_\vv,W_\vv)-\E[U_j(Y_\vv,W_\vv)]).$$
In the sequel, we denote for any $\tilde{\gamma}>0$,
$$\Omega_n(\tilde{\gamma})=\left\{\max_{1\leq \el \leq n}|\e_\el|\leq s\sqrt{2\tilde{\gamma}\log n}\right\}.$$
We have that
\begin{equation} \label{ProbaComplementaireOmega_n}
\P(\Omega_n(\tilde{\gamma})^c) \leq n^{1-\tilde \gamma}.
\end{equation}

Note that on $\Omega_n(\tilde{\gamma})$,
$$\|U_j(\cdot,\cdot)\|_\infty\leq C_j,$$
we recall that
$$C_j=(\|m\|_\infty+s\sqrt{2\tilde{\gamma}\log n})\|T_j\|_\infty.$$
\begin{lemma}\label{conc-sj}
For any $\tilde{\gamma}>1$ and any $u>0$, there exists a sequence $e_{n,j} >0$ such that $\limsup_j e_{n,j}=0$ and 
$$\P\left(\left.\sigma_j^2\geq s_j^2+2C_j\sigma_j\sqrt{\frac{2u(1+e_{n,j})}{n}}+\frac{\sigma_j^2 u}{3n}\right|\Omega_n(\tilde{\gamma})\right)\leq e^{-u}.$$
\end{lemma}
\begin{proof}

We denote
$$\P_{\Omega_n(\tilde{\gamma})}(\cdot)=\P\left(\cdot |\Omega_n(\tilde{\gamma})\right), \quad \E_{\Omega_n(\tilde{\gamma})}(\cdot)=\E\left(\cdot |\Omega_n(\tilde{\gamma})\right).$$

Note that conditionally to $\Omega_n(\tilde{\gamma})$ the variables $U_j(Y_1,W_1),\ldots,U_j(Y_n,W_n)$ are independent. So, we can apply the classical Bernstein inequality to the variables
$$V_\el:=\frac{\sigma_j^2-(U_j(Y_\el,W_\el)-\E[U_j(Y_\el,W_\el)])^2}{n}\leq \frac{\sigma_j^2}{n}.$$

Furthermore, as 
\begin{eqnarray}
\E_{\Omega_n(\tilde{\gamma})}[U_j(Y_1,W_1)]&=&\E[m(X_1)T_j(W_1)|\Omega_n(\tilde{\gamma})]+\E[\eps_1T_j(W_1)|\Omega_n(\tilde{\gamma})] \nonumber\\
&=&\E[m(X_1)T_j(W_1)]\nonumber\\
&=&\E[U_j(Y_1,W_1)] \label{esperancecondi}
\end{eqnarray}
we get
\begin{eqnarray*}
\sum_{\el=1}^n\E_{\Omega_n(\tilde{\gamma})}[V_\el^2]&=&\frac{\E_{\Omega_n(\tilde{\gamma})}\left[\left(\sigma_j^2-\left(U_j(Y_1,W_1)-\E[U_j(Y_1,W_1)]\right)^2\right)^2\right]}{n}\\
&=&\frac{\sigma_j^4+\E_{\Omega_n(\tilde{\gamma})}\left[\left(U_j(Y_1,W_1)-\E[U_j(Y_1,W_1)]\right)^4\right]-2\sigma_j^2\E_{\Omega_n(\tilde{\gamma})}\left[\left(U_j(Y_1,W_1)-\E[U_j(Y_1,W_1)]\right)^2\right]}{n}\\
&\leq&\frac{\sigma_j^4+(4C_j^2-2\sigma_j^2)\E_{\Omega_n(\tilde{\gamma})}\left[\left(U_j(Y_1,W_1)-\E[U_j(Y_1,W_1)]\right)^2\right] }{n}.\\
\end{eqnarray*}

We shall find an upperbound for $\E_{\Omega_n(\tilde{\gamma})}\left[(U_j(Y_1,W_1)-\E[U_j(Y_1,W_1)])^2\right]$: 
\begin{eqnarray*}
\E_{\Omega_n(\tilde{\gamma})}\left[(U_j(Y_1,W_1)-\E[U_j(Y_1,W_1)])^2\right]&=&\var(m(X_1)T_j(W_1))+\E[\eps_1^2T_j^2(W_1)|\Omega_n(\tilde{\gamma})]\\
&=&\var(m(X_1)T_j(W_1))+\E[T_j^2(W_1)]\frac{\E[\eps_1^2\mathds{1}_{\Omega_n(\tilde{\gamma})}]}{\P(\Omega_n(\tilde{\gamma}))} \\
&\leq& \var(m(X_1)T_j(W_1))+\E[T_j^2(W_1)] \frac{s^2}{ \P(\Omega_n(\tilde{\gamma})) } \\
&\leq & \var(m(X_1)T_j(W_1))+\E[T_j^2(W_1)] \frac{s^2}{1-n^{1-\tilde \gamma }} \\
&=& \var(m(X_1)T_j(W_1))+ \E[T_j^2(W_1)] s^2(1+\tilde{ e}_n),
\end{eqnarray*}
where $\tilde{e}_n=n^{1-\tilde \gamma}+ o( n^{1-\tilde \gamma}) $. Using (\ref{eqsigma_j}) we have
\begin{equation}\label{esperanceconditU_j}
\E_{\Omega_n(\tilde{\gamma})}\left[(U_j(Y_1,W_1)-\E[U_j(Y_1,W_1)])^2\right] \leq (1+e_{n,j})\sigma_j^2,
\end{equation}
where $(e_{n,j})$ is a sequence such that $\limsup_j e_{n,j}= 0$.

Now let us find a lower bound for $\E_{\Omega_n(\tilde{\gamma})}\left[(U_j(Y_1,W_1)-\E[U_j(Y_1,W_1)])^2\right]$ :
\begin{eqnarray*}
\E_{\Omega_n(\tilde{\gamma})}\left[(U_j(Y_1,W_1)-\E[U_j(Y_1,W_1)])^2\right]&=& \var(m(X_1)T_j(W_1))+\E[T_j^2(W_1)]\frac{\E[\eps_1^2\mathds{1}_{\Omega_n(\tilde{\gamma})}]}{\P(\Omega_n(\tilde{\gamma}))} \\
&\geq& \var(m(X_1)T_j(W_1))+\E[T_j^2(W_1)]  \E[\eps_1^2\mathds{1}_{\Omega_n(\tilde{\gamma})}] \\
&=&  \var(m(X_1)T_j(W_1))+\E[T_j^2(W_1)]  \E[\eps_1^2(1-\mathds{1}_{\Omega_n^c(\tilde{\gamma})})]   \\
&= &  \sigma_j^2 -\E[T_j^2(W_1)]  \E[\eps_1^2\mathds{1}_{\Omega_n^c(\tilde{\gamma})}] .  
\end{eqnarray*}
Now using  Cauchy Scharwz, (\ref{MajorationMomentGaussienne}) and (\ref{ProbaComplementaireOmega_n}) we have
\begin{eqnarray}\label{infenj}
\E_{\Omega_n(\tilde{\gamma})}\left[(U_j(Y_1,W_1)-\E[U_j(Y_1,W_1)])^2\right]&\geq &  \sigma_j^2 - \E[T_j^2(W_1)] ( \E[\eps_1^4])^{\frac 1 2}( \P(\Omega_n^c(\tilde{\gamma})))^{\frac 1 2 }\nonumber\\
&\geq &  \sigma_j^2 -C s^2 \E[T_j^2(W_1)]n^{\frac{1-\tilde \gamma}{2}}\nonumber\\
&=& \sigma_j^2(1+\tilde{e}_{n,j}),
\end{eqnarray}
where $ (\tilde{e}_{n,j})$ is a sequence such that $\limsup_j  \tilde{e}_{n,j} = 0$.

Finally, using the bounds we just got for $\E_{\Omega_n(\tilde{\gamma})}\left[(U_j(Y_1,W_1)-\E[U_j(Y_1,W_1)])^2\right]$ yields
\begin{eqnarray*}
\sum_{\el=1}^n\E_{\Omega_n(\tilde{\gamma})}[V_\el ^2]&\leq& \frac{\sigma_j^4+ 4C_j^2\sigma_j^2(1+e_{n,j})-2\sigma_j^4(1+\tilde{e}_{n,j}) }{n} \\
&\leq & \frac{4C_j^2\sigma_j^2(1+e_{n,j} )-\sigma_j^4(1+2 \tilde{e}_{n,j})}{n} \\
&\leq&  \frac{4C_j^2\sigma_j^2(1+e_{n,j} )}{n}.
\end{eqnarray*}

We obtain the claimed result.
\end{proof}
Now, we deal with $\xi_j.$
\begin{lemma}\label{conc-xi}
There exists an absolute constant $c>0$ such that for any $u>1$,
 $$\P\left(\left. \xi_j\geq c(n\sigma_j^2u+C_j^2u^2)\right|\Omega_n(\tilde{\gamma})\right)\leq 3 e^{-u}.$$
\end{lemma} 
\begin{proof}  Note that conditionally to $\Omega_n(\tilde{\gamma})$, the vectors $(Y_\el,W_\el)_{1\leq \el \leq n}$ are independent. We remind that by  (\ref{esperancecondi}), (\ref{esperanceconditU_j}) and \eqref{infenj} we have
\begin{equation}
\E_{\Omega_n(\tilde{\gamma})}[U_j(Y_1,W_1)]=\E[U_j(Y_1,W_1)]
\end{equation}
and
$$
\E_{\Omega_n(\tilde{\gamma})}\left[(U_j(Y_1,W_1)-\E[U_j(Y_1,W_1)])^2\right]=(1+ e_{n,j})\sigma_j^2.
$$
The $\xi_j$  can be written as
$$\xi_j=\sum_{\el=2}^n\sum_{\vv=1}^{\el-1}g_j(Y_\el,W_\el, Y_\vv,W_\vv),$$
with $$g_j(y,w,y',w')=(U_j(y,w)-\E[U_j(Y_1,W_1)]))\times(U_j(y',w')-\E[U_j(Y_1,W_1)]).$$ Previous computations show that conditions (2.3) and (2.4) of Houdr\'e and Reynaud-Bouret (2005) are satisfied. So that we are able to apply Theorem 3.1 of Houdr\'e and Reynaud-Bouret (2005): there exist absolute constants $c_1$, $c_2$, $c_3$ and $c_4$ such that for any $u>0$,
$$\P_{\Omega_n(\tilde{\gamma})}\left( \xi_j\geq c_1C\sqrt{u}+c_2Du+c_3Bu^{3/2}+c_4Au^2\right)\leq 3 e^{-u},$$
where $A,$ $B$, $C$, and $D$ are defined and controlled as follows.
We have:
$$A=\|g_j\|_\infty\leq 4C_j^2.$$
$$C^2=\sum_{\el=2}^n\sum_{\vv=1}^{\el-1}\E_{\Omega_n(\tilde{\gamma})} [g_j^2(Y_\el,W_\el, Y_\vv,W_\vv)]=\frac{n(n-1)}{2}\sigma_j^4(1+ e_{n,j})^2.$$
Let $${\mathcal A}=\left\{(a_\el)_\el, (b_\vv)_\vv: \quad  \E_{\Omega_n(\tilde{\gamma})}\left[\sum_{\el=2}^na_\el^2(Y_\el,W_\el)\right]\leq 1, \ \E_{\Omega_n(\tilde{\gamma})}\left[\sum_{\el=1}^{n-1}b_\el^2(Y_\el,W_\el)\right]\leq 1\right\}.$$
We have:
\begin{eqnarray*}
D&=&\sup_{(a_\el)_\el, (b_\vv)_\vv \in{\mathcal A}}\E_{\Omega_n(\tilde{\gamma})}\left[\sum_{\el=2}^n\sum_{\vv=1}^{\el-1}g_j(Y_\el,W_\el, Y_\vv,W_\vv)a_\el(Y_\el,W_\el)b_\vv(Y_\vv,W_\vv)\right]\\
&=&\sup_{(a_\el)_\el, (b_\vv)_\vv \in{\mathcal A}}\left[\sum_{\el=2}^n\sum_{\vv=1}^{\el-1}\E_{\Omega_n(\tilde{\gamma})}\left[(U_j(Y_\el,W_\el)-[U_j(Y_\el,W_\el)]))a_\el(Y_\el,W_\el)\right]\right.\\
&&\hspace{3cm}\times\left.\E_{\Omega_n(\tilde{\gamma})}\left[(U_j(Y_\vv,W_\vv)-\E[U_j(Y_\vv,W_\vv)]))b_\vv(Y_\vv,W_\vv)\right]\right]\\
&\leq&\sup_{(a_\el)_\el, (b_\vv)_\vv \in{\mathcal A}}\sum_{\el=2}^n\sum_{\vv=1}^{\el-1}\sigma_j^2(1+e_{n,j})\sqrt{\E_{\Omega_n(\tilde{\gamma})}[a_\el^2(Y_\el,W_\el)]\E_{\Omega_n(\tilde{\gamma})}[b_\vv^2(Y_\vv,W_\vv)]}\\
&\leq&\sigma_j^2(1+e_{n,j})\sup_{(a_\el)_\el, (b_\vv)_\vv \in{\mathcal A}}\sum_{\el=2}^n\sqrt{\el-1}\sqrt{\E_{\Omega_n(\tilde{\gamma})}[a_\el^2(Y_\el,W_\el)]\sum_{\vv=1}^{\el-1}\E_{\Omega_n(\tilde{\gamma})}[b_\vv^2(Y_\vv,W_\vv)]}\\
&\leq&\sigma_j^2(1+ e_{n,j})\sqrt{\frac{n(n-1)}{2}}.
\end{eqnarray*}
Finally,
\begin{eqnarray*}
B^2&=&\sup_{y,w}\sum_{\vv=1}^{n-1}\E_{\Omega_n(\tilde{\gamma})}\left[(U_j(y,w)-\E[U_j(Y_1,W_1)]))^2\times(U_j(Y_\vv,W_\vv)-\E[U_j(Y_1,W_1)])^2\right]\\
&\leq&4(n-1)C_j^2\sigma_j^2(1+ e_{n,j} ).
\end{eqnarray*}
Therefore, there exists an absolute constant $c>0$ such that for any $u>1$,
$$c_1C\sqrt{u}+c_2Du+c_3Bu^{3/2}+c_4Au^2\leq c(n\sigma_j^2u+C_j^2u^2).$$
\end{proof}
Let us go back to the proof of Proposition \ref{Th:conc}.  We apply Lemmas \ref{conc-sj} and \ref{conc-xi} with $u>1$ and we obtain, by setting
$$M_j(u)=\hat\sigma_j^2+2C_j\sigma_j\sqrt{\frac{2u(1+ e_{n,j})}{n}}+\frac{\sigma_j^2 u}{3n}+\frac{2c(n\sigma_j^2u+C_j^2u^2)}{n(n-1)},$$
\begin{eqnarray*}
\P\left(\sigma_j^2\geq M_j(u)\right)&\leq&\P\left(\sigma_j^2\geq s_j^2-\frac{2}{n(n-1)}\xi_j+2C_j\sigma_j\sqrt{\frac{2u (1+e_{n,j})}{n}}+\frac{\sigma_j^2 u}{3n}+\frac{2c(n\sigma_j^2u+C_j^2u^2)}{n(n-1)}\right)\\
&\leq&
\P\left(\left.\sigma_j^2\geq s_j^2+2C_j\sigma_j\sqrt{\frac{2u (1+e_{n,j})}{n}}+\frac{\sigma_j^2 u}{3n}\right|\Omega_n(\tilde{\gamma})\right)\\
&&\hspace{1cm}+\P\left(\left. \xi_j\geq c(n\sigma_j^2u+C_j^2u^2)\right|\Omega_n(\tilde{\gamma})\right)+1-\P(\Omega_n(\tilde{\gamma})).
\end{eqnarray*}
Therefore, with $u=\tilde{\gamma}\log n$ and $\tilde{\gamma}>1$, we obtain for $n$ large enough:
$$\P\left(\sigma_j^2\geq M_j(\tilde{\gamma}\log n)\right)\leq 5n^{-\tilde{\gamma}}.$$
And there exist $a$ and $b$ two absolute constants such that
$$\P\left(\sigma_j^2\geq \hat\sigma_j^2+2C_j\sigma_j\sqrt{\frac{2\tilde{\gamma}\log n (1+e_{n,j})}{n}}+ \frac{\sigma_j^2 a\tilde{\gamma}\log n}{n}+\frac{C_j^2b^2\tilde{\gamma}^2\log^2n}{n^2}\right)\leq 5n^{-\tilde{\gamma}}.$$ 
Now, we set 
$$\theta_1=\left(1-\frac{ a\tilde\gamma\log n}{n}\right), \quad\theta_2=C_j\sqrt{\frac{2\tilde \gamma\log n (1+ e_{n,j})}{n}},\quad \theta_3=\hat\sigma_j^2+\frac{C_j^2b^2\tilde\gamma^2\log^2n}{n^2}$$
so
$$\P\left(\theta_1\sigma_j^2-2\theta_2\sigma_j-\theta_3\geq 0\right)\leq 5n^{-\tilde \gamma}.$$
We study the polynomial
$$p(\sigma)=\theta_1 \sigma^2-2\theta_2\sigma-\theta_3.$$
Since $\sigma\geq 0$, $p(\sigma)\geq 0$ means that
$$\sigma \geq  \frac{1}{\theta_1}\left(\theta_2+\sqrt{\theta_2^2+\theta_1\theta_3}\right),$$
which is equivalent to
$$\sigma^2 \geq \frac{1}{\theta_1^2}\left(2\theta_2^2+\theta_1\theta_3+2\theta_2\sqrt{\theta_2^2+\theta_1\theta_3}\right).$$
Hence
$$\P\left(\sigma^2_j \geq \frac{1}{\theta_1^2}\left(2\theta_2^2+\theta_1\theta_3+2\theta_2\sqrt{\theta_2^2+\theta_1\theta_3}\right)\right)\leq 5n^{-\tilde \gamma}.$$
So,
$$\P\left(\sigma^2_j \geq \frac{\theta_3}{\theta_1}+\frac{2\theta_2\sqrt{\theta_3}}{\theta_1\sqrt{\theta_1}}+\frac{4\theta_2^2}{\theta_1^2}\right)\leq 5n^{-\tilde \gamma}.$$
So, there exist  absolute constants $\delta$, $\eta,$ and $\tau'$  depending only on $\tilde \gamma$ so that for $n$ large enough,
\begin{small}
$$
 \P\left(\sigma^2_j \geq \hat\sigma^2_j\left(1+\delta \frac{\log n}{n}\right)+ \left(1+\eta\frac{\log n}{n}\right)2C_j\sqrt{2\tilde \gamma\hat\sigma_j^2 (1+e_{n,j})\frac{\log n}{n}}+8\tilde \gamma C_j^2 \frac{\log n}{n}\left(1+\tau'\left(\frac{\log n}{n}\right)^{1/2}\right)\right) \leq 5 n^{-\tilde \gamma}.
$$
\end{small}
Finally, for all $\tilde\e>0$ there exists $R_4$ depending on $\e'$ and $\tilde \gamma$ such that for $n$ large enough 
$$\P(\sigma^2_j\geq (1+\e')\tilde\sigma^2_{j,\tilde\gamma})\leq R_4 n^{-\tilde \gamma}.$$
Combining this inequality with (\ref{conc1}), we obtain the desired result of Proposition \ref{Th:conc}.

\end{proof}

Proposition \ref{BorneMajorant} shows that the residual term in the oracle inequality is negligible.  

\begin{Prop}\label{BorneMajorant}
We have for any $q\geq 1$,
\begin{equation}
 \E \left [\sup_{j\in J} \left (  \left |  \hat{p}_j(x)-p_j(x) \right | -  \Gamma_{\gamma}(j) \right )_{+}^q \right ] = {o}(n^{-q}).
\end{equation}
\end{Prop}
\begin{proof}
We recall that  $J=\left \{ j\in \N^d:\quad 2^{\Sj}\leq\lfloor  {\frac{n}{\log^2n }} \rfloor \right \}$.\\
Let $\tilde \gamma >0$ and let us consider the event 
$$\tilde{\Omega}_{\tilde\gamma}=\left \{\sigma_j^2 \leq (1+\eps)\tilde\sigma_{j,\tilde\gamma}^2, \ \forall\, j\in J \right \}.$$ 
 Let $\gamma >0$. We set in the sequel
$$ E:=  \E \left [\sup_{j\in J} \left (  \left |  \hat{p}_j(x)-p_j(x) \right | -  \sqrt{\frac{2 \gamma (1+\e)\tilde\sigma_{j,\tilde\gamma}^2\log n}{n}}-\frac{c_j\gamma\log n}{n} \right )_{+}^q  \mathds{1}_{\tilde{\Omega}_{\tilde\gamma}}\right ],$$
and  $R_j:=  \left |  \hat{p}_j(x)-p_j(x) \right | $. We have:
\begin{eqnarray*}
E&=&\int_{0}^\infty \P \left [ \sup_{j\in J} \left  (  R_j  -   \sqrt{\frac{2\gamma(1+\e)\tilde\sigma_{j,\tilde\gamma}^2\log n}{n}}-\frac{c_j\gamma\log n}{n} \right )_+^q  \mathds{1}_{\tilde{\Omega}_{\tilde\gamma}} > y \right ] dy \\
& \leq & \sum_{j\in J} \int_{0}^\infty \P \left [ \left  (  R_j  -   \sqrt{\frac{2\gamma(1+\e)\tilde\sigma_{j,\tilde\gamma}^2\log n}{n}}-\frac{c_j\gamma\log n}{n} \right )_+^q  \mathds{1}_{\tilde{\Omega}_{\tilde\gamma}} > y \right ] dy  \\
&\leq & \sum_{j\in J} \int_{0}^\infty \P \left [ \left  (  R_j  -   \sqrt{\frac{2\gamma  \sigma_j^2\log n}{n}}-\frac{c_j\gamma\log n}{n} \right )^q > y \right ] dy.
\end{eqnarray*}
Let us take $u$ such that
$$ y=h(u)^q,$$
where
$$h(u)=\sqrt{\frac{2   \sigma_j^2 u}{n}}+ \frac{c_j u}{n}.$$
Note that for any $u>0$,
$$h'(u)\leq \frac{h(u)}{u}.$$ 
Hence 
\begin{eqnarray*}
E &  \leq & C \sum_{j\in J} \int_{0}^\infty \P \left [  R_j     >     \sqrt{\frac{2\gamma  \sigma_j^2\log n}{n}}+ \frac{c_j\gamma\log n}{n} +   \sqrt{\frac{2u  \sigma_j^2 }{n}}+ \frac{ u c_j }{n} \right ]  h(u)^{q-1}h'(u)  du \\
& \leq & C\sum_{j\in J} \int_{0}^\infty \P \left [  R_j     >     \sqrt{\frac{2\sigma_j^2 (\gamma \log n+u)}{n}}+ \frac{c_j(\gamma\log n+u)}{n}  \right ]  h(u)^{q-1}h'(u)  du .
\end{eqnarray*}
Now using concentration inequality (\ref{concentrationp_jsigma_j}), we get
\begin{eqnarray*}
 E &\leq &  C \sum_{j\in J} \int_{0}^\infty e^{-(\gamma \log n +u)}  h(u)^{q-1}h'(u)  du \\
  &\leq &  C \sum_{j\in J} \int_{0}^\infty e^{-(\gamma \log n +u)}  h(u)^{q}\frac 1 u   du \\
  &\leq &  C e^{-\gamma \log n}   \sum_{j\in J} \int_{0}^\infty e^{-u}  \left ( \sqrt{\frac{2   \sigma_j^2 u}{n}}+ \frac{c_j u}{n}  \right )^q \frac 1 u   du \\
  &\leq & C \left (e^{-\gamma \log n}   \sum_{j\in J}  \left  (  {\frac{\sigma_j^2 }{n}}  \right )^{q/2}   \int_{0}^\infty e^{-u}u^{\frac q 2 -1}  du +  \left (  \frac{c_j}{n}  \right )^q     \int_{0}^\infty e^{-u} u^{q-1} du  \right ).\\
\end{eqnarray*}
Now using Lemma \ref{ordredegrandeursigmajTj}, we have that $\sigma_j^2 \leq R_{10} 2^{\Sj(2\nu+1)}$ and $c_j \leq  C 2^{\Sj(\nu+1)}$. Hence, 
\begin{eqnarray*}
E &\leq& C \left ( e^{-\gamma \log n} \sum_{j\in J} \left ( {\frac{  2^{\Sj(2\nu+1)} }{n}}  \right )^{q/2}+ \left (\frac{ 2^{\Sj(\nu+1)} }{n} \right )^q  \right )\\
&\leq & C n^{-\gamma +q\nu}(\log n)^{-(2\nu+ 1 )q} = o(n^{-q}),
\end{eqnarray*}  as soon as $\gamma > q(\nu+1)$.


 It remains to find an upperbound for the following quantity: 

$$E':=  \E \left  [ \sup_{j\in J} \left (  \left |  \hat{p}_j(x)-p_j(x) \right | -  \sqrt{\frac{2\gamma(1+\e)\tilde\sigma_{j,\tilde\gamma}^2\log n}{n}}-\frac{c_j\gamma\log n}{n} \right )_{+}^q  \mathds{1}_{\tilde{\Omega}_{\tilde \gamma}^c} \right ] .$$

We have

\begin{eqnarray*}
E' &\leq&  \E\left [ \sup_{j\in J} \left (  |  \hat{p}_j(x)-p_j(x) \right | ^q  \mathds{1}_{\tilde{\Omega}_{\tilde \gamma}^c}\right ] \\
&\leq & 2^{q-1}\left (  \E \left [\sup_{j\in J} ( |\hat{p}_j(x)|)^q  \mathds{1}_{\tilde{\Omega}_{\tilde \gamma}^c} \right ] +  \E \left [ \sup_{j\in J}  ( |{p}_j(x)|)^q  \mathds{1}_{\tilde{\Omega}_{\tilde \gamma}^c} \right ] \right ) .
\end{eqnarray*}
First, let us deal with the term  $ \E \left [ \sup_{j\in J}  ( |{p}_j(x)|)^q  \mathds{1}_{\tilde{\Omega}_{\tilde \gamma}^c} \right ]$. 

Following the lines of the proof of Lemma \ref{Wavelets} we easily get that $  \sum_k \varphi^2_{jk}(x)  \leq C2^{\Sj}$, hence
\begin{eqnarray*} 
 |p_j(x) |&=& \left | \sum_k p_{jk} \varphi_{jk}(x)  \right | \leq \left  (\sum_k p_{jk}^2 \right )^{\frac 1 2}   \left  (\sum_k \varphi^2_{jk}(x) \right )^{\frac 1 2} \\
&\leq &C \| p\|_2 2^{\frac{ \Sj}{2}}.
\end{eqnarray*}
Now using Proposition \ref{Th:conc} which states that 
$  \P(\tilde{\Omega}_{\tilde \gamma}^c) \leq C n^{-\tilde \gamma}$
\begin{eqnarray}
\E \left [ \sup_{j\in J}  ( |{p}_j(x)|)^q  \mathds{1}_{\tilde{\Omega}_{\tilde \gamma}^c} \right ] &\leq &   \sup_{j\in J} (\| p\|_2 2^{\frac{ \Sj}{2}})^q \P(\tilde{\Omega}_{\tilde \gamma}^c) \\
&\leq &C \left ({\frac{n}{\log^2 n }} \right )^{\frac q 2} n^{-\tilde \gamma }.
\end{eqnarray}

It remains to find an upperbound for $   \E \left [\sup_{j\in J} ( |\hat{p}_j(x)|)^q  \mathds{1}_{\tilde{\Omega}_{\tilde \gamma}^c} \right ] $. We have
\begin{eqnarray*}
\E \left [\sup_{j\in J} ( |\hat{p}_j(x)|)^q  \mathds{1}_{\tilde{\Omega}_{\tilde \gamma}^c} \right ] &=& \E \left [\sup_{j\in J} \left |  \frac 1 n \sum_{\el =1}^n Y_\el T_j(W_\el )   \right |^q  \mathds{1}_{\tilde{\Omega}_{\tilde \gamma}^c} \right ] \\
&\leq & \frac{1}{n^q} \E  \left [ \sup_{j\in J}  \left (\sum_{\el =1}^n  \left | m(X_\el) + \eps_\el \right |  |T_j(W_\el ) |\right )^q \mathds{1}_{\tilde{\Omega}_{\tilde \gamma}^c} \right ]\\
&\leq & \frac{n^{q-1}}{n^q} \E  \left [ \sup_{j\in J}  \sum_{\el =1}^n  \left | m(X_\el) + \eps_\el  \right |^q  |T_j(W_\el) |^q \mathds{1}_{\tilde{\Omega}_{\tilde \gamma}^c} \right ]\\
&\leq & \frac{C}{n}   \E  \left [ \sup_{j\in J}  \sum_{\el =1}^n  \left (\| m \|_{\infty}^q+ |\eps_\el  \right |^q ) |T_j(W_\el ) |^q \mathds{1}_{\tilde{\Omega}_{\tilde \gamma}^c} \right ]\\
&\leq& C \left( \sup_{j\in J}( \|T_j \|_{\infty}^q ) \P(\tilde{\Omega}_{\tilde \gamma}^c) + \sup_{j\in J}( \|T_j \|_{\infty}^q ) \E \left [  |\eps_1|^q  \mathds{1}_{\tilde{\Omega}_{\tilde \gamma}^c}\right ] \right )\\
&\leq & C  \left (\sup_{j\in J}( \|T_j \|_{\infty}^q ) \P(\tilde{\Omega}_{\tilde \gamma}^c) +  \sigma_{\eps}^q\sup_{j\in J}( \|T_j \|_{\infty}^q ) \left(\E \left [  |Z |^{2q} \right ]\right)^{\frac 1 2 } \left (\P( {\tilde{\Omega}_{\tilde \gamma}^c}) \right )^{\frac 1 2 } \right ),
\end{eqnarray*}
where $Z\sim \mathcal{N}(0,1).$ Using (\ref{MajorationMomentGaussienne}) and $\| T_j\|_{\infty} \leq T_4 2^{\Sj(\nu+1)}$ , we get
\begin{eqnarray*}
\E \left [\sup_{j\in J} ( |\hat{p}_j(x)|)^q  \mathds{1}_{\tilde{\Omega}_{\tilde \gamma}^c} \right ]  & \leq &  C \left ( \frac{n}{\log^2 n} \right )^{(\nu+1)q} n^{-\frac{\tilde \gamma  }{2}},
\end{eqnarray*}

 We have
\begin{eqnarray*}
E' &\leq & C    n^{-\frac{\tilde \gamma }{2}} \left ( \left  ( {\frac{n}{\log^2 n }} \right )^{\frac q 2}   +  \left (\frac{n}{\log^2 n} \right )^{(\nu+1)q}   \right ) \\
&=&o(n^{-q}),
\end{eqnarray*}
as soon as $\tilde\gamma >2q(\nu+2) $. This ends the proof of Proposition \ref{BorneMajorant}.

\end{proof}

Proposition \ref{PropBiais} controls the bias term in the oracle inequality. 

\begin{Prop}\label{PropBiais}
For any $j=(j_1,\ldots,j_d)\in\Z^d$ and $j'=(j'_1,\ldots,j'_d)\in\Z^d$ and any $x$, if $p\in\mathbb{H}_d(\vec{\beta},L)$
$$|p_{j\wedge j'}(x)-p_{j'}(x)|\leq R_{12} L\sum_{\el =1}^d 2^{-j_\el \beta_\el },$$
where $R_{12}$ is a constant only depending on $\varphi$ and $\vec{\beta}$. We have denoted $$j\wedge j'=(j_1\wedge j'_1,\ldots, j_d\wedge j'_d).$$
\end{Prop}
\begin{proof}
We first state three lemmas.

\begin{lemma}\label{SansBiais}
For any $j$ and any $k$, we have:
$$\E[\hat{p}_{jk}]=p_{jk}.$$
\end{lemma}
\begin{proof}
Recall that 
$$
\hat{p}_{jk}:=  \frac 1 n \sum_{u=1}^n Y_u \times (\mathcal{D}_j \varphi)_{j,k}(W_u)=2^{\frac \Sj  2} \frac 1 n \sum_{u=1}^n Y_u \int e^{-i<t,2^jW_u-k>}  \prod_{\el=1}^d
\frac{\overline{\mathcal{F}(\varphi)(t_\el)}}{\mathcal{F}(g_\el)(2^{j_\el}t_\el)} dt. $$
Let us prove now that $\E(\hat{p}_{jk})=p_{jk}.$\\
We have
$$\E(\hat{p}_{jk})= 2^{\frac \Sj 2}  \left ( \int \E(Y_1e^{-i<t, 2^jW_1-k>})  \prod_{\el=1}^d
\frac{\overline{\mathcal{F}(\varphi)(t_\el)}}{\mathcal{F}(g_\el)(2^{j_\el}t_\el)} dt \right  ).$$
We shall develop the right member of the last equality.
We have :
 \begin{eqnarray*}
 \E \left [Y_1e^{-i < t,2^jW_1-k>} \right ] &=& \E \left [(m(X_1)+\varepsilon_1 ) e^{-i <t,2^jW_1-k>} \right ]\\
 &=& \E \left [m(X_1)e^{-i<t,2^jW_1-k>} \right ]\\
 &=& \E \left  [m(X_1)e^{-i<t,2^jX_1-k>} \right ] \E \left [e^{-i <t,2^j \delta_1>} \right ]\\
 &=& \int m(x) e^{-i<t,2^jx-k> }f_X(x)dx\times \mathcal{F}(g)(2^{j}t)\\
 &=&  e^{i<t,k>}\mathcal{F}(p)(2^jt) \mathcal{F}(g)(2^jt).
 \end{eqnarray*}
 Consequently 
 \begin{eqnarray*}
\E \left [\hat{p}_{jk} \right ] &=& 2^{\frac \Sj 2}  \int e^{i<t,k>}  \mathcal{F}(p)(2^j t) \mathcal{F}(g)(2^jt)  \prod_{\el=1}^d
\frac{\overline{\mathcal{F}(\varphi)(t_\el)}}{\mathcal{F}(g_\el)(2^{j_\el}t_\el)} dt \\
&=&  2^{\frac \Sj 2} \int  e^{i<t,k>} \mathcal{F}(p)(2^jt) \prod_{\el=1}^d \overline{ \mathcal{F}(\varphi)(t_\el)}dt \\
&=& \int \mathcal{F}(p)(t)\overline{ \mathcal{F}(\varphi_{jk})(t) }dt.
 \end{eqnarray*}
Since by Parseval equality, we have 
$$p_{jk}= \int p(t)\varphi_{jk}(t) dt=  \int \mathcal{F}(p)(t) \overline{\mathcal{F}(\varphi_{jk})(t)} dt,$$
the result follows.

Note that in the case where we don't have any noise on the variable i.e $g(x)=\delta_0(x)$, since $\mathcal{F}(g)(t)=1$, the proof above remains valid and we get $\E[\hat{p}_{jk}]=p_{jk}$.

\end{proof}

\begin{lemma}\label{BiasGap:Gene}
If for any $\el $, $\lfloor \beta_\el \rfloor\leq N$, the following holds:  for any $j\in\Z^d$ and any $ p\in\mathbb{H}_d(\vec{\beta},L) $, 
$$
|\E[\hat p_j(x)]-p(x)|\leq L(\|\varphi\|_{\infty}\|\varphi\|_1)^d(2A+1)^d\sum_{\el =1}^d\frac{(2A\times 2^{-j_\el })^{\beta_\el }}{{\lfloor\beta_\el \rfloor !}}.
$$
\end{lemma}
\begin{proof}
Let $x$ be fixed and $j=(j_1,\ldots,j_d)\in\Z^d$. We have:
$$\int K_j(x,y)dy=\int\sum_{k_1}\cdots\sum_{k_d}\prod_{\el =1}^d[2^{j_\el }\varphi(2^{j_\el }x_\el -k_\el )\varphi(2^{j_\el }y_\el -k_\el )dy_\el ]=1.$$
Therefore, using lemma \ref{SansBiais}
\begin{eqnarray*}
\E[\hat{p}_j(x)]-p(x)= p_j(x)-p(x)&=&\int K_j(x,y)(p(y)-p(x))dy\\
&=&\sum_{k}\varphi_{jk}(x)\int\varphi_{jk}(y)(p(y)-p(x))dy\\
&=&\sum_{k_1\in \mathcal{Z}_{j,1}(x)}\cdots\sum_{k_d\in \mathcal{Z}_{j,d}(x)}\varphi_{jk}(x)\int\prod_{\el =1}^d2^{\frac{j_\el }{2}}\varphi(2^{j_\el }y_\el -k_\el )(p(y)-p(x))dy.
\end{eqnarray*}
Now, we use that
$$p(y)-p(x)= \sum_{\el =1}^d p(x_1,\ldots,x_{\el-1},y_\el ,y_{\el+1},\ldots,y_d)-p(x_1,\ldots x_{\el -1},x_\el ,y_{\el +1},\ldots,y_d),$$
with $p(x_1,\ldots,x_\el ,y_{\el +1},\ldots,y_d)=p(x_1,\ldots,x_d)$ if $\el=d$ and $p(x_1,\ldots,x_{\el -1},y_\el ,\ldots,y_d)=p(y_1,\ldots,y_d)$ if $\el =1$. Furthermore, the Taylor expansion gives: for any $\el \in\{1,\ldots,d\}$, for some $u_\el \in [0;1],$
\begin{eqnarray*}
\hspace{-1cm} p(x_1,\ldots,x_{\el -1},y_\el ,y_{\el +1},\ldots,y_d)-p(x_1,\ldots x_{\el -1},x_\el ,y_{\el +1},\ldots,y_d)=&&\\  \sum_{k=1}^{\lfloor \beta_\el\rfloor}\frac{\partial^k p}{\partial x_ \el^k}(x_1,\ldots x_{\el -1},x_\el ,y_{\el +1},\ldots,y_d)\times \frac{(y_\el-x_\el )^k}{k!}+ 
\\ \frac{\partial^{\lfloor \beta_\el \rfloor} p}{\partial x_\el ^{\lfloor \beta_\el \rfloor}}(x_1,\ldots x_{\el -1},x_\el +(y_\el -x_\el )u_\el,y_{\el +1},\ldots,y_d)\times \frac{(y_\el -x_\el )^{\lfloor \beta_\el \rfloor}}{\lfloor \beta_\el \rfloor!}\\-\frac{\partial^{\lfloor \beta_\el \rfloor} p}{\partial x_\el ^{\lfloor \beta_\el \rfloor}}(x_1,\ldots x_{\el -1},x_\el ,y_{\el+1},\ldots,y_d)\times \frac{(y_\el -x_\el )^{\lfloor \beta_\el \rfloor}}{\lfloor \beta_\el \rfloor!}&&.
\end{eqnarray*}
Using vanishing moments of $K$ and $p\in\mathbb{H}_d(\vec{\beta},L)$, we obtain:
\begin{eqnarray*}
|p_j(x)-p(x)|&\leq&\sum_{k_1\in \mathcal{Z}_{j,1}(x)}\cdots\sum_{k_d\in \mathcal{Z}_{j,d}(x)}|\varphi_{jk}(x)|\int\prod_{\el =1}^d2^{\frac{j_\el }{2}}|\varphi(2^{j_\el }y_\el -k_\el )|\sum_{\el=1}^d L\frac{|y_{\el}-x_{\ell}|^{\beta_{\el}}}{\lfloor\beta_{\el}\rfloor !}dy\\
&\leq&\|\varphi\|_{\infty}^d\sum_{k_1\in \mathcal{Z}_{j,1}(x)}\cdots\sum_{k_d\in \mathcal{Z}_{j,d}(x)}\int_{[-A;A]^d}\prod_{\el =1}^d|\varphi(u_\el )|\sum_{\el=1}^d L\frac{|2^{-j_{\el}}(u_{\el}+k_{\el})-x_{\el}|^{\beta_{\el}}}{\lfloor\beta_{\el}\rfloor !}du.
\end{eqnarray*}
Since  for any $\el$, $k_{\el}\in \mathcal{Z}_{j,\el}(x)$, we finally obtain
\begin{eqnarray*}
|p_j(x)-p(x)|&\leq&\|\varphi\|_{\infty}^d\sum_{k_1\in \mathcal{Z}_{j,1}(x)}\cdots\sum_{k_d\in \mathcal{Z}_{j,d}(x)}\int_{[-A;A]^d}\prod_{\el=1}^d|\varphi(u_\el )|\sum_{\el=1}^d L\frac{(2A\times 2^{-j_{\el}})^{\beta_{\el}}}{\lfloor\beta_{\el}\rfloor !}du\\
&\leq& L(\|\varphi\|_{\infty}\|\varphi\|_1)^d(2A+1)^d\sum_{\el=1}^d\frac{(2A\times 2^{-j_\el})^{\beta_\el}}{\lfloor\beta_{\el}\rfloor !}.
\end{eqnarray*}
\end{proof}
\begin{lemma}\label{projection}
We have for any $j=(j_1,\ldots,j_d)\in\Z^d$ and $j'=(j'_1,\ldots,j'_d)\in\Z^d$ and any $x$,
$$K_{j'}(p_j)(x)=p_{j\wedge j'}(x).$$
\end{lemma}

\begin{proof}
We only deal with the case $d=2$. The extension to the general case can be easily deduced. If for $i=1,2$, $j_i\leq j'_i$ the result is obvious. It is also the case if for $\el =1,2$, $j'_\el \leq j_\el $. So, without loss of generality, we assume that $j_1\leq j'_1$ and  $j'_2\leq j_2.$ We have:
\begin{eqnarray*}
K_{j'}(p_j)(x)&=&\int K_{j'}(x,y)p_j(y)dy\\
&=&\int \sum_{k}\varphi_{j'k}(x)\varphi_{j'k}(y)p_j(y)dy\\
&=&\int \sum_{k_1}\sum_{k_2}\varphi_{j'_1k_1}(x_1)\varphi_{j'_2k_2}(x_2)\varphi_{j'_1k_1}(y_1)\varphi_{j'_2k_2}(y_2)p_j(y)dy_1dy_2\\
&=&\int \sum_{k_1}\sum_{k_2}\varphi_{j'_1k_1}(x_1)\varphi_{j'_2k_2}(x_2)\varphi_{j'_1k_1}(y_1)\varphi_{j'_2k_2}(y_2)\\
&&\hspace{1cm}\times\sum_{\ell_1}\sum_{\ell_2}\varphi_{j_1\ell_1}(y_1)\varphi_{j_2\ell_2}(y_2)\varphi_{j_1\ell_1}(u_1)\varphi_{j_2\ell_2}(u_2)p(u_1,u_2)du_1du_2dy_1dy_2.
\end{eqnarray*}
Since $j_1\leq j'_1$, we have in the one-dimensional case, by a slight abuse of notation,  $V_{j_1}\subset V_{j'_1}$ and
$$\int \sum_{k_1}\varphi_{j'_1k_1}(x_1)\varphi_{j'_1k_1}(y_1)\varphi_{j_1\ell_1}(y_1)dy_1=\int K_{j'_1}(x_1,y_1)\varphi_{j_1\ell_1}(y_1)dy_1=\varphi_{j_1\ell_1}(x_1).$$
Similarly, since $j'_2\leq j_2$, we have $V_{j'_2}\subset V_{j_2}$ and
$$\int \sum_{\ell_2}\varphi_{j_2\ell_2}(y_2)\varphi_{j_2\ell_2}(u_2)\varphi_{j'_2k_2}(y_2)dy_2=\int K_{j_2}(u_2,y_2)\varphi_{j'_2k_2}(y_2)dy_2=\varphi_{j'_2k_2}(u_2).$$
Therefore, with $\tilde j=j\wedge j'$,
\begin{eqnarray*}
K_{j'}(p_j)(x)&=&\int \sum_{k_2}\sum_{\ell_1}\varphi_{j'_2k_2}(x_2)\varphi_{j_1\ell_1}(u_1)\varphi_{j_1\ell_1}(x_1)\varphi_{j'_2k_2}(u_2)p(u_1,u_2)du_1du_2\\
&=&\int \sum_{\ell_1}\sum_{\ell_2}\varphi_{\tilde j_2\ell_2}(x_2)\varphi_{\tilde j_1\ell_1}(u_1)\varphi_{\tilde j_1\ell_1}(x_1)\varphi_{\tilde j_2 \ell_2}(u_2)p(u_1,u_2)du_1du_2\\
&=&\int\sum_{\ell}\varphi_{\tilde j\ell}(x)\varphi_{\tilde j\ell}(u)p(u)du\\
&=&p_{\tilde j}(x),
\end{eqnarray*}
which ends the proof of the lemma.
\end{proof}
Now, we shall go back to the proof of Proposition \ref{PropBiais}. We easily deduce  the result :
\begin{eqnarray*}
p_{j\wedge j'}(x)-p_{j'}(x)&=&K_{j'}(p_j)(x)-K_{j'}(p)(x)\\
&=&\int K_{j'}(x,y)(p_j(y)-p(y))dy.
\end{eqnarray*}
Therefore,
\begin{eqnarray*}
|p_{j\wedge j'}(x)-p_{j'}(x)|&\leq&\int |K_{j'}(x,y)||p_j(y)-p(y)|dy\\
&\leq &R_{12} L\sum_{\el =1}^d 2^{-j_\el\beta_\el}\times \int |K_{j'}(x,y)|dy,
\end{eqnarray*}
where $R_{12}$ is a constant only depending on $\varphi$ and $\vec{\beta}$. 
We conclude by observing that
\begin{eqnarray*}
\int |K_{j'}(x,y)|dy&=&\int\sum_{k_1}\cdots\sum_{k_d}\prod_{\el =1}^d[2^{j'_\el }|\varphi(2^{j'_\el }x_\el -k_\el )||\varphi(2^{j'_\el }y_\el -k_\el )|dy_i]\\
&\leq& \|\varphi\|_{\infty}^d\sum_{k_1\in \mathcal{Z}_{j',1}(x)}\cdots\sum_{k_d\in \mathcal{Z}_{j',d}(x)}\left(\int |\varphi(v)|dv\right)^d\\
&\leq&\left( \|\varphi\|_{\infty} \|\varphi\|_1(2A+1)\right)^d.
\end{eqnarray*}
We thus obtain the claimed result of Proposition \ref{PropBiais}.
\end{proof}

\subsubsection{Appendix}\label{appendix}

Technical lemmas are stated and proved below.

\begin{lemma}\label{Rosenthal} 
We have
$$
\E [(\tilde{\sigma}_{j,\tilde \gamma})^{ q  }] \leq R_5 2^{\Sj (2\nu+1)\frac{q}{2}},
$$
with $R_5$ a constant depending on $q,\tilde \gamma, \| m\|_{\infty}, s,  \| f_X\|_{\infty},\varphi, c_g, \mathcal{C}_g$.
\end{lemma}
\begin{proof}
First, let us focus on the case $q  \geq 2$. We recall the expression of  $\tilde{\sigma}_{j,\tilde \gamma}^2$
$$
\tilde\sigma^2_{j,\tilde\gamma} =\hat\sigma^2_j+2C_j\sqrt{2 \tilde\gamma \hat\sigma_j^2\frac{\log n}{n}}+8 \tilde\gamma  C_j^2 \frac{\log n}{n}.
 $$
We shall first prove that $$ \E[(\hat{\sigma}_j)^{q } ]\leq C2^{\Sj (2\nu+1)\frac{q}{2}}  . $$  Let us remind that
$$ \hat{\sigma}_j^2 =\frac{1}{2n(n-1)}\sum_{\el\neq v} (U_j(Y_\el,W_\el)-U_j(Y_\vv,W_\vv))^2.$$
We easily get
$$
 \hat{\sigma}_j^2 \leq \frac{C}{n}\sum_{\el} (U_j(Y_\el,W_\el)-\E [U_j(Y_1,W_1)])^2.
$$
 First let us remark that
\begin{equation*}
 \left  (\sum_{\el} (U_j(Y_\el,W_\el)-\E [U_j(Y_1,W_1)])^2 \right )^{\frac q 2 }  \leq C \left ( \left (\sum_\el  (  (U_j(Y_\el,W_\el)-\E [U_j(Y_1,W_1)])^2 -\sigma_j^2) \right )^{\frac q 2} + n^{\frac q 2}\sigma_j^q \right )
\end{equation*}
We will use Rosenthal inequality (see \cite{Hardle}) to find an upper bound for
$$ \E \left [   \left (\sum_\el  (  (U_j(Y_\el,W_\el)-\E [U_j(Y_1,W_1)])^2 -\sigma_j^2 ) \right )^{\frac q 2 }  \right ].$$
We set 
$$ B_\el:=   (U_j(Y_\el,W_\el)-\E [U_j(Y_1,W_1)])^2 -\sigma_j^2.$$
The variables $B_\el$ are i.i.d and centered. We have to check that $ \E[|B_\el |^{\frac q 2}] < \infty $. We have
$$  \E[|B_\el |^{\frac q 2}]  \leq C(  \E [  |  (U_j(Y_\el,W_\el)-\E [U_j(Y_1,W_1)] |^q ] + \sigma_j^q ),$$
but 
$$   \E [  |  (U_j(Y_\el,W_\el)-\E [U_j(Y_1,W_1)] |^q ] =\frac{A_q}{n},$$
with $A_q$ defined in (\ref{defA_q}).  Hence
\begin{equation}\label{Bq}
  \E[|B_\el |^{\frac q 2}]  \leq C\left (\frac{A_q}{n}+ \sigma_j^q  \right ).
  \end{equation}
 Using the control of $A_q$ in (\ref{controlA_q}), equation (\ref{sigma}) and Lemma \ref{ordredegrandeursigmajTj} we have
\begin{eqnarray}
A_q &\leq& Cn\sigma_j^2 \| T_j\|_{\infty}^{q-2} \nonumber \\
&\leq & C n2^{\Sj(q\nu +q-1)} \label{control2A_q}.
\end{eqnarray}
Now, we are able to apply the Rosenthal inequality to the variables $B_\el$ which yields
\begin{eqnarray*}
\E \left [ \left  (\sum_\el B_\el  \right )^{\frac q 2 }\right ] &\leq &C \left ( \sum_\el \E[|B_\el|^{\frac q 2 }] + \left (\sum_\el \E [ B_\el^2] \right )^{\frac q 4 }\right),  \\
  \end{eqnarray*}
   and using (\ref{Bq}) and (\ref{control2A_q}) we get
  \begin{eqnarray*}
 \E \left [ \left  (\sum_\el B_\el  \right )^{\frac q 2 }\right ] &\leq & C \left(  \sum_\el \left (\frac{A_q}{n}+\sigma_j^q  \right) + \left (\sum_\el \left (\frac{A_4}{n} + \sigma_j^4 \right  ) \right )^{\frac q 4 } \right)  \\
&\leq& C \left ( A_q +n\sigma_j^q +(A_4)^{\frac q 4 } +n^{\frac q 4} \sigma_j^q \right) \\
 &\leq & C\left(   n2^{\Sj(q\nu+q-1)}  +n 2^{\Sj(2\nu+1)\frac{q}{2}} + ( n2^{\Sj(4\nu+3} )^{\frac q 4 }  \right) .
 \end{eqnarray*}
Consequently
\begin{eqnarray*}
\E[\hat{\sigma}_j^{q }] &\leq&  C{n^{-\frac q 2 }} \left( n2^{\Sj(q\nu+q-1)} +n 2^{\Sj(2\nu +1)\frac{q}{2} } + ( n2^{\Sj(4\nu+3} )^{\frac q 4 }  + n^{\frac q 2 } 2^{\Sj(2\nu +1)\frac{q}{2} } \right) \\
&\leq & C (n^{1-\frac q 2}2^{\Sj(q\nu+q-1)} +  n^{1-\frac q 2} 2^{\Sj(2\nu+1)\frac{q}{2}}  +n^{-\frac q 4} 2^{\Sj(4\nu+3)\frac q 4} + 2^{\Sj(2\nu+1)\frac{q}{2}}).
\end{eqnarray*}
Let us compare each term of the r.h.s of the last inequality. We have 
$$ n^{1-\frac q 2}2^{\Sj(q\nu+q-1)} \leq 2^{\Sj(2\nu+1)\frac{q}{2}}  \Longleftrightarrow 2^{\Sj} \leq n,$$
which is true by (\ref{setJ}). Similarly we have 
$$ n^{-\frac q 4} 2^{\Sj(4\nu+3)\frac q 4} \leq  2^{\Sj(2\nu+1)\frac{q}{2}}  \Longleftrightarrow 2^{\Sj} \leq n ,  $$
and obviously
$$  n^{1-\frac q 2} 2^{\Sj(2\nu+1)\frac{q}{2}}  \leq 2^{\Sj(2\nu+1)\frac{q}{2}}. $$

Thus we get that the dominant term in r.h.s  is $2^{\Sj(2\nu+1)\frac{q}{2}}$. Hence
$$ \E[\hat{\sigma}_j^q] \leq C2^{\Sj (2\nu+1)\frac{q}{2}}  . $$
Now using that
\begin{eqnarray*}
 \E [ \tilde\sigma_{j,\tilde\gamma}^{ q } ] &\leq & C \left(  \E[ \hat{\sigma}_j^{ q }]  +\left (2C_j\sqrt{2\tilde \gamma \frac{\log n}{n}}\right) ^{\frac q 2}\E[ \hat{\sigma}_j^{\frac q 2}] +   \left (8 \tilde\gamma  C_j^2 \frac{\log n}{n} \right)^{\frac q 2 } \right),
 \end{eqnarray*}
  and since $C_j \leq C \sqrt{\log n} 2^{\Sj(\nu+1)}$, we have
 \begin{eqnarray*}
  \E [ \tilde\sigma_{j,\tilde\gamma}^{ q } ] &\leq &C  \left (  2^{\Sj(2\nu+1)\frac{q}{2}} +((\log n)  n^{-\frac 1 2 } 2^{\Sj(\nu+1)})^{\frac q 2}  2^{\Sj (2\nu+1)\frac{q}{4}} +
 \left (\frac{\log^2 n }{n} 2^{2\Sj(\nu+1)}\right)^{\frac q 2}
    \right).\\
       \end{eqnarray*}
   Let us compare the three terms of the right hand side. We have 
   
   \begin{eqnarray*}
    2^{\Sj\frac{q(2\nu+1)}{2}} &\geq&  ((\log n)  n^{-\frac 1 2 } 2^{\Sj(\nu+1)})^{\frac q 2}  2^{\Sj(2\nu+1)\frac{q}{4}} \Longleftrightarrow 2^{\Sj(q\nu + \frac q 2)} \geq (\log n)^{\frac q 2} n^{-\frac q 4} 2^{\Sj(q\nu +\frac{3q}{4})}   \Longleftrightarrow  2^{\Sj} \leq  \frac{n}{\log^2 n},
   \end{eqnarray*}
   which is true by (\ref{setJ}).
   Furthermore we have
   \begin{equation}
   2^{\Sj\frac{q(2\nu+1)}{2}} \geq   \left (\frac{\log^2 n }{n} 2^{2\Sj(\nu+1)}\right)^{\frac q 2}  \Longleftrightarrow 2^{\Sj(q\nu +\frac q 2)} \geq \left (\frac{\log^2 n}{n}\right)^{\frac q2} 2^{\Sj(q\nu +q)}  \Longleftrightarrow 2^{\Sj} \leq  \frac{n}{\log^2 n},
 \end{equation}
    which is true again by (\ref{setJ}).
Consequently
$$
 \E [ \tilde\sigma_{j,\tilde\gamma}^{q} ] \leq R_5 2^{\Sj(2\nu+1)\frac{q}{2}},
 $$
 with $R_5$ a constant depending on $q,\tilde \gamma, \| m\|_{\infty}, s,  \| f_X\|_{\infty},\varphi, c_g, \mathcal{C}_g$
 and the lemma is proved for $q\geq 2$. \\For the case $ q \leq  2$ the result follows from Jensen inequality.
\end{proof}

\begin{lemma} \label{Wavelets}
Under assumption (A1) on the father wavelet $\varphi$, we have for any $j=(j_1,\ldots,j_d)$ and any $x \in \R^d$,
$$\sum_{k}|\varphi_{jk}(x)|\leq (2A+1)^d\|\varphi\|_{\infty}^d2^{\frac{\Sj}{2}}.$$
\end{lemma}
\begin{proof}
Let $x \in \R^d$ be fixed. We set for any $j$ and any $\el\in\{1,\ldots,d\}$,
$$\mathcal{Z}_{j,\el}(x)=\left\{k_\el :\quad |2^{j_\el}x_\el-k_\el|\leq A\right\},$$
whose cardinal is smaller or equal to $(2A+1)$.
Since 
$$\varphi_{jk}(x)=\prod_{\el=1}^d2^{\frac{j_\el}{2}}\varphi(2^{j_\el}x_\el-k_\el),$$ then
$$\varphi_{jk}(x)\not=0\Rightarrow \forall\,\el\in\{1,\ldots,d\}, \ k_\el\in \mathcal{Z}_{j,\el}(x).$$
Now,
\begin{eqnarray*}
\sum_{k}|\varphi_{jk}(x)|&=&\sum_{k_1\in \mathcal{Z}_{j,1}(x)}\cdots\sum_{k_d\in
\mathcal{Z}_{j,d}(x)}\prod_{\el=1}^d2^{\frac{j_\el}{2}}|\varphi(2^{j_\el}x_\el-k_\el)|\\
&\leq&\sum_{k_1\in \mathcal{Z}_{j,1}(x)}\cdots\sum_{k_d\in \mathcal{Z}_{j,d}(x)}\|\varphi\|_{\infty}^d2^{\frac{\Sj}{2}}\\
&\leq&(2A+1)^d\|\varphi\|_{\infty}^d2^{\frac{\Sj}{2}}.
\end{eqnarray*}
\end{proof}

\begin{lemma}\label{deriveephi}
Under condition (A1) and $\varphi$ is $\mathcal{C}^r$, there exist constants $R_6$  and $R_7$ depending on $\varphi$ such that 
 \begin{equation}\label{deriveeTFwavelet1}
\left | \mathcal{F}^{}(\varphi)(t)  \right | \leq  R_6 (1+|t|)^{-r}, \quad \textrm{for} \; any\;  t.
 \end{equation}
 and
  \begin{equation}\label{deriveeTFwavelet2}
  \left | \overline{\mathcal{F}^{}(\varphi)(t)}'  \right | \leq  R_7 (1+|t|)^{-r}, \quad \textrm{for} \; any\;  t.
 \end{equation}
 \end{lemma}
\begin{proof}
First, let us focus on the case $|t| \geq 1$. \\
 
We have by integration by parts that
\begin{eqnarray*}
 \mathcal{F}(\varphi)(t)  = \int e^{-itx} \varphi(x) dx = \left [  -\frac{1}{it}e^{-itx}\varphi(x) \right ]^{\infty}_{-\infty}+\frac{1}{it}\int e^{-itx}\varphi'(x) dx.
\end{eqnarray*}
Using that the father wavelet $\varphi$ is compactly supported on $[-A,A]$, we get 
\begin{eqnarray*}
 \mathcal{F}^{}(\varphi)(t)  = \frac{1}{it}\int e^{-itx}\varphi'(x) dx.
\end{eqnarray*}
By successive integration by parts and using that $|t|\geq 1$ one gets 
\begin{eqnarray*}
\left | \mathcal{F}(\varphi)(t)  \right |= \left |\frac{1}{(it)^r}\int e^{-itx}\varphi^{(r)}(x) dx \right | \leq \frac{2^r}{(1+|t|)^r} \int|\varphi^{(r)}(x)| dx,
\end{eqnarray*}
the integral $ \int_{-A}^{A}|\varphi^{(r)}(x)| dx$ being finite. \\
For the derivative we have
\begin{eqnarray*}
\overline{ \mathcal{F}(\varphi)(t) }' = i\int e^{itx}  x \varphi(x) dx.
\end{eqnarray*}
Following the same scheme as for $\mathcal{F}(\varphi)(t)$, one gets by integration by parts and using the Leibniz formula that
\begin{eqnarray*}
\left | \overline{ \mathcal{F}(\varphi)(t) }'  \right | &=& \left | \frac{1}{(it)^r}\int e^{itx}\frac{d^r}{dx^r}(x\varphi(x)) dx \right |  = \left | \frac{1}{(it)^r}\int e^{itx}\sum_{k=0}^{r}\binom{r}{k} x^{(k)}\varphi(x)^{(r-k)} dx \right | \\
 &\leq&  \frac{2^r}{(1+|t|)^r}   \sum_{k=0}^{r}\binom{r}{k} \int | x^{(k)}\varphi(x)^{(r-k)}| dx,
\end{eqnarray*}
the quantity  $\sum_{k=0}^{r}\binom{r}{k} \int_{-A}^{A} | x^{(k)}\varphi(x)^{(r-k)}| dx$ being finite.

Hence the lemma is proved for $|t| \geq 1$.\\
The result for $|t|\leq 1$ is obvious since
 $$\left | \mathcal{F}^{}(\varphi)(t) \right | =\left | \int e^{-itx} \varphi(x)dx \right | \leq \int | \varphi(x)| dx < \infty ,$$
 and
  $$\left | \overline{ \mathcal{F}^{}(\varphi)(t) }'\right | =\left | i\int e^{itx}x \varphi(x)dx \right | \leq \int |x \varphi(x)| dx < \infty .$$
 \\ 
Then the lemma is proved for any $t$. 
\end{proof}

\begin{lemma}\label{operateurK_j} Under conditions (A1) and (A3), for $\nu \geq 0$, 
we have
$$\left |  (\mathcal{D}_j\varphi )(w) \right | \leq R_8 2^{\Sj \nu}  \prod_{\el=1}^d (1+|w_\el|)^{-1}, \; w \in \R^d$$

where $R_8$ is a  constant depending on $\varphi$, $ {\mathcal C_g}$ and $c_g$.
\end{lemma}

\begin{proof} If all the $  |w_\el| <1$ then using (\ref{hypobruit1}), Lemma \ref{deriveephi}  and $r\geq \nu+2$ with $\nu \geq 0 $ we have
\begin{eqnarray}   
\left | ( \mathcal D_j \varphi ) (w)   \right |  &\leq&   \prod_{\el=1}^d \int \frac{\left | 
\mathcal{F}(\varphi)(t_\el) \right | }{\left |  \mathcal{F}(g_\el)(2^{j_\el} t_\el) \right |} dt_\el \\
&\leq &  C  \prod_{\el=1}^d  \int    \left |\mathcal{F}(\varphi)(t_\el)(1+2^{j_\el} |t_\el|)^{\nu}
\right | dt_\el \\
&\leq & C 2^{\Sj \nu}     \prod_{\el=1}^d \int      (1+| t_\el |)^{\nu-r} dt_\el \\
&\leq &  C 2^{\Sj \nu}  \leq  C 2^{\Sj \nu} 
\prod_{\el=1}^d (1+|w_\el|)^{-1}. \label{wi}
\end{eqnarray}

Now we consider the case where there exists at least one $w_\el$ such that $|w_\el| \geq 1$. We have 
$$
( \mathcal D_j \varphi ) (w) =   \prod_{\el=1,|w_\el |\leq 1}^d \int e^{-it_\el w_\el} \frac{
\overline{\mathcal{F}(\varphi)(t_\el)}  }{  \mathcal{F}(g_\el)(2^{j_\el} t_\el) } dt_\el\times\prod_{\el=1,|w_\el |\geq  1}^d \int e^{-it_\el w_\el} \frac{\
\overline{\mathcal{F}(\varphi)(t_\el) } }{  \mathcal{F}(g_\el)(2^{j_\el} t_\el) } dt_\el.
$$
For the left-hand product on $|w_\el | \leq 1$ we use the result (\ref{wi}). Now let us consider the right-hand product with $|w_\el | \geq 1$. We set in the sequel
$$\eta_\el (t_\el) :=\frac{ \overline{ \mathcal{F}(\varphi)(t_\el)}}{ \mathcal{F}(g_\el)(2^{j_\el} t_\el )  }. $$ We have

$$
\prod_{\el=1, |w_\el |\geq  1}^d \int e^{-it_\el w_\el} \frac{\
\overline{\mathcal{F}(\varphi)(t_\el) } }{  \mathcal{F}(g_\el)(2^{j_\el} t_\el) } dt_\el =  \prod_{\el=1,|w_\el|\geq  1}^d \int e^{-it_\el w_\el} \eta_\el(t_\el)  dt_\el.
$$
Since $|\eta_\el(t_\el) | \rightarrow 0$ when $t_\el\rightarrow \pm  \infty$, an integration by part  yields
$$
\int e^{-it_\el w_\el} \eta_\el(t_\el)  dt_\el= iw_\el^{-1}  \int_{} e^{-it_\el w_\el }\eta_\el '(t_\el) dt_\el.
$$
Let us compute the derivative of $\eta_\el(t_\el)$
 \begin{eqnarray*}
 \eta_\el '(t_\el)=\frac{\overline{\mathcal{F}(\varphi)(t_\el)}'\mathcal{F}(g)(2^{j_\el}t_\el )-2^{j_\el} \mathcal{F}'(g)(2^{j_\el }t_\el ) \overline{\mathcal{F}(\varphi)(t_\el )}}{ (\mathcal{F}(g)(2^{j_\el }t_\el ))^2 }.
 \end{eqnarray*}
 Using Lemma \ref{deriveephi}, (\ref{hypobruit1}) and  (\ref{hypobruit2}) 
 \begin{eqnarray*}
 | \eta_\el '(t_\el ) | &\leq& \left |   \frac{\overline{\mathcal{F}(\varphi)(t_\el)}'}{\mathcal{F}(g)(2^{j_\el}t_\el) } \right | +  2^{j_\el} \left |  \frac{ \mathcal{F}'(g)(2^{j_\el}t_\el) \mathcal{F}(\varphi)(t_\el) } {    (\mathcal{F}(g)(2^{j_\el}t_\el))^2 } \right | \\
 &\leq & C \left((1+|t_\el |)^{-r}(1+2^{j_\el}|t_\el |)^\nu  +2^{j_\el }(1+2^{j_\el }|t_\el |)^{-\nu-1}(1+|t_\el |)^{-r}(1+2^{j_\el }|t_\el |)^{2\nu} \right) \\
 &\leq & C\left (2^{j_\el \nu}  (1+|t_\el |)^{-r}(2^{-j_\el }+|t_\el |)^\nu + 2^{j_\el }(1+2^{j_\el }|t_\el |)^{\nu-1}(1+|t_\el |)^{-r} \right) \\
  &\leq & C \left(2^{j_\el \nu}  (1+|t_\el |)^{-r}(2^{-j_\el }+|t_\el |)^\nu + 2^{j_\el \nu}(2^{-j_\el }+|t_\el |)^{\nu-1}(1+|t_\el |)^{-r} \right)\\
  &\leq & C 2^{j_\el \nu} \left(  (1+|t_\el |)^{-r}(2^{-j_\el}+|t_\el|)^\nu + (2^{-j_\el }+|t_\el |)^{\nu-1}(1+|t_\el |)^{-r}   \right). 
 \end{eqnarray*}
 Therefore,
\begin{eqnarray*}
\left |  \int e^{-it_\el w_\el } \eta_\el (t_\el)  dt_\el    \right | &\leq& |w_\el |^{-1}\int |\eta_\el  '(t_\el )| dt_\el\\
&\leq&C|w_\el |^{-1}2^{j_\el \nu}\int\left(  (1+|t_\el |)^{-r}(2^{-j_\el }+|t_\el |)^\nu + (2^{-j_\el }+|t_\el |)^{\nu-1}(1+|t_\el |)^{-r}   \right)dt_\el\\
&\leq&C|w_\el |^{-1}2^{j_\el \nu}(D_1+D_2+D_3),
\end{eqnarray*}
with $D_1$, $D_2$ and $D_3$ defined below.
\begin{eqnarray*}
D_1&:= &\int_{|t_\el | \leq 2^{-j_\el}} \left(  (1+|t_\el|)^{-r}(2^{-j_\el }+|t_\el |)^\nu + (2^{-j_\el }+|t_\el |)^{\nu-1}(1+|t_\el |)^{-r}   \right)dt_\el \\
&\leq&C\int_{|t_\el | \leq 2^{-j_\el }} \left( (2^{-j_\el }+|t_\el |)^\nu + (2^{-j_\el }+|t_\el |)^{\nu-1}  \right)dt_\el \\
&\leq&C2^{-j_\el }(2^{-j_\el \nu}+2^{-j_\el (\nu-1)})\\
&\leq&C.
\end{eqnarray*}
\begin{eqnarray*}
D_2&:= &\int_{2^{-j_\el }\leq |t_\el | \leq 1}\left(  (1+|t_\el |)^{-r}(2^{-j_\el }+|t_\el |)^\nu + (2^{-j_\el }+|t_\el |)^{\nu-1}(1+|t_\el |)^{-r}   \right)dt_\el \\
&\leq&C\int_{2^{-j_\el }\leq |t_\el | \leq 1} \left( (2^{-j_\el }+|t_\el|)^\nu + (2^{-j_\el}+|t_\el|)^{\nu-1}  \right)dt_\el \\
&\leq&C\int_1^{2^{j_\el }}((2^{-j_\el }+2^{-j_\el }s)^\nu+(2^{-j_\el }+2^{-j_\el }s)^{\nu-1})2^{-j_\el }ds\\
&\leq&C2^{-j_\el (\nu+1)}\int_1^{2^{j_\el }}s^\nu ds+C2^{-j_\el \nu}\int_1^{2^{j_\el }}s^{\nu-1}ds\\
&\leq&C,
\end{eqnarray*}
as soon as $\nu>0$.
\begin{eqnarray*}
D_3&:= &\int_{ |t_\el |\geq 1}\left(  (1+|t_\el |)^{-r}(2^{-j_\el }+|t_\el |)^\nu + (2^{-j_\el }+|t_\el |)^{\nu-1}(1+|t_\el |)^{-r}   \right)dt_i\\
&\leq&C\int_{ |t_\el |\geq 1}\left(|t_\el |^{\nu-r}+|t_\el |^{\nu-1-r}   \right)dt_\el \\
&\leq&C,
\end{eqnarray*}
since $\nu-r\leq -2$.

When $\nu=0$  we still have
$$ \left  | \int e^{-it_\el w_\el} \eta_\el (t_\el )  dt_\el  \right| \leq C|w_\el |^{-1}2^{j_\el \nu} =   C|w_\el |^{-1}.$$
Indeed when $\nu=0$ 
 $$\eta_\el (t_\el)= \overline{\mathcal{F}(\varphi)(t_\el )},$$
 and
 \begin{eqnarray*}
\left  | iw_\el^{-1}  \int_{} e^{-it_\el w_\el }\eta_\el '(t_\el) dt_\el \right|  &= &  \left   | iw_\el ^{-1}  \int  e^{-it_\el w_\el }  \overline{\mathcal{F}(\varphi)(t_\el )}'   dt_\el \right | \\
&\leq & |w_\el |^{-1}  \int_{} \left |  \overline{\mathcal{F}(\varphi)(t_\el )}' \right | dt_\el \\
&\leq & C |w_\el |^{-1 } \int (1+|t|)^{-r} dt < C |w_\el |^{-1},
 \end{eqnarray*}
 using Lemma \ref{deriveephi}  and $r\geq 2$.

%
\end{proof}

\begin{lemma}\label{ordredegrandeursigmajTj}
There exist constants $T_3$ depending on $\|m\|_{\infty}$, $\sigma_{\varepsilon}$, $\|f_X\|_{\infty}$, $\varphi$ ,$c_g$, $\mathcal{C}_g$ and $T_4$ depending on $\varphi$, $c_g$, $\mathcal{C}_g$ such that
$$\sigma_j^2\leq R_{10} 2^{{\Sj(2\nu+1)}},\quad \|T_j\|_{\infty}\leq R_{11} 2^{\Sj(\nu+1)}.$$
\end{lemma}
\begin{proof}
We have 
\begin{eqnarray*}
\sigma_j^2=\var{(U_j(Y_1,W_1))} &\leq& \E\left [ \left | U_j(Y_1,W_1) \right |^2\right ] \\
&=& \E  \left [\left |Y_1\sum_k  \left (\mathcal{D}_j \varphi \right)_{j,k}(W_1)\varphi_{jk}(x)  \right |^2 \right ] \\
&=& \E  \left [\left |(m(X_1)+\e_1)\sum_k  \left (\mathcal{D}_j \varphi \right)_{j,k}(W_1)\varphi_{jk}(x)  \right |^2 \right ]  \\
&\leq& 2(\| m\|_{\infty}^2+\sigma_{\e}^2)  \E  \left [\left |\sum_k  \left (\mathcal{D}_j \varphi \right)_{j,k}(W_1)\varphi_{jk}(x)  \right |^2 \right ]  \\
&\leq & 2(\| m\|_{\infty}^2+\sigma_{\e}^2) \int \left  |\sum_k  \left (\mathcal{D}_j \varphi \right)_{j,k}(w)\varphi_{jk}(x) \right |^2 f_{W}(w)dw \\
&\leq &  2(\| m\|_{\infty}^2+\sigma_{\e}^2) \| f_{X}\|_{\infty }\int 2^{\Sj}\left  |\sum_k  \left (\mathcal{D}_j \varphi \right)(2^jw-k)\varphi_{jk}(x) \right |^2dw. \\
\end{eqnarray*}
 Now making the change of variable $z= 2^jw-k$, we get using Lemma \ref{Wavelets} and  Lemma \ref{operateurK_j} to bound $(\mathcal{D}_j \varphi )(z)$
\begin{eqnarray*}
\sigma_j^2 &\leq&  2(\| m\|_{\infty}^2+\sigma_{\e}^2) \| f_{X}\|_{\infty } \int \left  |\sum_k  \left (\mathcal{D}_j \varphi \right)(z)\varphi_{jk}(x) \right |^2 dz \\
&\leq &  C \int 2^{2\Sj\nu} \prod_{i=\el }^{d }\frac{1}{(1+|z_\el |)^2}  \left(\sum_k | \varphi_{jk}(x)|\right)^2 dz \\
&\leq& R_{10} 2^{ {\Sj(2\nu+1)} },
\end{eqnarray*}
where $R_{10}$ is a constant depending on $\|m\|_{\infty}, s, \|f_X\|_{\infty}, \varphi,c_g,\mathcal{C}_g$. This gives the bound for $\sigma_j^2$. 

For $\|T_j\|_\infty$, using again Lemma \ref{Wavelets} and  Lemma \ref{operateurK_j}, we have 
\begin{eqnarray*}
\|T_j\|_\infty&\leq& \max_k \|(\mathcal{D}_j \varphi)_{j,k}\|_\infty \sum_k | \varphi_{jk}(x)|\leq 2^{\frac \Sj 2}  \|(\mathcal{D}_j \varphi)\|_\infty \sum_k | \varphi_{jk}(x)|\\
&\leq & R_{11}  2^{\Sj(\nu+1)},
\end{eqnarray*}
where $R_{11}$ is a constant depending on $\varphi$, $c_g$, $\mathcal{C}_g$.
 
\end{proof}

\noindent \textbf{Acknowledgements}:  The research of Thanh Mai Pham Ngoc and Vincent Rivoirard is partly supported by the french Agence Nationale de la Recherche (ANR 2011 BS01 010 01 projet Calibration). Micha\"el Chichignoud now works at Winton Capital Management, supported in part as member of the German-Swiss Research Group FOR916 (Statistical Regularization and Qualitative Constraints) with grant number 20PA20E-134495/1.

\bibliographystyle{apalike}
\bibliography{biblio}

\begin{thebibliography}{}

\bibitem[Bertin et~al., 2013]{BLR}
Bertin, K., Lacour, C., and Rivoirard, V. (2013).
\newblock Adaptive pointwise estimation of conditional density function.
\newblock {\em Submitted}.

\bibitem[Carroll et~al., 2009]{CarrollDelaigleHall}
Carroll, R.~J., Delaigle, A., and Hall, P. (2009).
\newblock Nonparametric prediction in measurement error models.
\newblock {\em J. Amer. Statist. Assoc.}, 104(487):993--1003.

\bibitem[Chesneau, 2010]{Chesneau}
Chesneau, C. (2010).
\newblock On adaptive wavelet estimation of the regression function and its
  derivatives in an errors-in-variables model.
\newblock {\em Curr. Dev. Theory Appl. Wavelets}, 4(2):185--208.

\bibitem[Comte and Lacour, 2013]{ComteLacour}
Comte, F. and Lacour, C. (2013).
\newblock Anisotropic adaptive kernel deconvolution.
\newblock {\em Ann. Inst. Henri Poincar\'e Probab. Stat.}, 49(2):569--609.

\bibitem[Comte and Taupin, 2007]{ComteTaupin}
Comte, F. and Taupin, M.-L. (2007).
\newblock Adaptive estimation in a nonparametric regression model with
  errors-in-variables.
\newblock {\em Statist. Sinica}, 17(3):1065--1090.

\bibitem[Delaigle et~al., 2015]{DelaigleHallJamshidi}
Delaigle, A., Hall, P., and Jamshidi, F. (2015).
\newblock Confidence bands in non-parametric errors-in-variables regression.
\newblock {\em J. R. Stat. Soc. Ser. B. Stat. Methodol.}, 77(1):149--169.

\bibitem[Doumic et~al., 2012]{DHRR}
Doumic, M., Hoffmann, M., Reynaud-Bouret, P., and Rivoirard, V. (2012).
\newblock Nonparametric estimation of the division rate of a size-structured
  population.
\newblock {\em SIAM J. Numer. Anal.}, 50(2):925--950.

\bibitem[Du et~al., 2011]{DuZouWang}
Du, L., Zou, C., and Wang, Z. (2011).
\newblock Nonparametric regression function estimation for errors-in-variables
  models with validation data.
\newblock {\em Statist. Sinica}, 21(3):1093--1113.

\bibitem[Fan and Koo, 2002]{FanKoo}
Fan, J. and Koo, J.-Y. (2002).
\newblock Wavelet deconvolution.
\newblock {\em IEEE Trans. Inform. Theory}, 48(3):734--747.

\bibitem[Fan and Masry, 1992]{Fan92}
Fan, J. and Masry, E. (1992).
\newblock Multivariate regression estimation with errors-in-variables:
  asymptotic normality for mixing processes.
\newblock {\em J. Multivariate Anal.}, 43(2):237--271.

\bibitem[Fan and Truong, 1993]{Fan93}
Fan, J. and Truong, Y.~K. (1993).
\newblock Nonparametric regression with errors in variables.
\newblock {\em Ann. Statist.}, 21(4):1900--1925.

\bibitem[Gach et~al., 2013]{Nickl}
Gach, F., Nickl, R., and Spokoiny, V. (2013).
\newblock Spatially adaptive density estimation by localised {H}aar
  projections.
\newblock {\em Ann. Inst. Henri Poincar\'e Probab. Stat.}, 49(3):900--914.

\bibitem[Goldenshluger and Lepski, 2011]{GL}
Goldenshluger, A. and Lepski, O. (2011).
\newblock Bandwidth selection in kernel density estimation: oracle inequalities
  and adaptive minimax optimality.
\newblock {\em Ann. Statist.}, 39(3):1608--1632.

\bibitem[H{\"a}rdle et~al., 1998]{Hardle}
H{\"a}rdle, W., Kerkyacharian, G., Picard, D., and Tsybakov, A. (1998).
\newblock {\em Wavelets, approximation, and statistical applications}, volume
  129 of {\em Lecture Notes in Statistics}.
\newblock Springer-Verlag, New York.

\bibitem[Ioannides and Alevizos, 1997]{Ioannides}
Ioannides, D.~A. and Alevizos, P.~D. (1997).
\newblock Nonparametric regression with errors in variables and applications.
\newblock {\em Statist. Probab. Lett.}, 32(1):35--43.

\bibitem[Koo and Lee, 1998]{Koo98}
Koo, J.-Y. and Lee, K.-W. (1998).
\newblock {$B$}-spline estimation of regression functions with errors in
  variable.
\newblock {\em Statist. Probab. Lett.}, 40(1):57--66.

\bibitem[Meister, 2009]{Meister}
Meister, A. (2009).
\newblock {\em Deconvolution problems in nonparametric statistics}, volume 193
  of {\em Lecture Notes in Statistics}.
\newblock Springer-Verlag, Berlin.

\bibitem[Patel and Read, 1982]{Patel}
Patel, J.~K. and Read, C.~B. (1982).
\newblock {\em Handbook of the normal distribution}, volume~40 of {\em
  Statistics: Textbooks and Monographs}.
\newblock Marcel Dekker, Inc., New York.

\bibitem[Whittemore and Keller, 1988]{Whittemore}
Whittemore, A.~S. and Keller, J.~B. (1988).
\newblock Approximations for regression with covariate measurement error.
\newblock {\em J. Amer. Statist. Assoc.}, 83(404):1057--1066.

\end{thebibliography}

%

\end{document}